\documentclass[article,reqno]{amsart}
\usepackage{enumerate}
\usepackage{amsfonts,amssymb,amsmath,amsthm,mathrsfs}
\usepackage{dsfont}
\usepackage{epsfig}
\usepackage{graphicx}
\usepackage{subcaption}
\usepackage[normalem]{ulem}
\usepackage{color}
\usepackage{comment}
\usepackage{soul}

\usepackage{tikz-cd}

\usepackage{hyperref}
\hypersetup{
    linktoc=all,  }   

\usepackage{tikz}
\usetikzlibrary{arrows.meta, decorations.pathreplacing}

\input xy 
\xyoption{all}
\numberwithin{equation}{section}

\definecolor{OrangeRed}{cmyk}{0,0.6,1,0}            
\definecolor{DarkBlue}{cmyk}{1,1,0,0.20}
\definecolor{DarkGreen}{cmyk}{1,0,0.6,0.2}
\definecolor{myblue}{rgb}{0.66,0.78,1.00}
\definecolor{Violet}{cmyk}{0.79,0.88,0,0}
\definecolor{Lavender}{cmyk}{0,0.48,0,0}
\newcommand{\violet }{\color{Violet}}

\newtheorem{thm}{Theorem}[section]
\newtheorem{theorem}[thm]{Theorem}
\newtheorem{main theorem}[thm]{Main Theorem}
\newtheorem{corollary}[thm]{Corollary}

\newtheorem{lemma}[thm]{Lemma}
\newtheorem{lem}[thm]{Lemma}
\newtheorem{prop}[thm]{Proposition}

\theoremstyle{definition}

\newtheorem{defn}[thm]{Definition}
\newtheorem{definition}[thm]{Definition}
\newtheorem{remark}[thm]{Remark}
\newtheorem{rem}[thm]{Remark}

\def\C{\mathbb C}
\def\P{\mathbb P}

\def\bcases{\begin{cases}}

\def\ecases{\end{cases}}

\newcommand{\D}{\mathbb D}

\newcommand{\im}{\text{\rm Im}\,}
\renewcommand{\Im}{\operatorname{Im}\,}

\newcommand{\N}{\mathbb N}

\newcommand{\R}{\mathbb R}
\newcommand{\re}{\text{\rm Re}\,}

\newcommand{\supp}{\operatorname{supp}}
\newcommand{\Id}{\operatorname{Id}}

\newcommand{\Z}{\mathbb Z}

\newcommand{\DD}{\mathbb D}

\newcommand{\bea}{\begin{eqnarray*}}
\newcommand{\eea}{\end{eqnarray*}}

\newcommand{\be}{\begin{equation}}
\newcommand{\ee}{\end{equation}}

\newcommand{\ra}{\rightarrow}

\renewcommand{\Re}{\mathrm{Re}\,}
\newcommand{\diam}{\mathrm{diam}}

\newcommand{\dist}{\operatorname{dist}}
\renewcommand{\epsilon}{\varepsilon}
\renewcommand{\phi}{\varphi}

\newcommand{\ov}{\overline}

\newcommand{\bC}{{\bf C}}
\newcommand{\QQ}{\mathcal{Q}}

\newcommand{\TT}{\mathcal{T}}
\newcommand{\XX}{\mathcal{X}}

\renewcommand{\H}{\mathbb{H}}

\renewcommand{\emptyset}{\varnothing}

\newcommand{\bG}{{\bf G}}

\newcommand{\PmuQN}{\P^{\mu}_{|_{Q_N}}}
\newcommand{\PmuQB}{\P^{\mu}_{|_{Q_B}}}
\newcommand{\PmuQBm}{\P^{\mu}_{|_{Q_{B-1}}}}

\makeatletter
\@namedef{subjclassname@2020}{\textup{2020} Mathematics Subject Classification}
\makeatother

\begin{document}

\title{Equidistribution Measures of infinite entropy for Transcendental Functions}

\author[L. Arosio]{Leandro Arosio$^{\dag}$}
\author[A.M. Benini]{Anna Miriam Benini}
\author[J.E.  Forn{\ae}ss ]{John Erik Forn{\ae}ss}
\author[H. Peters]{Han Peters}
\today

\thanks{\dag Partially supported by the MIUR Excellence Department Project 2023-2027 MatMod@Tov awarded to the	Department of Mathematics, University of Rome Tor Vergata, by PRIN Real and Complex Manifolds: Topology, Geometry and holomorphic dynamics n.2017JZ2SW5 and by GNSAGA of INdAM}
\subjclass[2020]{Primary: 37F80, 32H50, 37F10. Secondary:  37A35}

\begin{abstract}
In the 1980s Lyubich and Freire--Lopes--Ma\~n\'e proved that for any rational function of degree $d \ge 2$, both  preimages and periodic points equidistribute to the unique measure of maximal entropy $\log(d)$. Their results provide a fundamental understanding of the dynamics of iterated rational functions, and have since been generalized to many different contexts, including  classes of higher-dimensional polynomial and rational maps.

In the current paper we depart from the algebraic category and aim to prove analogous statements for transcendental functions in the complex plane, which have infinite topological entropy.  We introduce two different methods for constructing invariant measures in the transcendental setting, namely via embedded symbolic dynamical systems and via transfer operators  associated to suitably chosen weights. In the latter case we isolate three properties of the weights ---normality, tightness, and irreducibility--- which together imply convergence to an invariant measure. We provide examples  for each method, given by three classes of transcendental entire functions: disjoint-type maps, strongly polynomial-like maps, and a class of maps inspired by Baker's construction of  multiply connected wandering domains and by Bishop's construction of  Julia sets of Hausdorff dimension $1$, which we call Baker--Bishop maps.

 For each of these classes we prove that with respect to carefully chosen weights, preimages equidistribute to an invariant  mixing probability measure of infinite entropy. For Baker--Bishop maps and disjoint-type maps we also prove equidistribution of  periodic points. In contrast to the rational setting, the measures we construct are not unique: by varying the weights one obtains infinitely many distinct measures.
\end{abstract}
\maketitle
\tableofcontents

\section{Introduction}

A fundamental result in complex dynamical systems is the equidistribution of preimages and periodic points towards the measure of maximal entropy:

\begin{theorem}
    Let $f: \widehat{\mathbb C} \rightarrow \widehat{\mathbb C}$ be a rational map of degree $d \ge 2$. Then there exists a unique invariant ergodic probability measure $\mu$ of maximal entropy $\log d$. 
    
    Moreover,
    \begin{enumerate}
    \item  
   there exists an \emph{exceptional set} $\mathcal{E}_f \subset \widehat{\mathbb C}$ of cardinality at most $2$ such that, for any $w \in \widehat{\mathbb C} \setminus \mathcal{E}_f$, the sequence of probability measures
    $$
    \frac{1}{d^n}\sum_{z \in f^{-n}(w)} \delta_z
    $$
   converges weakly to $\mu$;
   \item  the sequence of finite measures  $$
    \frac{1}{d^{n}}\sum_{z : f^n(z)=z} \delta_z 
    $$
   converges  weakly to $\mu$.
    \end{enumerate}
\end{theorem}

This result is due to Lyubich~\cite{Lyubich83}, and independently to Freire--Lopes--Ma\~n\'e~\cite{FLM}. The measure of maximal entropy was previously constructed for polynomials by Brolin~\cite{Brolin} via different methods.

In the current paper we investigate the existence of an invariant ergodic measure of maximal entropy for transcendental functions: entire self-maps of $\mathbb C$ that are not polynomials. We consider three different families of maps:

\begin{enumerate}
\item[(1)] Transcendental functions for which the norms of the coefficients decrease sufficiently rapidly, referred to as \emph{elementary strongly polynomial-like maps}.
\item[(2)] A subclass of maps of class (1) for which all critical orbits escape rapidly, referred to as \emph{Baker--Bishop maps}.
\item[(3)] Transcendental functions of \emph{disjoint type}. To reduce the complexity of the arguments we will moreover assume that the functions have finitely many tracts and are finite compositions of functions of finite order.
\end{enumerate}

While Baker--Bishop maps are  a subset of elementary strongly polynomial-like maps,  the separate treatment of this class is necessary in order to be able to prove the results for class (1). Each of these classes will be discussed and defined more precisely later in the introduction.

\subsection{Weighted equidistribution}
Let $f\colon \C\to \C$ be a transcendental function.  By the great Picard theorem every point  $w\in \C$ except possibly one 
has infinitely many preimages. Thus one cannot define a probability measure that places point masses of equal weights at all preimages of a point $w$. Instead, 
 let $w\in \C$ which is not in the forward orbit of the omitted value (if any).
Then one can define for each $w \in \mathbb C$ a probability vector $Q(w,\cdot)$, where the second index runs over all preimages of $w$.
If $n\geq 1$, then for all $z\in f^{-n}(w)$  we can define
$$
Q^n(w,z) = \prod_{j=0}^{n-1} Q(f^{j+1}(z), f^j(z)),
$$
which gives new probability vectors $Q^n(w,\cdot)$. For a given initial point $w$ we therefore obtain a sequence of probability measures
$$
Q^n_w = \sum_{z\in f^{-n}(w)} Q^n(w,z)  \delta_z.
$$

We note that when $f$ is a rational function of degree $d$, setting $Q(w,z) = \frac{1}{d}$ gives the standard equidistribution, hence weighted equidistribution is a generalization thereof; a necessary generalization when intending to include transcendental functions.

In this paper we deal with several classes of maps that can be understood well enough to define weights $Q(w,z)$ in such a way that the sequence of measures $Q^n_w$  converges weakly to a probability measure $\mu$, which  captures the full  complexity of the dynamical system.

%
%

\subsection{Main result}

Our main result is the following:

\begin{theorem}\label{Theorem: Main1}
For any of the maps in class (1), (2), or (3), there exist infinitely many   invariant  mixing (hence ergodic)  measures of infinite entropy $\widetilde \mu$ supported on the Julia set that can be obtained via weighted equidistribution of preimages, that is, the sequence of probability measures  $$\sum_{z\in f^{-n}(w)}Q^n(w,z)\delta_z$$ converges  weakly to $\widetilde \mu$.
\end{theorem}
For a discussion on which $w\in \C$ are suitable in the three cases we refer  to the corresponding sections.
Just as for rational functions, we prove that such measures  can also arise via equidistribution of periodic points instead of preimages:

\begin{theorem}\label{Theorem: Main2}
For any of the maps in class (2) or (3) let  $\widetilde\mu$ denote an invariant measure constructed in Theorem \ref{Theorem: Main1}.
 Then $\widetilde \mu$    can also be obtained via weighted equidistribution of periodic cycles,
     that is, the sequence of measures
     $$\sum_{z:f^n(z)=z}Q^n(z,z)\delta_z$$ converges  weakly to $\widetilde \mu$.
\end{theorem}

The dynamics and the arising ergodic measures  will be described most explicitly for the Baker--Bishop maps in class (2). For these maps the measures  can be constructed via a relatively simple symbolic dynamical system, analogous to the construction of symbolic dynamical systems for many hyperbolic polynomials and rational functions. The main difference is that for rational functions the symbolic dynamical systems are given by shift maps on finitely many symbols, while  in our setting the measures arise via shift maps acting on a space with countably many symbols.
The construction and description of these symbolic dynamical systems will play a central role in this paper. The symbolic dynamical systems provide natural weights for the equidistribution of preimages and of periodic points. It turns out that the same class of symbolic dynamical systems can be used to obtain invariant measures for any of the disjoint-type maps in class (3).

For maps in the larger class (1) we do not obtain a symbolic dynamical system, but instead construct the invariant measure directly via weighted equidistribution of preimages, using a transfer operator similarly as in \cite{Lyubich83}, defined in terms of the weights  $Q(w,\cdot)$.  An important difference with the setting of rational functions, due to the non-compactness of the complex plane, is that the convergence under applications of the pullback operator can only be uniform on compact subsets, hence does not correspond to convergence in a Banach space. The convergence in \cite{Lyubich83} relies on the almost periodicity of the transfer operator and the simplicity of its unitary spectrum. In the locally uniform setting, we require different conditions ---normality, tightness and irriducibility--- on the weights  $Q(w,\cdot)$ in order to obtain convergence, see Theorem \ref{thm: non-uniform Lyubich} for the precise statement.


 We choose weights for the equidistribution of preimages which are identical to those used for similar maps in class (2). Thus the symbolic dynamical system, while not embedded in the dynamics of the actual transcendental map,  still plays an important role for maps in class (1),  as inspiration for choosing weights for which 
$Q^n_w\to\tilde \mu$, the entropy of $\tilde\mu$ is infinite and the support of $\tilde \mu$ is the Julia set.
 In particular the choice of the weights implies that mass does not escape to infinity and that for all bounded continuous function $\varphi $ on $\C$ the sequence $(A^n(\varphi))$ is equicontinuous, where $A$ is the transfer operator associated to the weights $Q(w,\cdot)$.
It is worth noticing that, lacking the symbolic representation, the equidistribution of periodic points appears to be more challenging in case (1), and is left as an open problem.

\subsection{Entropy of rational and transcendental functions}

Let us briefly discuss known results regarding the entropy of rational and transcendental maps. We start with two classical results from the 1970's.

\begin{theorem}[Przytycki--Misiurewicz, Gromov, 1977]
    Let $f:\mathbb C \rightarrow \mathbb C$ be a rational function of degree $d \ge 2$. Then the topological entropy of $f$ equals $\log(d)$.
\end{theorem}

The estimate from below $h_\mathrm{top}(f) \ge \log(d)$ was proved by Przytycki and Misiurewicz~\cite{MP1977} and the estimate from above was proved shortly after by Gromov~\cite{Gromov}. Both results are valid in a more general context: the estimate from below holds well outside of the holomorphic category, and the estimate from above holds for holomorphic endomorphisms of projective space in any dimension, where $d$ then refers to the topological degree of the map. The fact that the topological entropy of a complex H\'enon map of algebraic degree $d$ is also given by $\log(d)$ was proved by Smillie~\cite{Smillie90}.

As discussed before, the existence of the unique measure of maximal entropy $\log(d)$ was proved by Lyubich~\cite{Lyubich83} and by Freire--Lopez--Ma\~n\'e~\cite{FLM}.
The construction of the measure of maximal entropy has since been generalized to several different contexts in higher dimensions, including to holomorphic endomorphisms of projective space by Forn{\ae}ss and Sibony~\cite{FS94}, and to complex H\'enon maps by Bedford--Lyubich--Smillie~\cite{BLS93}.

The topological entropy of transcendental functions was settled in the PhD thesis of Wendt~\cite{Wendt,WendtMan}, and independently in the recent papers~\cite{BFP1, BFP2}.

\begin{theorem}[Wendt, Benini--Forn{\ae}ss--Peters]\label{inftop}
    Any transcendental function has infinite topological entropy.
\end{theorem}

The methods used by Wendt were adapted by the authors in~\cite{henon3} to show that transcendental H\'enon maps also have infinite topological entropy.

\medskip

It remains an open question whether there always exists an ergodic measure of infinite entropy, a question that motivates the current research. We recall the following related result due to Christensen and Fisher~\cite{FiCh}:

\begin{theorem}
    For any transcendental function $f: \mathbb C \rightarrow \mathbb C$ there exists an ergodic measure $\mu$ of positive entropy.
\end{theorem}

The proof of this result is non-constructive and cannot be used to obtain any positive estimate from below on the entropy of $\mu$. The authors show that for any transcendental function $f$ there exists a topological disk $D \subset \mathbb C$ and an integer $n \in \mathbb N$ such that $f^{-n}(D)$ contains two univalent inverse branches $D_1$ and $D_2$, both properly contained in $D$. As a result one obtains a symbolic dynamical system on the Cantor set
$$
\bigcap_{k \in \mathbb N} f^{-k\cdot n}\left(D_1 \cup D_2\right),
$$
and the restriction of $f^n$ to this Cantor set induces a shift map on two symbols. The unique measure of maximal entropy for this restricted dynamical system therefore has entropy $\log(2)$, and one obtains an ergodic measure for $f$ of entropy $\log(2)/n$.

\subsection{Embedding of a symbolic dynamical system arising from a Markov process}
 We will construct a symbolic dynamical system given by a countable state space $S$, the shift map $\sigma: S^{\mathbb N} \rightarrow S^{\mathbb N}$, and an ergodic $\sigma$-invariant probability measure $\mathbb{P}^\mu$ on $S^{\mathbb N}$ with infinite entropy. The measure is induced by an ergodic Markov process on $S$.
Let $\Theta$ denote the support of the measure $\mathbb{P}^\mu$.
Now suppose that $f: \mathbb C \rightarrow \mathbb C$ is a transcendental function,
 that  $\Sigma \subset \Theta$ is a  Borel set of full measure which is completely invariant for the map $\sigma|_\Theta$, and that $\varphi : \Sigma \rightarrow \mathbb C$ is a homeomorphism 
 with its image
that induces a commutative diagram:

\[ \begin{tikzcd}
\Sigma \arrow{r}{\sigma} \arrow[swap]{d}{\varphi} & \Sigma \arrow{d}{\varphi} \\%
\mathbb C \arrow{r}{f}& \mathbb C
\end{tikzcd}
\]

When this is the case we say that the symbolic dynamical system  $(S^\N, \sigma)$ is \emph{embedded} in  the complex dynamical system $(\mathbb C,f)$. Whenever the symbolic dynamical system can be embedded, the measure $\mathbb{P}^\mu$ can be pushed forward to an ergodic $f$-invariant probability measure $\varphi_*\mathbb{P}^\mu$ on $\mathbb C$  whose entropy equals the entropy of $\mathbb{P}^\mu$, which by construction is infinite.  Notice that since $\Sigma\subset \Theta$ and $\P^\mu(\Sigma)=1$ we have that the support of the measure $\P^\mu|_\Sigma$  is $\Sigma$ itself, and thus
\begin{equation}\label{supportpushforward}
{\rm Supp}\,\varphi_*\mathbb{P}^\mu=\overline{\varphi(\Sigma)}.
\end{equation}

 We will use this approach for Baker--Bishop maps.
  For disjoint-type maps the same approach works with some technical adjustments which are needed since the map $\varphi$ is not  continuous in that case.
 In both cases the weights of the equidistribution 
are given by the transition probabilities of the time-reversed Markov process.
 For  strongly polynomial-like maps, as explained previously, we will use a different approach. 

\subsection{Different classes of transcendental maps}
\subsubsection{Baker--Bishop maps}

 In \cite{Bak63},\cite{Bak76} Baker constructed  an entire function 
 of the form
$$
f(z) = \prod_{j=1}^\infty \left( 1- \frac{z}{b_j}\right)
$$
where the norms of the roots $b_j$ grow sufficiently fast such that the product converges on all of $\mathbb C$, and such that it creates a sequence of annuli centered at increasing radii whose image is compactly contained in the following annulus. This creates multiply connected Fatou components  on which the family of iterates converges uniformly to infinity, and by showing that such components never intersect, he provided 
the first examples of wandering domains in transcendental iteration.

Later in~\cite{Bis18}, Bishop  used a similar map to  prove the existence of transcendental functions whose Julia set has Hausdorff dimension 1. In these examples the Julia set is given by the closure of  a countable union of smooth circles. While the smoothness of the circles and the Hausdorff dimension of the Julia set are not relevant for our purpose, the maps we will cover have the same dynamical behavior and the Julia set has the same topological structure. 

For our purposes we find it more convenient to work with power series expansions whose coefficients decrease sufficiently rapidly. After conjugating with a suitable translation we can eliminate the linear coefficient, and we are left with maps of the form
$$
f(z) = a_0 + a_2 z^2 + \cdots.
$$
For further convenience we will assume that $a_j>0$ for all $j \in \{0, 2,3, \ldots\}$. We note that in this form the roots of $f$ are approximately given by $\pm \frac{\sqrt{a_0}}{a_2}$ and $- \frac{a_j}{a_{j+1}}$ for $j \ge 2$.

The maps are normalized such that $x_1=0$ is a critical point. The other critical points $x_d$, for $d \ge 2$ are approximately given by
$$
x_d \sim - \frac{d}{d+1} \cdot \frac{a_{d}}{a_{d+1}}.
$$
By choosing the norms of the coefficients carefully we can make sure that the orbits of every critical point $x_j$ escapes rapidly to infinity.

\begin{defn}
    Let $f(z) = a_0 + a_2 z^2 + \cdots$ be a transcendental entire function. We say that $f$ is a \emph{Baker--Bishop map} when there exists a sequence of disjoint smoothly bounded topological annuli $B_i$ whose union contains all critical values, such that $f(B_i)\Subset B_{i+1}$ and $\inf_{z\in B_i }|z|\ra\infty$ as $i\ra\infty$.
\end{defn}

We will work with a specific subclass of Baker--Bishop maps, for which the sets $B_i$ can be chosen as round annuli that each contain at most a single critical value. In Section \ref{sec: Bishop maps} we prove  Theorems \ref{Theorem: Main1} and \ref{Theorem: Main2} for such maps.
 The goal of our treatment of these  maps is not to present the most general statement, but to provide a test case which can be used  as inspiration for the disjoint-type maps and the strongly polynomial-like  maps that are treated in later sections.

\subsubsection{Disjoint-type maps}
Disjoint-type maps are (from a functional perspective)  entirely different from the Baker--Bishop maps. However we will show that   the symbolic dynamical system arising from an appropriate modification of the ergodic Markov process embeds also in this case.

\begin{definition}
    A transcendental function $f\colon \C\to \C$ is said to be of \emph{bounded-type} if the set of singular values is bounded.    A bounded-type function $f$ is of \emph{disjoint type} if there exist a holomorphic disc $D\Subset \C$ with real-analytic boundary,  containing    the set $S(f)$ of singular values, whose  image $f(D)$ is relatively compact in $D$.
\end{definition}

Bounded-type maps were introduced by Eremenko  and Lyubich in \cite{EL92} and are among most studied classes in transcendental dynamics (see e.g. \cite{RRRS}, and the survey \cite{SixsmithB}). Indeed, their dynamics bears some similarities with polynomial dynamics, and from the functional theoretical point of view, they can be subdivided into (uncountably many) natural parameter spaces parametrized by complex manifolds.
Disjoint-type maps were introduced in \cite{BK07}. They are simple maps from the dynamical point of view: there exists a unique attracting fixed point whose basin contains all singular values, and the map is expanding on the complement of $D$ in a precise sense (see Section~\ref{sec:class B}). On the other hand, for any map $f$ of bounded type, the map $\lambda f$ is of disjoint type for all $\lambda$ sufficiently small, while still belonging to the same parameter space as $f$; so in some sense, from the function theoretical perspective disjoint-type maps are universal in the class of bounded-type maps.

\begin{defn}
Let $f$ be a transcendental function, and for all radii $r>0$ define $M_f(r):=\sup_{|z|=r}|f|$. The function   $f$ is of \emph{finite order} if
$$
\operatorname{limsup}_{r\ra\infty}\frac{\log\log M_f(r)}{\log r}<\infty.
$$
\end{defn}
Essentially, a function is of finite order if its modulus $M_f(r)$ grows at most like the exponential of a polynomial in $r$. The composition of maps of finite order may not be of finite order, consider for example $e^{e^z}$.

In Section \ref{sec:class B} we prove Theorems \ref{Theorem: Main1} and \ref{Theorem: Main2} for a transcendental function $f$ which is  a finite composition of transcendental functions  of finite order,  is of disjoint type, and has finitely many tracts (see Section~\ref{sec:class B} for a definition of tracts).
The assumption of finiteness of the numbers of tracts simplifies the presentation. We also expect that neither the finite order nor the disjoint-type assumption are necessary to obtain ergodic measures of infinite entropy for functions of bounded-type. 
Injectivity may be lost when one considers maps that are merely of bounded-type. While one can still push forward the ergodic measure under a continuous map that is not necessarily injective, the entropy of the push forward measure may be reduced. In order to prove that the entropy is still infinite one needs to have some description of how many points can be mapped down to a single point; a worthwhile future endeavor.

\subsubsection{Strongly polynomial-like maps}\label{sec:sigma proper intro} Let us recall the notion of a polynomial-like map:

\begin{definition}
    Let $U, V \subset \mathbb C$ be bounded topological disks, with $U$ relatively compact in $V$. A \emph{polynomial-like map} of degree $d$ is a proper holomorphic map $f: U \rightarrow V$ of degree $d$.
\end{definition}

Polynomial-like maps have been introduced in \cite{DH85}. The notion of polynomial-like maps has played a crucial role in the field, for example in describing the structure of the Mandelbrot set and for defining the renormalization of polynomials. 

\begin{definition}
    A transcendental function $f: \mathbb C \rightarrow \mathbb C$ is called \emph{strongly polynomial-like} if there exists an increasing union of  bounded topological disks $(U_n)_{n \in \mathbb N}$ with smooth boundaries  such that $\bigcup_n U_n=\C$, and such that $f_n:=f|_{U_n}:U_n\ra f(U_n)$  is polynomial-like for any $n$. 
\end{definition}

Since $f(U_n)\Supset U_n $ for any polynomial like map, it follows that  $f$ is surjective. Since for a transcendental  function a generic point has infinitely many preimages, the degree of the maps $f_n$ tends to infinity. 

\begin{remark}
Osborne \cite[Theorem 1.7]{Osborne}  lists several   sufficient conditions for a map to be  strongly polynomial-like.
Also notice that strongly polynomial-like maps are $\sigma$-proper in the sense of McMullen \cite[p. 131]{McmullenRenormalization}. 
 \end{remark}

Entire transcendental functions of the form $f(z) = a_0 + a_1 z+a_2z^2+ \ldots$ are strongly polynomial-like provided $|a_n|\ra0$ sufficiently fast. 
 Indeed, for each $d \ge 2$, the polynomial $p_d(z) = a_0 + a_1 z + \cdots + a_d z^d$ acts as a polynomial-like map of degree $d$ on sets $U, V$ where $V= \mathbb{D}_R$ for $R> 0$ sufficiently large, and $U = p_d^{-1}(V)$. Having fixed $R = R_d$, we can make sure that $f: U \rightarrow V$ is a polynomial-like map of degree $d$ from $U = f^{-1}(V) \cap V$ to $V = \D_R$ by letting the norms of the coefficients $a_{d+1}, a_{d+2}, \ldots$ be sufficiently small. Thus, by recursively making sure that each consecutive coefficient is sufficiently small, we can guarantee that for each $d \ge 2$ there exists a restriction of $f$ to a sufficiently large topological disk that is polynomial-like of degree $d$.

The class of maps (1) considered in this paper is a subset of the class of strongly polynomial-like maps, referred to as \emph {elementary strongly polynomial-like maps},
for which $|a_n|\ra0$ sufficiently fast as to satisfy several additional requirements specified in Section \ref{sec: strongly polynomial-like}, where we prove  Theorems \ref{Theorem: Main1} and \ref{Theorem: Main2} for such maps.

\subsection{Conformal measures.}

A related area of research is the construction of \emph{conformal measures}. We refer the reader to the book~\cite{MaUrBook} on thermodynamic formalism by Mayer and Urba\'nski, and its citations to the extensive literature on the subject. Conformal measures can be seen as an analogue of the probability measures we obtain via weighted equidistribution, with the difference that the weights are largely determined by the derivatives at the inverse images, see for example Definition 14 in \cite{MaUrSurvey}:

\begin{defn}
Let $t>0$ be a parameter. Let $\tau\geq 0$, and denote by $|\cdot|_\tau$ the metric $\frac{|dz|}{1+|z|^\tau}$. Let $f
: \mathbb C \rightarrow \mathbb C$ be a meromorphic function. A Borel probability measure $m_t$
on $J(f)$ is called $\lambda|f^\prime|^t_\tau$-\emph{conformal} if
\begin{equation}\label{eq:conformal measure}
m_t(f(E)) = \int_E \lambda|f^\prime(z)|^t_\tau \, dm_t
\end{equation}
for every Borel $E \subset J(f)$ such that the restriction $f_{\mid E}$ is injective. The scalar $\lambda$ is called the
conformal factor and, if $\lambda = 1$, then $m_t$ is called a $t$–conformal measure.
\end{defn}

In \cite{MaUrBook} more general integrands in equation~\ref{eq:conformal measure} are considered: functions of the form $e^{-\phi}$ where
$$
\phi = -t \log |f^\prime|_\tau + h,
$$
is a potential on $J(f)$ with $h$ a bounded H\"older continuous function, thus providing even more freedom in constructing measures $\mu_\phi$.

Conformal measures are clearly not intended to be invariant under $f$, however there often exists a unique invariant measure $\mu_t$ that is absolutely continuous with respect to $m_t$, a so-called \emph{Gibbs state}, see for example Theorem 25 in \cite{MaUrSurvey} which states the existence of conformal measures and Gibbs states for dynamically semi-regular meromorphic functions.

The notion of conformal measures originated in the work of Patterson~\cite{Pat76, Pat87}, and was adapted to the setting of rational functions by Sullivan~\cite{Su79, Su82, Su84}. The initial goal for the construction of conformal measures for rational functions was to describe different types of dimensions of the Julia set. While conformal measures were not constructed in order to obtain a description of the entropy of complex dynamical systems, and this aspect has not received as much attention in the literature, we point the reader to Theorem 6.25 in~\cite{MaUrBook}, an analogue of the Variational Principle.

\begin{remark}
    The potentials $\phi$ can of course only be defined when $f^\prime$ does not vanish on the Julia set. As a consequence, conformal measures were successfully constructed and exploited only under the assumption that the meromorphic function satisfies at least some hyperbolicity assumption: the postcritical set does not intersect the Julia set and is either bounded (hyperbolic) or has some positive distance to the Julia set (E-hyperbolic).  The disjoint-type maps that we consider are hyperbolic, while the Baker--Bishop maps are E-hyperbolic. In fact, conformal measures of finite order disjoint-type maps are explicitly considered in \cite{MaUrSurvey}.

We note however that the elementary strongly polynomial-like maps that we consider here may be neither hyperbolic nor E-hyperbolic, as critical points may well lie on the Julia set and can play a crucial role in the behavior of the maps.
\end{remark}

\subsection{Organisation of the paper} In section 2 we cover necessary background on entropy for measure-preserving dynamical systems, including those dynamical systems induced by a Markov process. In section 3 we construct the Markov processes of infinite entropy, in section 4 we embed these Markov process for Baker--Bishop maps, and in section 5 for disjoint-type maps, obtaining ergodic measures of infinite entropy. In section 6 we use the Markov processes as inspiration for the construction of an ergodic measure of infinite entropy for elementary strongly polynomial-like maps.

\subsection{Notation}

Throughout the paper we will use $D(a,r)$ for the disk of radius $r$ centered at $a$. When $a =0$ we may write $\DD_r$ instead of $D(0,r)$. We will write $\mathbb N = \{0,1,\ldots\}$.
\subsection{Acknowledgements} We are thankful to Walter Bergweiler and Lasse Rempe for their precious suggestions.
 \section{Preliminaries}

\subsection{Measure Theoretic Entropy}

Let $X$ be a compact metric space, $f: X \rightarrow X$ a continuous map.
The topological entropy of $f$ can be defined in a number of ways (see e.g.\cite{AKM65,Bow71, Din70}); all equivalent for a continuous self-map of a compact metric space. The definition given by Bowen and Dinaburg is the following.
Given $\epsilon>0$ and $n \in \mathbb N$ we say that a finite set $E \subset X$ is $(\epsilon, n)$-separated if for any $(x,y)\in E\times E$ there exists $0 \le j \le n-1$ such that $d(f^j(x), f^j(y)) > \epsilon$. We denote by $N(\epsilon, n)$ the maximum cardinality of an $(\epsilon, n)$-separated set, which is necessarily finite by compactness of $X$ and continuity of $f$. The \emph{topological entropy} of $f$ is defined by
$$
h_{\mathrm{top}}(f) = \lim_{\epsilon \rightarrow 0} \limsup_{n \rightarrow \infty} \frac{\log(N(\epsilon, n))}{n}.
$$
We note that this definition of topological entropy is independent of the metric inducing the topology, and indeed can be defined directly in terms of the topology.

Let now   $\mu$  be a Borel probability measure on $X$ that is invariant under $f$.
The measure theoretic entropy, or metric entropy, as it is often called, of the measure preserving dynamical system $(X,f,\mu)$ was defined by Kolmogorov and Sinai as follows (see \cite{Kol58,Kol59,Sin59}). Let $\xi$ be a finite measurable partition of $X$, and for $n \in \mathbb N$ define the refinement
\begin{equation}\label{refinement}
\xi_n = \xi \vee f^{-1}(\xi) \vee \cdots \vee f^{-(n-1)}(\xi).
\end{equation}
Then one defines
$$
h_\mu(f,\xi) = \lim_{n \rightarrow \infty} -\frac{1}{n}\sum_i \mu(C_i)\log (\mu(C_i)),
$$
where the sum runs over all subsets $C_i$ in the partition $\xi_n$ with $\mu(C_i)>0$. The \emph{measure theoretic entropy} $h_\mu(f)$ is obtained by taking the supremum over all finite measurable partitions $\xi$.

There is a clear relationship between topological and measure theoretic entropy:

\begin{theorem}[Variational Principle]
    For $X$ a compact metric space and $f: X \rightarrow X$ continuous, one has
    $$
    h_{\mathrm{top}}(f) = \sup_{\mu\in \mathcal{M}'_f(X)} h_\mu(f),
    $$
    where $\mathcal{M}'_f(X)$ denotes the set of all ergodic Borel probability measures.
\end{theorem}

This leaves open the question of whether there exists an ergodic probability measure $\mu$ for which $h_{\mathrm{top}}(f) = h_\mu(f)$, and if it exists, whether such measure of maximal entropy is unique. In general neither of these properties needs to hold unless additional properties are assumed, for example that the map is expanding or expansive. While neither assumption holds for typical rational functions,  in this setting the uniqueness of the measure of maximal entropy still holds by the results of Lyubich~\cite{Lyubich83} and Freire--Lopes--Ma\~n\'e~\cite{FLM}, as previously discussed.

The situation is more complicated for a transcendental function $f\colon \C\to \C$, since $\C$ is not compact. 
For continuous functions on non-compact spaces there are several notions of topological entropy; notions that are in general not equivalent.  We refer to \cite{BFP1,henon3} for discussions  on the definition of topological entropy  in the non-compact case. 
We define  $h_{\mathrm{top}}(f) $ as the supremum of the topological entropy of $f|_K$ as $K$ varies among all $f$-invariant compact subsets. As stated in Theorem \ref{inftop} we have  $h_{\mathrm{top}}(f) =\infty$, and by the variational principle applied to the invariant compact subsets it immediately follows that 
 $$
    \sup_{\mu\in \mathcal{M}'_f(X)} h_\mu(f)=\infty.
    $$
     In the current paper we are only considering measure theoretic entropy, so the non-compactness of the space is less of an issue.
     Notice  that  the support of an $f$-invariant measure with infinite entropy is necessarily non-compact, since on an $f$-invariant compact subset $K\subset \C$ the topological entropy can be bounded from above using the lower box dimension of $K$ and the Lipschitz constant of $f|_K$.

Computing the entropy of a measurable dynamical system may in general be very difficult. There is a relatively simple formula available when the dynamical system is induced by a Markov process, a setting which we will discuss in the next subsection.

\subsection{Markov processes}
We review some basic notions on Markov processes.
We consider Markov processes which depend on the following data: a finite or countable state space $S$, and for each pair $(s,t) \in S^2$ a transition probability $P_{st}\geq 0$, where for each fixed $s$ we have that $\sum_tP_{st}=1$. We think of $P_{st}$ as the probability of passing from state $s$ to state $t$; a probability that is independent of the time at which the system is in state $s$.  Notice that the probability of passing from state $s$ to state $t$ after $n$ steps is given by $(P^n)_{st}$, where $P^n$ denotes the $n$-th power of the (possibly infinite) matrix $(P_{st})$.
Let $\nu$ be an initial probability measure on $S$, that is, we interpret $\nu(t)$ as the probability of being in the state $t$ at time 0.
Then if $n\geq 1$, the probability  of being in the state $t$ at time $n$ is given by 
$$\nu^{(n)}(t)= \sum_{s \in S} \nu(s) \cdot (P^n)_{st}.$$
We say that a probability measure $\mu$ on $S$ is \emph{stationary} if
$\mu^{(1)}=\mu$ (and thus $\mu^{(n)}=\mu$ for all $n\geq 1$), that is
 for all $t\in S$,
\begin{equation}\label{eq:stationary}
\mu(t) = \sum_{s \in S} \mu(s) \cdot P_{st}.
\end{equation}

A  Markov process with a probability measure $\mu$ on $S$ induces a probability space, called the \emph{canonical realization} of the process, as follows:
consider the Borel $\sigma$-algebra on the space $S^{\mathbb N}$ (where as topology we consider the product topology w.r.t. the discrete topology on $S$).  A \emph{cylinder set}   is a subset of $S^{\mathbb{N}}$ of the form
$$
C^{m_1,\dots ,m_k}_{t_1,\dots, t_k} = \{(x_0, x_1, \ldots)\in S^{\mathbb{N}} : x_{m_i} = t_{i} \; \mathrm{for} \; 1\le i \le k\},
$$
where $m_j$ are in $\N$. 
In case $m_1=0, m_2=1,\dots, m_k={k-1}$ we denote $C^{0,\dots ,k-1}_{t_1,\dots, t_{k}}$ simply by  $C_{t_1,\dots, t_{k}}$.
The probability measure $\mathbb{P}^\mu$ on $S^{\mathbb N}$ is defined by
\begin{equation}\label{eq:cylinder measure}
\mathbb{P}^\mu (C_{s_0,\dots ,s_{k}}) = \mu(s_0) \cdot \prod_{j=0}^{k-1} P_{s_j s_{j+1}}.
\end{equation}
One can prove that such a measure exists and is unique, and it is easy to see that the measure $\mathbb{P}^\mu$ is invariant with respect to the left shift $\sigma\colon S^{\mathbb N}\to S^{\mathbb N}$ if and only if $\mu$ is stationary for the Markov process.
Given a random variable $T$ on $S^\N$, we  denote by $\mathbb{E}^\mu[T]:=\int_{S^\N}Td\mathbb{P}^\mu$
 its expected value with respect to the probability measure $\mathbb{P}^\mu$.
Given a state $s\in S$, let $T_s\colon S^\N\to \N$ be the random variable $$T_s(x):=\min\{n\geq 1:x_n=s\}.$$
For all $s,t\in S$ we denote $m_{st}:=\mathbb{E}^{\delta_s}[T_t],$ thus    $m_{ss}$ is the expected  return time to $s$, and if $s\neq t$, $m_{st}$ is the expected first passage time from $s$ to $t$.
 We will need  the following lemma giving upper estimates for the expected first passage  times to  a base state that we denote by 0.
\begin{lemma} \label{alternativeFoster}
Fix a base-state $0\in S$.
Let  $(y_s)\in (\R_{\geq 0})^{S\setminus\{0\}}$ be such that
\begin{equation}\label{Fosterlemma}
\sum_{t\in S\setminus\{0\}} P_{st} y_t \leq y_s-1  \quad \mathrm{for} \; \; s \neq 0.
\end{equation}
Then for every $s\neq 0$ we have $ m_{s0} \leq y_s.$
\end{lemma}
\begin{proof}
We have $$m_{s0}=\mathbb{E}^{\delta_s}[T_0]=\sum_{n\geq 1}n\P^{\delta_s}(T_0=n),$$ and 
$$\P^{\delta_s}(T_0=n)=\sum_{{s_1,\dots,s_{n-1}\neq 0}} P_{ss_1}\dots P_{s_{n-1}0}.$$
We show by induction that for all $k\geq 1$ $$y_s\geq \sum_{n=1}^kn\P^{\delta_s}(T_0=n)+ \sum_{s_1,\dots,s_k\neq 0} P_{ss_1}\dots P_{s_{k-1}s_k}(k+y_{s_k}).$$ Indeed by \eqref{Fosterlemma}
$$k+y_{s_k}\geq k+1+\sum_{s_{k+1}\neq 0}P_{s_ks_{k+1}}y_{s_{k+1}},
$$
 and thus, using 
 $\sum_{s_{k+1}\in S}P_{s_ks_{k+1}}=1$,
$$ 
\sum_{s_1,\dots,s_k\neq 0} P_{ss_1}\dots P_{s_{k-1}s_k}(k+y_{s_k}) \geq (k+1)\P^{\delta_s}(T_0=k+1)+\sum_{s_1,\dots,s_{k+1}\neq 0} P_{ss_1}\dots P_{s_{k}s_{k+1}}(k+1+y_{s_{k+1}}).$$

\end{proof}

 A Markov  process is \emph{irreducible} if for every $s,t\in S$ there exists $n\geq 1$ such that
 $(P^n)_{st}>0$.
If the Markov process is irreducible, then  there exists at most one stationary measure on $S$.
When the state space is finite, irreducibility is sufficient to guarantee that the unique stationary measure exists. However, when the state space is countably infinite, a second condition enters into play:
an irreducible Markov process is  \emph{positive recurrent} if the expected first return time $m_{ss}$  to state $s$  is finite (this condition does not depend on the state $s$). 
This implies that the expected first-passage time $m_{st}$  is finite for all states $s\neq t\in S$.
Moreover, if we fix a base-state $0\in S$, then for all $s\in S$ we have
\begin{equation}\label{Fosternew}m_{s0}=1+\sum_{t\in S\setminus\{0\}} P_{st} m_{t0}.
\end{equation}

 An irreducible Markov process admits a unique stationary  measure $\mu$ if and only if it is positive recurrent, and for all $s\in S$,
 $$\mu(s)=\frac{1}{m_{ss}}>0.$$

 A useful sufficient condition for an irreducible Markov process to be positively recurrent is given by Foster's Criterium:
\begin{theorem}[Theorem 2 in \cite{Foster}] \label{thm:Foster}
An irreducible Markov process on a countable state space $S$ with a base state $0\in S$ is positive recurrent if there exists a non-negative vector $y=(y_s)$ for which
\begin{equation}\label{eq:Foster1}
\sum_{t\in S\setminus\{0\}} P_{st} y_t \le y_s -1 \quad \mathrm{for} \; \; s \neq 0,
\end{equation}
and such that
\begin{equation}\label{eq:Foster2}
\sum_{t \in S \setminus\{0\}} P_{0t} y_t < \infty.
\end{equation}
\end{theorem}
Notice that Lemma \ref{alternativeFoster} gives a short alternative proof of Foster's Criterium, since 
$$m_{00}-1=\sum_{t \in S \setminus\{0\}} P_{0t} m_{t0}\leq \sum_{t \in S \setminus\{0\}} P_{0t} y_t<\infty.$$
We remark that condition (\ref{eq:Foster2}) is empty if there are only finitely many non-zero transition probabilities $P_{0t}$.


\begin{remark} 
Assume that the Markov process is irreducible and positive recurrent, and let $\mu$ be its unique stationary measure. Then
the  measure-preserving dynamical system $(S^\mathbb{N},\sigma,\mathbb{P}^\mu)$ is ergodic.
\end{remark}

\begin{defn}
Given an irreducible Markov process, a state $s$ is called  \emph{aperiodic} if $(P^n)_{ss}>0$ for all sufficiently large $n$.
One can show that if a state is aperiodic, then every state is, and then the Markov process is called  \emph{aperiodic}.
A  Markov process which is  aperiodic, irreducible and positive recurrent is called \emph{ergodic}. This terminology is widely adopted although it may create some  confusion, indeed as the previous remark shows aperiodicity is not necessary to obtain the ergodicity of the associated dynamical system.

\end{defn}

We will be interested in ergodic Markov processes mainly due to the following  result.
\begin{theorem}\label{theoremergodic}
Assume that the Markov process is ergodic and let $\mu$ be its unique stationary measure. Let $\nu$ be a probability measure on $S$. Then for all $t\in S$,
\begin{equation}\label{limitmeasure}\mu(t)=\lim_{n\to+\infty}\nu^{(n)}(t)=\lim_{n\to+\infty}\sum_{s\in S}\nu(s)(P^n)_{st}.\end{equation} 
In particular for all $s\in S$ one has
$\mu(t)=\lim_{n\to+\infty}(P^n)_{st}.$
Moreover the  measure-preserving dynamical system $(S^\mathbb{N},\sigma,\mathbb{P}^\mu)$ is \emph{mixing}.
\end{theorem}
It is easy to see that \eqref{limitmeasure} fails  without assuming aperiodicity: consider for example a two-state  Markov process with transition matrix
$P=\begin{bmatrix}
0&1\\
1&0
\end{bmatrix}.$
Then for all $n\geq 0$ we have $P^{2n}=I$ and $P^{2n+1}=P$.
\begin{remark}
As a consequence of Theorem \ref{theoremergodic}, we will show in Section \ref{sectionequidistribution} that, if the Markov process is ergodic, the measure-preserving dynamical system $(S^\mathbb{N},\sigma,\mathbb{P}^\mu)$  satisfies equidistribution of preimages and periodic points. For this reason in what follows we will work with ergodic Markov processes.
\end{remark}

We will construct ergodic Markov processes by explicitly constructing transition probabilities that provide a desired stationary measure. Then we will use Foster's Criterium to argue that small perturbations of the transition probabilities still give ergodic Markov processes with a unique stationary measure.
The measure theoretic entropy $h_{\mathbb{P}^\mu}(\sigma)$ of the dynamical system  $(S^{\mathbb N}, \sigma, \mathbb{P}^\mu)$ equals the \emph{entropy rate} of the Markov process, that is
\begin{equation}\label{eqn:entropy rate}
 h_{\mathbb{P}^\mu}(\sigma)=- \sum_{s,t} \mu(s) P_{st} \log(P_{st}),
\end{equation}
where the sum runs over all pairs $(s,t)$ for which $P_{st} > 0$, see for example formula (4.4.1) on page 175 of \cite{KatokHasselblatt}.
This simple formula will be used to determine that the entropy of the measures that we construct for the Baker--Bishop maps and the disjoint-type maps is infinite.

\section{Markov processes with infinite entropy rate}\label{sec:Markov}

In this section we will define explicit ergodic Markov processes with infinite entropy rate. The challenge lies not with defining such processes, but in constructing processes that can later be embedded into the desired complex dynamical systems.

We will first construct a simple model Markov process, for which we can easily compute the stationary measure and the entropy rate and which can easily be embedded into different complex dynamical systems. We will then observe that   the support of the measure when it is pushed forward to the complex plane is not the entire Julia set. Finally, we observe that by perturbing the Markov process in such a way  that many more transition probabilities become small but non-zero, we obtain new ergodic measures on $S^\mathbb{N}$ that still  can be pushed forward to the complex plane, and whose support is in fact equal to the entire Julia set.

\subsection{The model Markov process.}\label{sec:model markov process}
In this subsection we introduce the model Markov process with infinite entropy which we will  modify in the next subsection.
Let $k_2, k_3, \ldots$ be an increasing sequence of natural numbers, write $q_2 = 1$, and for $d \ge 2$ we recursively define $q_{d+1} = q_{d} + k_d$.

We define the state space $S$ by
$$
S := S_0\cup \bigcup_{d=2}^{\infty} S_d,
$$
where the sets $S_d$ are finite subsets of $\mathbb \N^2$ given by $S_0 = \{(0,1), (0,2)\}$ and for $d \ge 2$:

$$
S_d := \{(i,j)\; : \; q_{d} \le i < q_{d+1}, \quad 1 \le j \le d\}.
$$

We will define transition probabilities $P_{st}$ between states $s= (i_1, j_1)$ and $t = (i_2, j_2)$ satisfying the following three axioms:
\begin{itemize}
    \item[(A1)] $P_{st} > 0$ if and only if one of the following holds:
    \begin{enumerate}
        \item  $i_2 = i_1+1$,
        \item  $i_1 = i_2 = 0$, or
        \item  $(i_1, j_1) = (q_d, 1)$ for some $d \ge 3$ and $i_2 = 0$.
    \end{enumerate}
    \item[(A2)] Unless $i_2 = q_d$ for some $d \ge 3$, the transition probability $P_{st}$ is independent of $j_2$. Moreover, when $i_2 = q_d$ for some $d \ge 3$ the transition probabilities to states $(i_2, 2), \ldots, (i_2, d)$ are equal.
    \item[(A3)] Unless $i_1 = q_d$ for some $d \ge 3$, the transition probability $P_{st}$ is independent of $j_1$. Moreover, when $i_1 = q_d$ for some $d \ge 3$ the transition probabilities from states $(q_d,2), \ldots, (q_d, d)$ to a state $(q_d+1, j_2)$ are equal.
\end{itemize}

For all $ q\in \N $ define the \emph{column}    $$\mathscr{C}_q:=\{(q,j)\colon j=1,\dots, d\}.$$ 
For all $d\geq 3$ we call both the column $\mathscr{C}_{q_d}$  and the state $(q_d,1)$  \emph{special}. Notice that the column $\mathscr{C}_{q_2}=\mathscr{C}_{1}$ and the state  $(q_2,1)=(1,1)$ are not special.

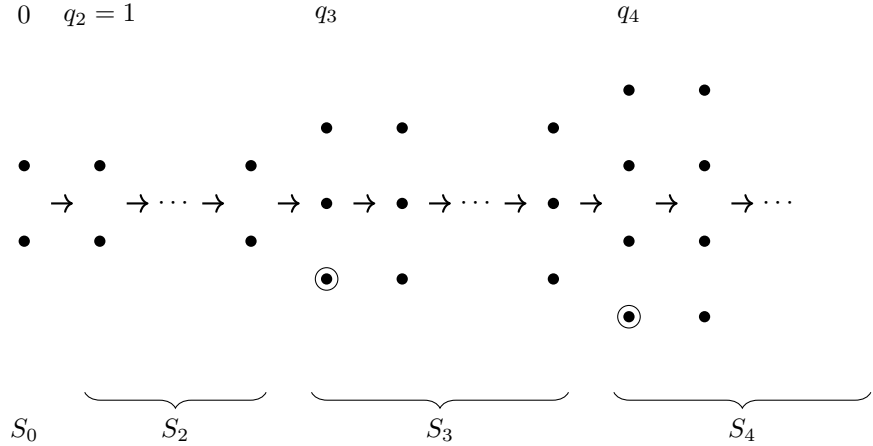
\begin{figure}[h]
  \begin{tikzpicture}[
    dot/.style={circle, fill=black, inner sep=1.5pt},
    circdot/.style={circle, draw=black, inner sep=3pt},
    arrow/.style={->, thick, shorten >=3pt, shorten <=3pt}
]

\node[dot] (a1) at (0,0) {};
\node[dot] (a2) at (0,-1) {};

\node[dot] (b1) at (1,0) {};
\node[dot] (b2) at (1,-1) {};

\node at (2,-0.5) {$\cdots$};

\node[dot] (c1) at (3,0) {};
\node[dot] (c2) at (3,-1) {};

\node[dot] (a1) at (4,0.5) {};
\node[dot] (a2) at (4,-0.5) {};
\node[dot] (a2) at (4,-1.5) {};

\node[dot] (a1) at (5,0.5) {};
\node[dot] (a2) at (5,-0.5) {};
\node[dot] (a2) at (5,-1.5) {};

\node at (6,-0.5) {$\cdots$};

\node[dot] (a1) at (7,0.5) {};
\node[dot] (a2) at (7,-0.5) {};
\node[dot] (a2) at (7,-1.5) {};

\node[dot] (a1) at (8, 1) {};
\node[dot] (a2) at (8, 0) {};
\node[dot] (a2) at (8,-1) {};
\node[dot] (a2) at (8,-2) {};

\node[dot] (a1) at (9, 1) {};
\node[dot] (a2) at (9, 0) {};
\node[dot] (a2) at (9,-1) {};
\node[dot] (a2) at (9,-2) {};

\node at (10,-0.5) {$\cdots$};


\node[circdot] at (4,-1.5) {};
\node[circdot] at (8,-2) {};


\node at (0,2) {$0$};

\node at (1,2) {$q_2=1$};

\node at (4,2) {$q_3$};

\node at (8,2) {$q_4$};


\draw[arrow] (0.25,-0.5) -- (0.75,-0.5);
\draw[arrow] (1.25,-0.5) -- (1.75,-0.5);
\draw[arrow] (2.25,-0.5) -- (2.75,-0.5);
\draw[arrow] (3.25,-0.5) -- (3.75,-0.5);
\draw[arrow] (4.25,-0.5) -- (4.75,-0.5);
\draw[arrow] (5.25,-0.5) -- (5.75,-0.5);
\draw[arrow] (6.25,-0.5) -- (6.75,-0.5);
\draw[arrow] (7.25,-0.5) -- (7.75,-0.5);
\draw[arrow] (8.25,-0.5) -- (8.75,-0.5);
\draw[arrow] (9.25,-0.5) -- (9.75,-0.5);


\node at (0,-3.5) {$S_0$};

\draw[decorate,decoration={brace,mirror,amplitude=6pt}]
    (0.8,-3) -- (3.2,-3);
\node at (2,-3.5) {$S_2$};

\draw[decorate,decoration={brace,mirror,amplitude=6pt}]
    (3.8,-3) -- (7.2,-3);
\node at (5.5,-3.5) {$S_3$};

\draw[decorate,decoration={brace,mirror,amplitude=6pt}]
    (7.8,-3) -- (11.2,-3);
\node at (9.5,-3.5) {$S_4$};

\end{tikzpicture}

\caption{An illustration of the model Markov Chain. The black dots indicate the states $(q,j)$, where some of the first coordinates are indicated above the columns. The encircled dots correspond to the special states $(q_d,1)$ for $d \ge 3$. The arrows signify the positive transition probabilities from all states in a column to all states in the next column. Note that the positive transition probabilities from each of the special states $(q_d,1)$ to the states in $S_0$ are not visible in the illustration.}
\label{fig:modelmarkov}
\end{figure}

\begin{remark}\label{remarkgraph}
Note that axiom (A1) sets the underlying directed graph, and  that every  Markov process with this underlying directed graph is irreducible.
Notice also that the directed graph  is locally finite: from each state $s$ there are only finitely many states $t$ for which $P_{st} > 0$.  Since the directed graph contains loops from states in $S_0$ to itself, the Markov process is also aperiodic. Hence if we show that it has a stationary measure, then it is necessarily positive recurrent and thus ergodic.
\end{remark}

\begin{remark}
If $\nu$ is a stationary measure for a Markov process satisfying the axioms (A1)-(A3), then for every $d\geq 2$ every state of $S_d$ not lying in a special column must have the same mass, which we denote by $m_d>0$. Analogously, if $d\geq 3$, the non-special states in a special column must have the same mass, which we denote $m_d'>0$.

\end{remark}

Instead of first defining the transition probabilities and then determining the corresponding stationary measure, we will take the opposite route:

\begin{prop} [Construction of the model Markov process]\label{prop:existence}
Let $c_2 , c_3, \ldots$ be a sequence of strictly positive real numbers satisfying
$$
\sum_{d\geq 2} c_d = \frac{1}{2}.
$$ 
and let $k_2, k_3, \ldots$ be  a sequence of positive integers increasing sufficiently fast.
 Then there exists a unique probability measure $\nu$ on the set $S$ 
 which is a stationary measure for an ergodic Markov process whose transition probabilities satisfy axioms (A1) - (A3), and
 such that 
\begin{itemize}
\item $\nu(0,j) = 1/4$ for $j = 1,2$;  
\item $\nu(S_d) = c_d$ for each $d\geq 2$;
\item for all $d\geq 3$ every non-special state has the same mass (thus $m_d=m'_d$).
 \end{itemize}
 Furthermore the Markov process is uniquely determined.
\end{prop}

\begin{proof}
Set $c'_2:=c_2$, and define inductively $c_3',c_4'\ldots$ such that 
\begin{equation}\label{defc'}
c_d= c_d'+\frac{c_{d-1}'}{k_{d-1}}-\frac{c_d'}{k_d},
\end{equation}
that is $c_d'(1-1/k_d)=c_d-c_{d-1}'/k_{d-1}$.
Notice that every  $c_d'$ is strictly positive if the sequence   $(k_d)$ increases  sufficiently fast, and that we can moreover ensure that for all $d\geq 3$,
\begin{equation}\label{ensurepositive}\frac{c_{d-1}'}{k_{d-1}}-\frac{c_d'}{k_d}>0
\end{equation} and thus $c'_d<c_d$.
We will construct a  measure $\nu$  on $S$ 
 which is a stationary measure for a Markov process whose transition probabilities satisfy axioms (A1) - (A3), and
 such that 
\begin{itemize}
\item $\nu(0,j) = 1/4$ for $j = 1,2$;  
\item  for all $d\geq 2$, every non-special state of $S_d$ has mass $\frac{c'_d}{dk_d}$.
 \end{itemize}
Notice that the only masses to be determined are the masses of the special states. The mass of  the state $(q_d,1)$, where $d\geq 3$, is easy to determine using stationarity of $\nu$ on the special column $\mathscr{C}_{q_d}$. Indeed we have that necessarily
$$
\sum_i\nu(q_d-1,i)= \frac{c_{d-1}'}{ k_{d-1}} =\sum_i\nu(q_d,i)=(d-1)\frac{c_d'}{d k_d} +\nu(q_d,1),
$$
and thus 
$$\nu(q_d,1) =\frac{c_{d-1}'}{ k_{d-1}}-(d-1)\frac{c_d'}{d k_d},$$
which is positive thanks to \eqref{ensurepositive}. 
Stationarity also allows us to compute the mass $x_d$ that goes from the special state $(q_d,1)$, where $d\geq 3$, to $S_0$. Indeed $x_d$ is equal to the mass of the special column $\mathscr{C}_{q_d}$ minus the mass of the  column $\mathscr{C}_{q_d+1}$, that is
$$x_d=\frac{c'_{d-1}}{k_{d-1}}-\frac{c'_d}{k_d},$$
which again  is positive thanks to  \eqref{ensurepositive}.   Hence the total amount of mass which enters $S_0$ from  the special states $(q_d,1)_{d\geq 3}$ is
$$\sum_{d\geq 3} \frac{c'_{d-1}}{k_{d-1}}-\frac{c'_d}{k_d}=\frac{c'_2}{k_2}.$$
Notice that these data determine uniquely the measure $\nu$ and the transition probabilities of the Markov process. It follows from \eqref{defc'} that $\nu(S_d)=c_d$ for all $d\geq 2$, and thus $\nu$ is a probability measure.
By Remark \ref{remarkgraph}, the Markov process is ergodic.
 \end{proof}

\begin{prop}[Infinite Entropy]\label{prop:entropy}
Suppose that the  sequence $(c_d)_{d\geq 2}$ further satisfies
$$
\sum_{d\geq 2} c_d \log d = \infty.
$$
 Then if $(k_d)$ increases sufficiently fast the entropy  of the shift $\sigma\colon S^\N\to S^\N$ w.r.t. the measure $\mathbb{P}^\nu$ is infinite.
\end{prop}
\begin{proof}
  We compute the entropy rate of the Markov process  (see (\ref{eqn:entropy rate})).  Summing over the states of the form $s = (i,j)$ with $q_{d-1} < i < q_d-1$ gives $d \cdot (k_d-2)$ states, each having  mass $c'_d/(d k_d)$ and having $d$ transition probabilities all equal to $1/d$ gives the estimate
    $$
    h_{\mathbb{P}^\nu}(\sigma)  \ge 
      \sum_{d \ge 2} \frac{k_d-2}{k_d} c'_d \cdot \log d.
    $$
    For all $d\geq 2$
we have that $$c_d-c_d'\leq \frac{c_{d-1}'}{k_{d-1}}\leq \frac{c_{d-1}}{k_{d-1}}.$$
Hence,   if the sequence $(k_d)$ increases sufficiently fast, then the sequences $(c_d)$ and $(c'_d)$ are asymptotically equivalent.
    Since $\sum c'_d \log d$   diverges, we conclude that $    h_{\mathbb{P}^\nu}(\sigma) =\infty$.
\end{proof}

We note that for a suitably chosen constant $c>0$ the sequence
$$
c_d = \frac{c}{d (\log d)^2}
$$
satisfies the hypotheses.

\begin{remark}[Escaping sequences have measure zero]\label{remark: full measure}
   Consider a positively recurrent Markov process with any probability distribution $\mu$ on $S$. Then the subset of  \emph{escaping} sequences, that is, the sequences that eventually exit every given finite subset of $S$,   has measure zero  w.r.t. $\mathbb{P}^\mu$.
 Since there are only countably many states it is sufficient to prove that the set of escaping sequences that start at some fixed state  has measure zero. 
    Suppose for the purpose of a contradiction that the set of escaping sequences from some fixed state $s$ has positive measure. Then it follows that
     the set of sequences starting at state $s$ that never return to  $s$ has positive measure.
     But this implies that the expected return time from $s$ to itself is infinite, which contradicts positive recurrence.
\end{remark}
	
\subsection{Modified  Markov processes}\label{sec:modified Markov process}

To obtain our main result for Baker--Bishop maps and for disjoint-type maps, we will need to modify the model Markov process  in such a way that:
\begin{itemize}
\item The modified Markov process remains ergodic with infinite  entropy rate. (Denote by $\mu$ its stationary measure.)
\item The symbolic dynamical system associated with the modified Markov process embeds in the given complex dynamical system $(\C,f)$ in such a way that the Julia set of $f$ coincides with $\supp\varphi_*\mathbb{P}^\mu=\overline{\varphi(\Sigma)}$.
\end{itemize}

\begin{definition}
Consider a positive recurrent Markov process over a countable set of states $S$.
We say that a sequence $x=(x_0,x_1\ldots)\in S^\N$ is \emph{allowable} if  $P_{x_ix_{i+1}}>0$ for all $i\in\N$. 
\end{definition}
\begin{remark}\label{allowablesupport}
If $\mu$ denotes the stationary measure,  the support of the measure $\mathbb{P}^\mu$ coincides with the subset of allowable sequences in $S^\N$.
\end{remark}
In order to be able to obtain that $\supp \phi_*\mathbb{P}^\mu=J$, we will need to modify the process so that 
more transition probabilities become strictly positive,   enlarging the subset of allowable sequences, and thus the support of $\mathbb{P}^\mu$.

\begin{prop}[Perturbing the model Markov process]\label{prop:extension of the support} Denote by $(P_{st})$ the transition probabilities of the model Markov process.  Let  $E\subset S^2$ be the set of edges $(s,t)$ for which $P_{s t} > 0$. Let $ E\subset E'\subset S^2$.
Then there exists an ergodic Markov process on $S$ with infinite entropy rate whose transition probabilities $(P_{s t}')$ are 
strictly positive if and only if $(s,t) \in E'$.
\end{prop}
Recall  that in the model Markov process, for each $s$ there are only finitely many $t$ such that $P_{st}>0$.
\begin{proof}
First of all, notice that irreducibility and aperiodicity are properties of a Markov process which depend only on the associated directed graph with  vertex set $S$ and with an edge from $s$ to $t$ iff $P_{st}>0$. Irreducibility and aperiodicity are preserved by increasing the set of non-zero transition probabilities.
For all $s\in S$ let   $\TT(s)\subset S$ be the set of states $t\in S$ such that   $(s,t)\in  E$, and let $\TT'(s)\subset S$ be the set of $t\in S$ such that   $(s,t)\in E'\setminus E$. 
Let $s\in S$.
\begin{enumerate}
\item We set  $P'_{st}=0$ for all $t\in S\setminus (\TT(s)\cup \TT'(s)) $.
\item If  $\TT'(s)$ is finite, we set  $P'_{st}=\epsilon^{(s)}$  for all $t\in\TT'(s)$, where $\epsilon^{(s)}>0$ is small enough (to be determined later).
If $\TT'(s)$ is not finite, then  it is a countable set $(t_j)_{j\geq 1}$, and we set $P'_{st_j}=\epsilon^{(s)}_j$ for all $j\geq 1$, where $\epsilon_j^{(s)}\searrow 0$  fast enough (to be determined later).
\item  We set $0<P'_{st}\leq P_{st}$  for all $t\in \TT(s)$, in such a way that $\sum_{t\in S}P'_{st}=1.$
\end{enumerate}
The modified Markov process is irreducible and aperiodic.

 Let $\frac{1}{2}<a_0<1.$
We now claim that if we choose all the  $\epsilon^{(s)}$   small enough (or the $\epsilon_j^{(s)}\searrow 0$ fast enough),  the modified Markov process satisfies   Foster's Criterium,
 and, if  $\mu$ denotes the stationary measure of the modified Markov process,
 the mass $\mu(0,1)$ is larger than $\frac{a_0}{4}$. Choose $(0,1)$ as the base-state 0.
 Let $m'_{00}$ (resp. $m_{00}$) denote the expected  first return time to state 0  for the modified (resp. model) Markov process.
 Let $1<c<\frac{1}{3}(\frac{4}{a_0}-1)$.
 For every state $s\in S\setminus\{0\}$, let $m'_{s0}$ (resp, $m_{s0}$) denote the expected first passage time for the modified (resp. model) Markov process from  state $s$ to the base state $0$. 
Since the model process is positive recurrent, the vector $(m_{s0})$ satisfies  \eqref{Fosternew}.
Define  $y_s=cm_{s0}$ for all $s\neq 0$.
Then  for every state $s\neq 0$ we obtain a strict inequality
$$
\sum_{t\in S\setminus\{0\}} P_{s t} y_{t} < y_{s} -1.
$$
We can easily make sure that this inequality remains verified for the new transition probabilities $(P'_{st})$. Indeed,
if $\TT'(s)$ is  finite, choose $\epsilon^{(s)}$  small enough so that the inequality remains satisfied (notice that decreasing  transition probabilities with $t\in \TT(s)$ to obtain $\sum_{t\in S}P'_{st}=1$ will only improve the inequality). If $\TT'(s)$ is  countable, choose $(\epsilon^{(s)}_j)$  decreasing fast  enough. 

We also need to modify the transition probabilities when $s=0$.  Notice that
 $$3=m_{00}-1=\sum_{t\in S\setminus\{0\}} P_{0t}m_{t0},$$
 and  choose $\epsilon^{(0)}$ (or $(\epsilon^{(0)}_j)$) in such a way that  
 \begin{equation}\label{basestate}\sum_{t\in S\setminus\{0\}} P'_{0t}m_{t0}\leq \frac{1}{c}\left(\frac{4}{a_0}-1\right).
 \end{equation}
This proves that the modified Markov process is  positive recurrent. By Lemma \ref{alternativeFoster} we have 
$m'_{s0}\leq cm_{s0}$  for all $s\neq 0$, and thus $$m'_{00}=1+\sum_{t\in S\setminus\{0\}} P'_{0 t} m'_{t0}\leq 1+c\sum_{t\in S\setminus\{0\}} P'_{0t}m_{t0}\leq \frac{4}{a_0},$$
where we used  \eqref{basestate}.
 The proof for the  state $(0,2)$ is analogous.

Recall that for all $q\geq 0$, the column   $\mathscr{C}_q$ is defined as  $\{(q,j)\colon j=1,\dots, d\}\subset S.$
For all $q\geq 0$ define $$\delta_q:= \max_{\substack{ s\in \mathscr{C}_q\\ t\in  \mathscr{C}_{q+1}}} \frac{P_{st}}{P'_{st}}.$$
Notice that $\delta_q\geq 1$. Let $(a_q)_{q\geq 1}$ be a strictly decreasing sequence of positive numbers, with $a_1<a_0$ and $a_q\searrow \frac{1}{2}$.
We now prove that if 
the modified Markov process is such that for all $q\geq 0$
 $$\delta_q\leq \frac{a_{q}}{a_{q+1}},$$  then  for all $s\in S$ we have 
\begin{equation}\label{nodrop}
\mu(s)\geq\nu(s)a_q\geq  \frac{\nu(s)}{2}. 
\end{equation}
Notice that this  implies that the entropy rate of the modified Markov process is also infinite. 
We prove \eqref{nodrop} by induction. Indeed, the case $q=0$ immediately follows from the above.
 Fix $q\geq 1$, and assume that for all $s\in  \mathscr{C}_{q-1}$ we have that $\mu(s)\geq \nu(s)a_{q-1}.$
Let $t\in \mathscr{C}_q$. Since for the model Markov process the only mass flowing to the column $\mathscr{C}_q$ comes from the column  $\mathscr{C}_{q-1}$ we have that 
$$\nu(t)=\sum_{s\in  \mathscr{C}_{q-1}}\nu(s)P_{st}.$$

Therefore
$$\mu(t)\geq \sum_{s\in  \mathscr{C}_{q-1}}\mu(s)P'_{st}\geq \sum_{s\in  \mathscr{C}_{q-1}}a_{q-1}\nu(s)P_{st}\frac{a_q}{a_{q-1}}=\nu(t)a_q.$$

\end{proof}

\begin{remark}\label{modificationbishop}
In the case of Baker--Bishop maps (Section \ref{sec: Bishop maps}) we will modify the Markov process by setting 
$$
E'=E\cup \{(s,t) \text{ such that $s=(q_d,1)$ for some $d\geq2$ and $t=(i,j)$ with $i\leq q_d$}\}.
$$
\end{remark}
We now discuss the modification of the model Markov process in the case of disjoint-type maps (Section \ref{sec:class B}).
We  identify the countable state space $S$ with $\mathbb N$  by ordering $S$ lexicographically, that is
$$(i_1,j_1)\leq (i_2,j_2) \iff i_1< i_2\,\,\mbox{or}\, \,(i_1= i_2\,\,\mbox{and}\, \,j_1\leq j_2).$$
 \begin{definition}\label{def:quadratic growth}
Let $\QQ\subset \N^\N$ denote the subset of sequences that grow at most quadratically, that is the set of sequences $x= (x_0, x_1, \ldots) \ldots\in\N^\N$ for which there exists $a,b>0$ such that
$$
x_n \leq (a+bn)^2 \text{\ \ \ for all $n\in\N$}.
$$
\end{definition}
\begin{remark}
Identify $S$ with $\N$ as above. Then, if the sequence $(k_n)$ increases sufficiently fast,  the transition probabilities $P_{s t}$ for the model Markov process satisfy 
 
 \begin{equation}\label{eq: vanishing probabilities}
    P_{st}=0 \text{\ \ \ for  $t\geq (\sqrt{s}+1)^2$.} 
 \end{equation}

\end{remark}

\begin{prop}[Modification for disjoint-type maps]\label{prop:modify process to prescribe support}
Consider the model Markov process  from Section~\ref{sec:model markov process} and let   $E\subset \N^2$ be the set of edges $(s,t)$ for which $P_{s t} > 0$.  Consider a set $E'\subset \N^2$ such that  $ E'\supset E$. 
Then there exists a Markov process on $\N$ which is  ergodic and with infinite entropy rate, whose transition probabilities $(P_{s t}')$ are 
strictly positive if and only if $(s,t) \in E'$, and such that $\mathbb{P}^\mu(\mathcal{Q})=1$, where $\mu$ denotes its unique stationary measure.

\end{prop}
\begin{proof}
  Define the  transition probabilities $P'_{st}$ in such a way    that Proposition~\ref{prop:extension of the support} holds and moreover
   \begin{equation}\label{eq:quadratic assumption}
       P'_{st}\leq 2^{-2(s+t)} \text{\ \ \ \ for $t \geq (\sqrt{s}+1)^2$}.
   \end{equation}
   The last condition can be imposed because of (\ref{eq: vanishing probabilities}). Then  the modified Markov process remains  ergodic and the entropy remains infinite by Proposition~\ref{prop:extension of the support}. It remains to show that $\mathbb{P}^\mu (\QQ)=1$.

Fix an initial state  $s\in \N$.
We  will show by induction on $n\geq1$   that for any $n_0\geq n\in\N$,
 \begin{equation}\label{eq: hypothesis}
\mathbb{P}^{\delta_{s}}(x_n\geq (\sqrt{x_0}+n+n_0)^2)\leq \frac{1}{2^{2n_0-n}}\leq \frac{1}{2^{n_0}}.
\end{equation}
By \eqref{eq:cylinder measure}, we have for all $n,t\in \N$,
$$
\mathbb{P}^{\delta_{s}}(x_n=t)= \sum_{\substack{x\in S^{n+1} \\ x_0=s\\ x_n=t}}  \prod_{i=0}^{n-1} P'_{x_i x_{i+1}}.
$$

For $n=1$ and for all $n_0\geq1$,  using (\ref{eq:quadratic assumption}) we have that
$$
\mathbb{P}^{\delta_{s}}(x_1\geq (\sqrt{x_0}+1+n_0)^2)=\sum_{t\geq (\sqrt{s}+1+n_0)^2}P'_{st}\leq\sum_{t\geq (\sqrt{s}+1+n_0)^2}2^{-2(s+t)}\leq
\frac{4}{3}\frac{1}{2^{2(\sqrt{s}+1+n_0)^2+2s}}\leq \frac{1}{2^{2n_0-1}},
$$
hence (\ref{eq: hypothesis}) is satisfied for $n=1$.

Now assume by induction  that~\eqref{eq: hypothesis} holds for a given $n\ge 1$ and all $n_0\geq n$, and let us  prove that it  holds for $n+1$ and all $n_0\ge n+1$. 
 We have
$$
 \mathbb{P}^{\delta_s}(x_{n+1}  \ge ( \sqrt{x_0} + n+ n_0+ 2)^2 )=\sum_{k\in\N}\mathbb{P}^{\delta_s}(\{x_{n+1}\geq ( \sqrt{x_0} + n+ n_0+2)^2 \}\cap\{ x_n=k\}).$$
We split the sum over all $k \in \mathbb N$ into two sums, depending on whether $k$ is smaller than or  equal to $(\sqrt{s} + n + n_0 + 1)^2$. We have that
\begin{align*}
&\sum_{k \ge (\sqrt{s} + n + n_0 + 1)^2}\mathbb{P}^{\delta_s}(\{x_{n+1}\geq ( \sqrt{x_0} + n+ n_0+2)^2 \}\cap\{ x_n=k\}) \leq \sum_{k \ge (\sqrt{s} + n + n_0 + 1)^2}\mathbb{P}^{\delta_s}(x_n=k)\\
&=\mathbb{P}^{\delta_s}(x_n\geq (\sqrt{x_0} + n + n_0 + 1)^2 )\le \frac{1}{2^{2(n_0+1)-n}},
\end{align*}
where we  used \eqref{eq: hypothesis}.

On the other hand, assumption \eqref{eq:quadratic assumption} implies that 
\begin{align*}
&\sum_{k < (\sqrt{s} + n + n_0 + 1)^2}\mathbb{P}^{\delta_s}(\{x_{n+1}\geq ( \sqrt{x_0} + n+ n_0+2)^2 \}\cap\{ x_n=k\}) =  \sum_{\substack{k< (\sqrt{s} + n + n_0 + 1)^2\\  t\geq(\sqrt{s} + n + n_0 + 2)^2}}\mathbb{P}^{\delta_s}(x_n=k)P'_{kt} \\
 &\leq   \left(\sum_{k< (\sqrt{s} + n + n_0 + 1)^2}\mathbb{P}^{\delta_s}(x_n=k) \right)\left(\sum_{t\geq(\sqrt{s} + n + n_0 + 2)^2}\frac{1}{2^{2t}}\right)\leq  \frac{1}{2^{2(\sqrt{s} + n + n_0 + 2)^2}}\frac{4}{3}\leq \frac{1}{2^{2n_0}},
\end{align*}
where we used that any sum over probabilities $\P^{\delta_s}(x_n = k)$ for fixed $n$ but different values of $k$ equals at most $1$. Adding the two inequalities obtained for the two subsums implies equation \eqref{eq: hypothesis}.

Let $k\geq 1$. 
Notice that
$$\N^\N\setminus\QQ\subset  \bigcup_{n\geq 1}\{x_n \geq (\sqrt{x_0}+(1+k)n)^2\}.$$
By equation  \eqref{eq: hypothesis} we obtain
$$
\mathbb{P}^{\delta_s}(\N^\N\setminus\QQ)\leq \mathbb{P}^{\delta_s}(\cup_{n\geq 1}\{x_n \geq (\sqrt{x_0}+(1+k)n)^2\} ) \leq \sum_{n=1}^{\infty}\frac{1}{2^{kn}}=\frac{1/2^k}{1-1/2^k}\ra0 \text{\ as $k\ra\infty$},
$$
and thus  $\mathbb{P}^{\delta_s}(\N^\N\setminus\QQ)=0$. The result follows since $$\mathbb{P}^{\mu}(\QQ)=\sum_{s\in \N}\mu(s)\mathbb{P}^{\delta_s}(\QQ)=1.$$
\end{proof}

\begin{remark}
Notice that we can make sure that $\QQ$ is contained in the support of $\mathbb{P}^\mu$ by taking $E'=\N^2$. Indeed, this implies that the support of $\mathbb{P}^\mu$ is $\N^\N$ by Remark~\ref{allowablesupport}.

\end{remark}

\subsection{Equidistribution of preimages and periodic points}\label{sectionequidistribution}

The purpose of this subsection is to show that for any ergodic Markov process on a countable state space $S$, it follows that the ergodic measure on $S^\N$ can also be obtained by weighted equidistribution of either preimages or periodic cycles. It follows that the embedded measure in the complex plane can similarly be obtained via weighted equidistribution.

 Recall that, if $X$ is a Polish space and  $(\mu_n)$ is a sequence of finite  measures on $X$,  the sequence $(\mu_n)$  converges {\sl weakly} to a finite measure $\mu_\infty$ if for all continuous bounded functions $f\colon X\to \R$ one has that $\int fd\mu_n\to\int fd\mu_\infty$.
\subsubsection{Time-reversed Markov process} 
 Consider a positive recurrent Markov process $(P_{st})$ on a countable state space $S$, and let $\mu$ be its unique stationary measure. Then, given $n\geq 1$ and $s,t\in S$, we have by the stationarity of $\mu$,
 $$
 \mathbb{P}^\mu(x_{n-1}=s|x_{n}=t)=\frac{\mu(s)}{\mu(t)}P_{st},
 $$
independently on $n$, and we have 
$$
\sum_{s\in S}\mathbb{P}^\mu(x_{n-1}=s|x_{n}=t)=1.
$$ 
Therefore 
$$
Q_{ts}:=\mathbb{P}^\mu(x_{n-1}=s|x_{n}=t)
$$
are transition probabilities of a new Markov process, called \emph{time-reversed}.
We will use the transition probabilities of the time-reversed Markov process to define the equidistribution weights. 

For future reference  we now compute the transition probabilities of the time-reversed Markov process $(Q_{ts})$ associated with the model Markov process $(P_{st})$ defined in Proposition \ref{prop:existence}.
\begin{lemma}\label{computationtrans}
\,
\begin{enumerate}
\item  $Q_{(q+1,j)(q,i)}=1/d$ for all $d\geq 2$, and  $q_d\leq q< q_{d+1}$;
\item $Q_{(1,j)(0,i)}=\frac{1}{2}$ for all $i,j=1,2$;
\item $Q_{(0,j)(0,i)}=\frac{1}{2}-\frac{c_2}{k_2}$ for all $i,j=1,2$;
\item $Q_{(0,j)(q_d,1)}=\frac{2c'_{d-1}}{k_{d-1}}-\frac{2c'_d}{k_d}$ for all $j=1,2$ and $d\geq 3$.

\end{enumerate}
\end{lemma}

\begin{proof}
We prove (1). Let  $q_d\leq q<q_{d+1}-1$ and assume that $(q,i)$ is not a special state. Then $\mu(q,i)=\mu(q+1,j)$, and thus $Q_{(q+1,j),(q,i)}=P_{(q,i)(q+1,j)}=\frac{1}{d}$.
Assume now that $(q,i)$ is   special, that is $q=q_d$, $i=1$ and $d\geq 3.$
The mass of each state $(q+1,j)$ and each state $(q, i)$ for $i \neq 1$  is equal to a number $m_d$.

By the stationarity of the measure we have
$$
P_{(q_d,1) (q_d+1,j)} =\frac{ \frac{m_d}{d} }{\mu(q_d,1)}, 
$$
and thus
$$
Q_{(q_d+1,j),(q_d,1)} = P_{(q_d,1) (q_d+1,j)} \cdot \frac{\mu(q_d,1)}{\mu(q_d+1,j)}= \frac{1}{d}.
$$
We now consider the case $q=q_{d+1}-1$. Let $1\leq j\leq d+1$ and $1\leq i\leq d$.
Then $$Q_{(q_{d+1},j)(q,i)}=P_{(q,i)(q_{d+1},j)}\frac{\mu(q,i)}{\mu(q_{d+1},j)},$$
which does not depend on $i$ thanks to the axiom A3 in  Section \ref{sec:model markov process}.
Since $\sum_{1\leq i\leq d}Q_{(q_{d+1},j)(q,i)}=1$, this implies that $Q_{(q_{d+1},j)(q,i)}=\frac{1}{d}$ for all $i$.

(2),(3),(4) are simple computation.
\end{proof}
\begin{definition}
For all $d\geq 3$ we denote by $\alpha_d$ the transition probability $$Q_{(0,1)(q_d,1)}=Q_{(0,2)(q_d,1)}=\frac{2c'_{d-1}}{k_{d-1}}-\frac{2c'_d}{k_d}.$$

\end{definition}

\subsubsection{Weighted equidistribution of preimages} \label{sec:Equidistribution preimages Markov process}

 Consider an ergodic Markov process $(P_{st})$ on a countable state space $S$, so that there exists a unique stationary measure $\mu$. Let $\mathbb{P}^\mu$ be  the probability measure induced on $S^{\mathbb N}$. Denote by $\Theta:={\rm supp}(\mathbb{P}^\mu)$  the subset of allowable sequences.
 Fix an allowable sequence $w = (w_0, w_1, \ldots) \in S^\mathbb N$, and let  $z = (z_0, z_1, \ldots) \in \sigma^{-1}(w)$. We define the weight $Q(w,z)$ as  
$$Q(w,z)=Q_{w_0z_0}=\frac{\mu(z_0)}{\mu(w_0)}P_{z_0w_0}.$$
Clearly $Q(w,z)\neq 0$ if and only if $z$ is allowable. 
We have that $$\sum_{z\in \sigma^{-1}(w)}Q(w,z)=\sum_{z_0\in S}Q_{w_0,z_0}=1,$$ and thus 
$Q_w:=\sum_{z\in \sigma^{-1}(w)}Q(w,z)\delta_z$ is a probability measure on $S^\N$ with 
${\rm supp}(Q_w)=\sigma^{-1}(w)\cap \Theta.$
We can use the weights $Q(w,z)$ to define a sequence of probability measures $(Q^n_w)$ on $S^\N$ as follows. For $n \ge 1$ and $z\in \sigma^{-n}(w)$ set
$$Q^n(w,z):= \prod_{j=0}^{n-1} Q(\sigma^{j+1}(z), \sigma^j(z))=\frac{\mu(z_0)}{\mu(w_0)}\prod_{i=0}^{n-1}P_{z_iz_{i+1}}=\mathbb{P}^\mu(x_0=z_0,\dots, x_{n-1}=z_{n-1}|x_n=w_0),$$ where we recall that $z_n=w_0$. 
We have that $\sum_{z\in \sigma^{-n}(w)}Q^n(w,z)=1$, and thus 
$Q^n_w:=\sum_{z\in \sigma^{-n}(w)}Q^n(w,z)\delta_z$ is a probability measure on $S^\N$ with 
${\rm supp}(Q^n_w)=\sigma^{-n}(w)\cap\Theta.$

%
%

\begin{prop}[Weighted equidistribution of preimages]\label{equiprop}
Let $(P_{st})$ be an ergodic Markov process on a countable state space $S$. Let $w\in \Theta$. Then
  the measures $Q^n_w$ converge weakly to the measure  $\mathbb{P}^\mu$.
\end{prop}
\begin{proof}
Let $\mathcal{C}$ be the family of all cylinder sets  $C_{z_0,\dots, z_m}$ where $m\geq 1$ and $z_i\in S$. Then $\mathcal{C}$ is a countable basis for the product topology on $S^\N$, and it is closed w.r.t. finite intersections. It follows from \cite[Theorem 2.2]{Bil99} that to prove weak convergence it is enough to  
 show that for every cylinder set $C=C_{z_0,\dots, z_m}\in \mathcal{C}$ we have that 
$$Q^n_w(C)\stackrel{n\to\infty}\longrightarrow \mathbb{P}^\mu(C).$$
We can assume that $n>m$. Then 
$$Q^n_w(C)=\frac{\mu(z_0)}{\mu(w_0)}\left(\prod_{i=0}^{m-1}P_{z_iz_{i+1}}\right) (P^{n-m})_{z_mw_0}.$$
Since the Markov process is ergodic, it follows from Theorem \ref{theoremergodic}
that $(P^{n-m})_{z_mw_0}\stackrel{n\to\infty}\longrightarrow \mu(w_0)$.
\end{proof}

\subsubsection{The backward dynamics of the shift as a Markov process on $\Theta$}
Notice that the subset of allowable sequences $\Theta:={\rm supp}(\mathbb{P}^\mu)$ is forward invariant for $\sigma$.
The family of probability measures $(Q_w)_{w\in \Theta}$ can be thought of as the transition kernel of  a Markov process on the state space $\Theta$, that is,  
 $Q_w$  is the probability measure representing position  at time 1 if we start at  time 0 with an initial  probability measure $\delta_w$.
Recalling that the space of probability measures on $\Theta$ is equipped with the topology of weak convergence, where test functions are bounded continuous functions, it follows that the map $w\to Q_w$ is continuous from $\Theta$ to the space of probability measures on $\Theta$, hence the Markov process is Feller.

If we start at time 0 with  an  initial probability measure $\nu$ on $\Theta$, then the probability measure representing position  at time $n$ is given by
$$\nu^{(n)}(B):=\int_\Theta Q^n_w(B)d\nu(\omega),$$ where $B\subset \Theta$ is any measurable subset.

A probability measure $\nu$ on $\Theta$ is {\sl stationary} if
$\nu^{(1)}=\nu$, that is, for all measurable $B\subset \Theta$ we have
\begin{equation}\label{eq:measure of preimages integral}
\nu(B)=\int_\Theta Q_w(B)d\nu(w)=\int_\Theta\left( \sum_{z\in \sigma^{-1}(w)\cap B} Q(w,z)\right) d\nu(w).
\end{equation}
If $\nu$ is stationary, and if $g\colon B\to \Theta$ is a continuous branch of $\sigma^{-1}$ defined on a Borel set $B\subset \Theta$, then 
\begin{equation}\label{endsummer}
\nu(g(B))=\int_BQ(w,g(w))d\nu(w).
\end{equation}

Proposition \ref{equiprop} immediately yields the following corollary.
\begin{corollary}\label{Pmustationary}
The probability measure $\mathbb{P}^\mu$ is stationary for the Markov process $(Q_w)$ and for all probability measure $\nu$ on $\Theta$, the sequence
$(\nu^{(n)})$ converges weakly to $\mathbb{P}^\mu$.
\end{corollary}

\subsubsection{Equidistribution of periodic cycles}\label{equiperiod}
 Consider a positive recurrent Markov process on a countable state space $S$, and let $\mu$ be its unique stationary measure. 
Our aim  is to introduce for all $n\geq 1$ a measure $\nu_n$ on $S^{\mathbb N}$,  with support contained on the subset ${\rm Per}_n(\sigma)=\{z\in S^\N\colon \sigma^n(z)=z\}$ of periodic points of period $n$, in such a way that  $\nu_n\to \mathbb{P}^\mu$ weakly as $n \rightarrow \infty$. 
Let $y=(y_0,y_1,\dots)\in {\rm Per}_n(\sigma),$ that is, $y_n=y_0$.
 We define the mass of $y$  as 
 $$
 Q^n(y,y)=\prod_{i=0}^{n-1}P_{y_iy_{i+1}}=\mathbb{P}^\mu(x_0=y_0,\dots, x_{n-1}=y_{n-1}|x_n=y_0).
 $$
We thus define $$\nu_n=\sum_{y\in {\rm Per}_n(\sigma)}\left( \prod_{i=0}^{n-1}P_{y_iy_{i+1}} \right)\delta_y.$$
In contrast with what happens for the measure $Q^n_w$, the measures $\nu_n$ need not be probability measures, the total mass $\nu_n(S^\N)$ could be any finite number including $0$, and could be infinite.
On the other hand, every measure $\nu_n$ is clearly $\sigma$-invariant, and
since the mass of a periodic point $y$  is  non-zero only when all consecutive transition probabilities are strictly positive, the support of the measure $\nu_n$ is the subset $ {\rm Per}_n(\sigma)\cap {\rm Supp}(\mathbb{P}^\mu)$.

\begin{lemma}\label{lemma: return chance}
For each state $s \in S$ one has
$$
\sum_{\substack{ y\in {\rm Per}_n(\sigma)\\ y_0=s}} \nu_n(y) \rightarrow \mu(s)
$$
as $n \rightarrow \infty$.
\end{lemma}
\begin{proof}
Since the Markov process is ergodic, by Theorem \ref{theoremergodic} for each pair of states $(s,t)$ the probability that a path starting at $s$ ends in $t$ after $n$ steps converges to $\mu(t)$ as $n \rightarrow \infty$, and thus
\begin{equation}\label{eq:convergence to stationary measure}
\mu(t)=\lim_{n\to+\infty} (P^n)_{st}= \lim_{n\to+\infty} \sum_{\substack{z\in S^{n+1} \\ z_0=s \\ z_n=t}}\prod_{j=0}^{n-1} P_{z_j z_{j+1}}.
\end{equation}
The claim then follows  taking $s=t$.
\end{proof}

If the total  masses of the measures $\nu_n$ are eventually finite and uniformly bounded, then Lemma \ref{lemma: return chance} implies that the total mass of the measures $\nu_n$ converges to $1$. However this  need not be the case for a general
positive recurrent Markov process on a countable state space $S$.
We will therefore consider the specifics of the model Markov process defined in Proposition \ref{prop:existence}.

\begin{lemma}\label{lemma: convergence of mass}
Consider the model Markov process  on the countable state space $S$, 
and denote by $\mu$  its unique stationary measure.  
Then there exists $n_0\geq 0$ such that for all $n\geq n_0$ and  for all $s\in S$ we have
\[
\sum_{\substack{ y\in {\rm Per}_n(\sigma)\\ y_0=s}} \nu_n(y) \le 4 \mu(s),
\] and in particular $\nu_n(S^\N)\leq 4$ for all $n\geq n_0$.
\end{lemma}
\begin{proof}

Consider a state $s\in \mathscr{C}_q$ with $q\geq 1$.
 Note that a path  starting at $s$  can return to $s$ at time $n\geq 1$ only if it visits  $S_0$ at time $n-q$.
  In particular $n$ needs to be at least as large as $q$.

We have
\[
\sum_{\substack{y\in S^{n-q+1} \\ y_0=s }} \prod_{i=0}^{n-q-1} P_{y_i, y_{i+1}} = 1,
\]
and thus
\[
\sum_{\substack{y\in S^{n-q+1} \\ y_0=s\\ y_{n-q}\in S_0 }} \prod_{i=0}^{n-q-1} P_{y_i, y_{i+1}} \le 1.
\]
It follows that
\begin{equation}\label{4measure}
\begin{aligned}
\sum_{\substack{ y\in {\rm Per}_n(\sigma)\\ y_0=s}} \nu_n(y) & =
\sum_{\substack{y\in S^{n+1} \\ y_0=y_n=s}}\prod_{i=0}^{n-1} P_{y_i y_{i+1}}
& \le \sum_{\substack{(y_{n-q},\dots, y_n)\in S^{q+1} \\ y_{n-q}\in S_0\\y_n=s}} \prod_{i=n-q}^{n-1} P_{y_i,y_{i+1}}.
\end{aligned}
\end{equation}
Notice that the only states $t\in S$ for which $P_{ts}^q$ is different from 0 are the two  states $(0,1),(0,2)$ in  $S_0$. Hence by 
stationarity of the measure $\mu$ we have that 
$\mu(s)=\frac{1}{4}(P^q_{(0,1)s}+P^q_{(0,2)s}),$ which, together with \eqref{4measure}, implies that for all $n\geq 1$,
$$\sum_{\substack{ y\in {\rm Per}_n(\sigma)\\ y_0=s}} \nu_n(y)\leq 4\mu(s).$$
It remains to consider the case $s\in S_0$. By the previous lemma, there exists $n_0\geq 1$ such that if $s\in S_0$, then for all $n\geq n_0$ we have $$\sum_{\substack{ y\in {\rm Per}_n(\sigma)\\ y_0=s}} \nu_n(y) \leq 4 \mu(s),$$
which yields the result.
\end{proof}

\begin{remark}
The proof of the above lemma uses in an essential way that any periodic orbit from a state $s \in \mathscr{C}_q$ for $q \geq 1$ must pass through $S_0$ before returning to $s$, a fact that no longer holds for the modified Markov processes  described in Section~\ref{sec:modified Markov process}. One can instead compute that given any $\epsilon>0$, the modifications to the model Markov process can be made sufficiently small such that 
\[
\sum_{\substack{ y\in {\rm Per}_n(\sigma)\\ y_0=s}} \nu_n(y) \le (4+\epsilon) \mu(s),
\]
which is sufficient for our next result.
\end{remark}

\begin{thm}\label{lido}
For both the model Markov process and the modified Markov process,  the measures $\nu_n$ converge  weakly to the measure  $\mathbb{P}^\mu$.
\end{thm}
\begin{proof}
By Lemma \ref{lemma: convergence of mass} the mass of the measures $\nu_n$ converges to $1$ as $n \rightarrow \infty$. It is therefore sufficient to prove that for every cylinder set $C=C_{z_0,\dots, z_m}$ we have that 
\begin{equation}\label{eq: cylinders converge}
\nu_n(C)\stackrel{n\to\infty}\longrightarrow \mathbb{P}^\mu(C).
\end{equation}

We can assume that $n>m$. Thus
$$
\nu_n(C) = \prod_{i=0}^{m-1} P_{z_i z_{i+1}} \cdot (P^{n-m})_{z_mz_0}.
$$
Since the Markov process is ergodic, it follows from Theorem \ref{theoremergodic}
that $(P^{n-m})_{z_mz_0}\stackrel{n\to\infty}\longrightarrow \mu(z_0)$, hence we obtain the desired equation \eqref{eq: cylinders converge}.
\end{proof}

\section{Invariant measures for Baker--Bishop maps} \label{sec: Bishop maps}

In this section we construct a subclass of Baker--Bishop maps and prove Theorems \ref{Theorem: Main1} and \ref{Theorem: Main2}.
 We start by constructing the maps (Theorem~\ref{thm:Construction Theorem 1D}),
 then embed an appropriate symbolic dynamical system, and use the results from Section~\ref{sec:Markov} to push-forward an ergodic measure with infinite entropy (Section~\ref{sec:Bishop measure}). 

\begin{definition}\label{defn:annuli}
  Let $4<M<6$. Given an increasing  sequence of radii $(r_q)_{q\geq 1}$ in $\mathbb R_+$ we define the sequences of annuli $C_q$ and $\bC_q$ by
    $$
    C_q:=\{r_q/10<|z|<10r_q\}, \quad     \bC_q:=\{r_q/(10M)<|z|<10M r_q\}.
    $$
If the sequence $(r_q)$ is such that the annuli $\bC_q$ have pairwise disjoint closures, we can define two sequences of {\sl gaps} between the annuli:
    $$
    G_q:=\{10r_q<|z|<r_{q+1}/10\},\quad \bG_q:=\{10Mr_q<|z|<r_{q+1}/(10M)\}.
    $$
\end{definition}

\begin{definition}\label{defn:covering}
Let $A$  be a bounded  domain  in $\C$ with positively oriented $C^1$ boundary.  Let $B\subset \C$ be a bounded domain.  Let $f: \C \rightarrow \C$ be a holomorphic map. We say that
$f|_A$ {\sl covers $n$ times} the domain $B$ if 
\begin{enumerate}
\item $f(\partial A)\cap  B = \emptyset$, and
\item the curve $f(\partial A)$ has winding number $n$ with respect to any point in $B$, or equivalently  every point in $B$ has  $n$  preimages  in $A$ (counted with multiplicities).
\end{enumerate}
\end{definition}

\begin{remark}\label{rem:stabilityofcoverings}
If in addition $f(\partial A)\cap {\overline B} = \emptyset$, then for any holomorphic function $g\colon \C\to \C$ sufficiently close to $f$ on $\overline A$ the map $g|_A$ also covers  $n$ times the domain $B$.
\end{remark}
We refer to a proper simply connected domain in $\C$ as a \emph{holomorphic disk}.

In the following theorem we construct the class of Baker--Bishop maps that we are interested in.

\begin{theorem}\label{thm:Construction Theorem 1D}
There exist
\begin{itemize}
\item a transcendental entire  function  of the form
\begin{equation}\label{eq:f}
f(z) = a_0 + z^2 + a_3 z^3 + \ldots,
\end{equation}
where all coefficients (except $a_1=0$) are positive real, and where the sequence $(a_d)_{d\geq 2}$ is strictly decreasing;
\item  an unbounded  strictly increasing sequence of positive radii $(r_q)_{\geq 1}$ defining two sequences of  annuli $(C_q)_{q\geq 1}$ and $(\bC_q)_{q\geq 1}$ as in Definition \ref{defn:annuli} with pairwise disjoint closures;
\item a set $C_0\Subset \DD_{r_1/10}$ consisting of two holomorphic disks $C_0^+$ and $C^-_0$ whose boundaries are real analytic and whose closures are disjoint;
\item  a strictly  increasing sequence of integers $(q_d)_{d\geq 2}$ with $q_2=1$;
\end{itemize}
such that the following hold:
\begin{itemize}
\item[(i)]  the maps $f|_{C_0^+}$ and $f|_{C_0^-}$ both cover once the disc $\DD_{10Mr_1}$ whose boundary is 
the outer boundary of the annulus $\bC_1$;
\item[(ii)] for all $d \ge 2$ and all $q_d \le q < q_{d+1}$, the map $f|_{C_q}$ covers $d$ times the annulus $\bC_{q+1}$;
\item[(iii)]  for all $d \ge 3$ the map $f|_{C_{q_d}}$ also covers once the disc $\DD_{10Mr_{q_d}}$
whose boundary is 
the outer boundary of the annulus $\bC_{q_d}$;
\item[(iv)] the domain $G_0:=\DD_{\frac{r_1}{10}}\setminus \overline C_0$ is mapped inside  the gap ${\bf G}_1$ and for all $q\geq 1$ the gap  $G_q$ is mapped  inside the gap ${\bf G}_{q+{1}}$, hence the complement of the set $\bigcup_{q\geq 0}{\overline C_q}$ is forward invariant and contained in the escaping set;
\item[(v)] all critical points of $f$ are contained in the forward invariant set $(\bigcup_{q\geq 0}{\overline C_q})^\complement$, and thus escape to infinity.
\end{itemize}
\end{theorem}

\begin{remark}

Let $\Gamma^\pm_0=\partial C_0^\pm$. 
 For all $q\geq 1$, let $\gamma_q, \Gamma_q$ denote respectively the interior and exterior components of the  boundary of the annulus $C_q$.
Then properties (i)-(iv) above follow from the following conditions on the images of $\gamma_q,\Gamma_q$:
\begin{enumerate}
\item the curves $f(\Gamma^\pm_0)$ are contained in the gap ${\bf G}_1$  and each has winding number $1$ around any point in the bounded component of ${\bf G}_1^\complement$;
\item the curve  $f(\gamma_1)$ is contained in the gap ${\bf G}_1$ and has winding number $2$ around any point in the bounded component of ${\bf G}_q^\complement$;
\item for all $d\geq 2$ and all $q_d< q\leq q_{d+1}$ the curve  $f(\gamma_q)$ is contained in the gap ${\bf G}_q$ and has winding number $d$ around any point in the bounded component  of ${\bf G}_q^\complement$;
\item for all $d\geq 2$  and all $q_d\leq  q< q_{d+1}$ the curve 
$f(\Gamma_q)$ is contained in the gap ${\bf G}_{q+1}$ and has winding number  $d$ around any point in the bounded component  of ${\bf G}_{q+1}^\complement$;

\end{enumerate}
\end{remark}

\begin{definition}\label{defn:special annuli} 
 We call  $(q_d)_{d\geq 3}$ \emph{special integers}. Notice that $q_2=1$ is not a special integer (cf.  (iii) of Theorem \ref{thm:Construction Theorem 1D}, where $d\geq 3$).
For all $d\geq 3$ we   set $\rho_d := r_{q_d}$ and we define the \emph{special annulus}
$$
A_d := C_{q_d}=\{\rho_{d}/10<|z|<10\rho_{d}\}.
$$
We will later choose the sequence $(\rho_d)_{d\geq 3}$ in such a way that for all $d\geq 2$,
$$
\rho_{d+1}=\frac{10\cdot2d}{2d+1}\frac{a_d}{a_{d+1}}.
$$

\end{definition}

\begin{rem}\label{importantremark}
\,
\begin{enumerate}
\item For any sufficiently fast increasing sequence of integers $(q_d)_{d\ge 2}$  one can construct a function $f$ as in Theorem~\ref{thm:Construction Theorem 1D}.
\item     It will be clear from the proof that the moduli of the sequence $(a_d)_{d \ge 3}$ can be guaranteed to decrease arbitrarily fast. 
 \item Working with positive coefficients may not be essential, but significantly simplifies the construction.
 \item The radii $\rho_d$ are chosen such that for $|z| \sim \rho_{d+1}$ the two terms $a_d z^d$ and $a_{d+1}z^{d+1}$ are comparable in modulus, and dominate all other terms of the power series of $f$.
\end{enumerate}
\end{rem}

\subsection{Four preparatory lemmas}
We state and prove several lemmas that are needed in the proof of Theorem~\ref{thm:Construction Theorem 1D}.
\begin{remark}\label{trivialestimate}
Let $p_d$ be a degree $d$ polynomial  and let $M$ be the maximum modulus of its coefficients.
Then if $|z|\geq 2$ we have
$$|p_d(z)|\leq \frac{M(|z|^{d+1}-1 )}{|z|-1}< 2M|z|^d.$$

\end{remark}
\begin{lem}\label{lem:basic estimates} Let $d\geq 2$ and let
$$
p_d(z):=a_0 + a_2 z^2 + a_3 z^3 + \ldots +a_d z^d
$$
with $a_i\in\R_+$ having strictly decreasing moduli.
Fix  $0<\delta<1$. Then for all $|z|\geq 2a_0/(\delta a_d)$ we have
 \begin{equation}\label{monomialestimate2}
 |p_d(z)-a_dz^d|=|\sum_{j=0}^{d-1}a_jz^j|< \delta a_d |z|^d.
 \end{equation}
and thus
\begin{equation}\label{eq:monomial estimate}
(1-\delta)a_d |z|^d< |p_d(z)|<  (1+\delta)a_d |z|^d.
\end{equation}
\end{lem}

\begin{proof}
If $|z|\geq 2a_0/(\delta a_d)$, then also $|z|\geq 2$, and by the previous remark
$$
|\sum_{j=0}^{d-1}a_jz^j|<  2a_0|z|^{d-1}\leq \delta a_d |z|^d.
$$
\end{proof}

\begin{lem}\label{lem:basic estimates new}
Let $d\geq 2$, let $p_d$ as in Lemma~\ref{lem:basic estimates}. Given $0<a_{d+1}<a_d$, define  $p_{d+1}:=p_d+a_{d+1}z^{d+1}$ and $\rho_{d+1}:=10\frac{2d}{2d+1}\frac{a_d}{a_{d+1}}$.
Then if $a_{d+1}$ is small enough,
\begin{itemize}
\item  for all $|z|\geq \rho_{d+1}$,
\begin{equation}\label{stimacruciale}
|p_{d+1}(z)-a_{d+1}z^{d+1}|=|p_d(z)|< \frac{1}{5}a_{d+1}|z|^{d+1}.
\end{equation}
\item for all $|z|\geq \frac{\rho_{d+1}}{10}$,
\begin{equation}\label{stimacruciale2}
|p_{d+1}(z)-a_{d+1}z^{d+1}|=|p_d(z)|< 2a_{d+1}|z|^{d+1}.
\end{equation}
\item for all $\frac{(4d+2)2a_0}{a_d}\leq |z|\leq \frac{\rho_{d+1}}{10}$,
\begin{equation}\label{stimacruciale3}
|p_{d+1}(z)-a_d z^d|< \frac{4d+1}{4d+2}a_d|z|^d.
\end{equation}
\item for all $\frac{(4d+2)2a_0}{a_d}\leq |z|\leq \frac{\rho_{d+1}}{100}$,
\begin{equation}\label{stimacruciale4}
|p_{d+1}(z)-a_d z^d|< \frac{1}{5}a_d |z|^d.
\end{equation}
\end{itemize}
\end{lem}
\begin{proof}
To prove \eqref{stimacruciale}, we use the previous lemma with $\delta=\frac{1}{2}$.
If $a_{d+1}$ is small enough to have that $\rho_{d+1}\geq \frac{4a_0}{a_d}$, then if $|z|\geq \rho_{d+1}$,
  $$
\frac{|p_d(z)|}{a_{d+1}|z|^{d+1}}< \frac{3}{2}\frac{a_{d}}{a_{d+1}}\frac{1}{|z|}\leq  \frac{3}{2}\frac{a_{d}}{a_{d+1}}\frac{1}{\rho_{d+1}}=\frac{3}{2}\frac{2d+1}{10\cdot 2d}< \frac{1}{5},
$$
which proves \eqref{stimacruciale}. The proof of~\eqref{stimacruciale2} follows similarly.

We now prove \eqref{stimacruciale3}. Assume that $a_{d+1}$ is small enough to have $\frac{\rho_{d+1}}{10}\geq \frac{(4d+2)2a_0}{a_d}$.
We have
$$
|p_{d+1}(z)-a_d z^d|\leq a_{d_1} |z|^{d+1}+|\sum_{j=0}^{d-1}a_jz^j|.
$$
Since $|z|\leq \frac{\rho_{d+1}}{10}$,
it follows that
$$
\frac{a_{d+1}|z|^{d+1}}{a_d |z|^d}=\frac{a_{d+1}}{a_d}|z|\leq \frac{a_{d+1}}{a_d}\frac{\rho_{d+1}}{10}=\frac{2d}{2d+1}.
$$
Equation \eqref{stimacruciale3} now follows using the previous lemma with $\delta=\frac{1}{4d+2}$, using the fact that
$|z|\geq (4d+2)2a_0/a_d$.

The proof of \eqref{stimacruciale4} is again similar.
\end{proof}

\begin{lem}\label{lem:covering properties middle piece}
Let $d\geq 2$, let $p_d$ as in Lemma~\ref{lem:basic estimates}. Given $0<a_{d+1}<a_d$, define  $p_{d+1}:=p_d+a_{d+1}z^{d+1}$ and $\rho_{d+1}:=10\frac{2d}{2d+1}\frac{a_d}{a_{d+1}}$.
Let $x_0:=\sqrt[d-1]{2\cdot 10^{d+1}M/a_d}.$
For $x>x_0$ denote the annuli
$$
C(x):=\{\frac{x}{10}<|z|<10x\},  \quad \bC(x):=\{\frac{x}{10M}<|z|<10Mx\},
$$
the radii
$$
x':=a_dx^d,\quad x'':=a_dx'^d=a_d^{d+1}x^{d^2},
$$
and the gaps
$$
G(x):=\{10x<|z|<\frac{x'}{10}\},\quad \bG(x):=\{10Mx<|z|<\frac{x'}{10M}\}.
$$
Denote $\gamma(x), \Gamma(x)$  respectively  the inner and outer boundary  of the annulus
$C(x)$.
Assume that $a_{d+1}$ is small enough such that
$$
\frac{\rho_{d+1}}{10^3}\geq \max\left\{x_0, \frac{(4d+2)10\cdot 2a_0}{a_d}\right\}:=\bar x.
$$
Then, for all $\bar x\leq x\leq \frac{\rho_{d+1}}{10^3}$,
\begin{itemize}
\item[(i)] the image of $\Gamma(x)$ via $p_{d+1}$ is contained in the unbounded component of $\overline{\bC(x')}^\complement$ and in the bounded component of $\overline{\bC(x'')}^\complement$, and has winding number $d$ around any point of the disk of radius $10Mx'$;
\item[(ii)] the image of $\gamma(x)$ via $p_{d+1}$ is contained in the unbounded component of $\overline{\bC(x)}^\complement$ and in the bounded component of $\overline{\bC(x')}^\complement$, and has winding number $d$ around any point of the disk of radius $10Mx$.
\end{itemize}
\end{lem}

We note that the gaps $G(x)$ and $\bG(x)$ are well-defined since $x>x_0\geq \sqrt[d-1]{10^2M^2/a_d}$.

\begin{proof}
Clearly the image of $\Gamma$ via $a_{d}z^{d}$ is a circle $C$ of radius $a_d x^d 10^d$.
The circle $C$ is contained in the gap $\bG(x')$ and
$$
d(C,\bC(x'))=a_d x^d (10^d-10M), \quad d(C,\bC(x''))=a_d x^d (a_d^d x^{d(d-1)}/10M-10^d).
$$
By \eqref{stimacruciale4},
$$
|p_{d+1}(z)-a_d z^d|< \frac{1}{5}a_d|z|^d= \frac{10^d}{5}a_d x^d< \min\{d(C,\bC(x')), d(C,\bC(x''))\}.
$$
Point (i) thus follows from Rouch\'e's theorem.

The image of $\gamma$ via $a_{d}z^{d}$ is a circle $c$ of radius $\frac{a_dx^d}{10^d}$.
The circle $c$ is contained in the gap $\bG(x)$ and
$$d(c,\bC(x'))=a_d x^d\frac{10^{d-1}-M}{10^nM},\quad d(c,\bC(x))= \frac{a_d x^d}{10^d}-10Mx.$$
By \eqref{stimacruciale4},
$$
|p_{d+1}(z)-a_d z^d|< \frac{1}{5}a_d |z|^d= \frac{1}{5\cdot 10^d}a_dx^d< \min\{d(c,\bC(x')), d(c,\bC(x))\}.
$$
Point (ii) therefore also follows from Rouch\'e's theorem.
\end{proof}

As an immediate corollary we also obtain:
\begin{corollary} Under the assumptions of Lemma~\ref{lem:covering properties middle piece} one  has that
 the map
$p_{d+1}|_{C(x)}$ covers $d$ times  $\bC(x')$   and that 
$p_{d+1}(G(x))\Subset  \bG(x').$
\end{corollary}

\begin{lem}[New critical point]\label{lem:new critical points}Let $d\geq 2$, and let $p_d$ as in Lemma~\ref{lem:basic estimates}. Given $0<a_{d+1}<a_d$, define  $p_{d+1}:=p_d+a_{d+1}z^{d+1}$ and $\rho_{d+1}:=10\frac{2d}{2d+1}\frac{a_d}{a_{d+1}}$.
 Then if $a_{d+1}$ is sufficiently small, there is a critical point of $p_{d+1}$ in the disk
 $$
 D\left(\frac{-d a_d}{(d+1)a_{d+1}},\frac{1}{2(d+1)}\frac{\rho_{d+1}}{10} \right)\subset \DD_\frac{\rho_{d+1}}{10}.
 $$
 \end{lem}

\begin{proof} Critical points of $p_{d+1}$ are solutions of the equation
$$
0=p_{d+1}'(z)=2a_2z+3a_3z^2+\ldots+(d+1)a_{d+1}z^{d}.
$$
For $|z|$ large enough, the two leading terms in the equation  are $d a_d z^{d-1}$ and $(d+1)a_{d+1}z^{d}$, so we expect the new critical point to be close to the nonzero solution of the equation $0=d a_d z^{d-1}+(d+1)a_{d+1}z^{d}$, which is
$$
x_0:=\frac{-da_d}{(d+1)a_{d+1}}.
$$
To make this more precise we invoke Rouche's Theorem again. Let  $g(z)=d a_d z^{d-1}+(d+1)a_{d+1}z^{d}$, which has a zero at $x_0$. Notice that $|x_0|<\rho_{d+1}/10$.

 Notice that the distance between $x_0$ and $\{|z|=\rho_{d+1}/10\}$ is computed as
 $$
 \frac{a_d}{a_{d+1}}\frac{d}{(d+1)(2d+1)}=\frac{\rho_{d+1}}{10}\frac{1}{2(d+1)},
 $$
 and that the two real points on the boundary of the disk $U$ centered at $x_0$ and tangent to the circle $\{|z|=\rho_{d+1}/10\}$
 are
 $$
 x_+:=-\frac{\rho_{d+1}}{10},\quad x_-:=-\frac{2d^2 a_d}{a_{d+1}(2d+1)(d+1)}=-\frac{d}{d+1}\frac{\rho_{d+1}}{10}.
 $$
Since $g(z)=z^{d-1}(z-x_0)a_{d+1}(d+1)$ it follows that   for all $z\in \partial U$ we have
$$
|g(z)|\geq |g(x_-)|=a_d\frac{d}{2d+1}\left(\frac{d}{d+1}\frac{\rho_{d+1}}{10}\right)^{d-1}.
$$

Let $z\in \partial U$. If $\rho_{d+1}$ is sufficiently large  (that is: $a_{d+1}$ is sufficiently small) we have by Remark \ref{trivialestimate}
$$
|(p_{d+1}'-g)(z)|\leq 2M|z|^{d-2}\leq 2M|x_+|^{d-2}=2M\left(\frac{\rho_{d+1}}{10}\right)^{d-2},
$$
where $M={\max_{2\leq j\leq d-1}}ja_j$.
 Hence $|p_{d+1}'-g|<|g|$ on $\partial U$ if  $a_{d+1}$ sufficiently small, and the claim follows by Rouche's Theorem.
\end{proof}

\subsection{Proof of  Theorem~\ref{thm:Construction Theorem 1D}}

Theorem~\ref{thm:Construction Theorem 1D} will be proved by induction. We start by defining a polynomial $p_2$ of degree 2 that will be used as the seed of the induction.
Given any $r_1>1$ define a sequence  $(r_q)_{q\geq 1}$ setting $r_{q+1}=r_q^2$ for all $q\geq 1$, that is $r_q=(r_1)^{2^{q-1}}$.  Notice that if $r_1>100M^2$ then the annuli $(\bC_q)_{q\geq 1}$ have pairwise disjoint closures. Note that some of the radii $r_q$, and hence also the corresponding annuli $C_{r_q}$ and $\bC_{r_q}$, may change every time a later coefficient $a_d$ is determined.

 \begin{lem}[Construction of $p_2$]\label{lem:seed}
 If $r_1>1$ is sufficiently large, then  we can choose  $a_0>0$ so that the polynomial $p_2:=a_0+z^2$ has the following properties:
\begin{itemize}
    \item[(i)] there exists  a  set $C_0\Subset \DD_{\frac{r_1}{10}}$ consisting of two holomorphic disks $C_0^{\pm}$ whose closures are disjoint and whose boundaries are real analytic, such that
    the two maps $p_2|_{C_0^{\pm}}\colon C_0^{\pm}\to \DD_{15M r_1}$ are biholomorphisms;
     \item[(ii)]  for all $q\geq 1$, the map
$p_2|_{C_q}$ covers $\bC_{q+1}$  two times;
    \item[(iii)] $G_0:=\DD_{\frac{r_1}{10}}\setminus\overline C_0$ is mapped compactly inside $\bG_1$ and for all $q\geq 1 $ the gap $G_q$ is mapped compactly inside $\bG_{q+1}$.
\end{itemize}
\end{lem}

\begin{proof} Let  $r_1$ sufficiently large such that
\begin{equation}    \label{eq:r1 large}
\frac{r_1}{10}>2\sqrt{10Mr_1},
\end{equation}
and set $a_0:=20Mr_1$.  Let  $\Sigma_1$ be the circle centered at the origin passing through the unique critical value $p_2(0)=a_0$, and let
    $\Sigma_0$ be the preimage of $\Sigma_1$. Then $\Sigma_0$ is a vertical figure eight  passing through 0.
 Since $p_2(0)=a_0=20Mr_1>15Mr_1$, the preimage of $\DD_{15Mr_1}$ consists of two holomorphic disks $C_0^{+}, C_0^{-} $,  one in each bounded connected component of the complement of  $\Sigma_0$.  By construction the closures of the two holomorphic disks are disjoint, and their boundaries are real analytic. Each of the holomorphic disks is mapped biholomorphically to $\DD_{15Mr_1}$.
    A direct computation shows that
\begin{equation}\label{eq:fig8}
|w|\leq\sqrt{{2a_0}}=2\sqrt{10Mr_1}  \text{\ \ \  for $w\in\Sigma_0$}.
\end{equation} It follows from \eqref{eq:r1 large} that $\Sigma_0\subset \DD_{r_1/10}$, and this proves (i).
Notice that, by  choosing $r_1>200M^2$, we have that  the circle $\partial \DD_{15Mr_1}$ and the critical value $a_0$ are  contained in the gap $\bG_1$. Points (ii) and (iii) are easily obtained by choosing $r_1$ large enough.
\end{proof}

\begin{proof}[Proof of Theorem~\ref{thm:Construction Theorem 1D}]
Let $a_0$  be given by the previous lemma.  Let $\delta_0$ be the distance between 0 and $\partial G_0$.
We will construct inductively
\begin{enumerate}
\item a sequence $(a_d)_{d\geq 2}$ of positive real numbers decreasing to 0 arbitrarily fast  with $a_2=1$;
\item a strictly  increasing sequence of integers $(q_d)_{d\geq 2}$ with $q_2=1$;
\item  for all $d\geq 2$, a sequence of radii $(r_q^{(d)})_{q\geq 1}$ strictly increasing to $+\infty$ defining two sequences of annuli $(C_q^{(d)})_{q\geq 1}$ and  $(\bC_q^{(d)})_{q\geq 1}$ with pairwise disjoint closures, where  $(r_q^{(2)})=(r_q)$  in the previous lemma;
\end{enumerate}
such that for all $d\geq 3$ we have that
$$
r_q^{(d)}=r_q^{(d-1)}\quad \forall q\leq q_d,
$$
and
$$
r_{q+1}^{(d)}=a_d\left(r_q^{(d)}\right)^d,\quad \forall q\geq q_d,
$$
and such that  for all $d\geq 3$ the polynomial
$$
p_d(z)=a_0+z^2+a_3z^3+\dots +a_dz^d
$$
satisfies the following properties, denoted $P(d)$:
\begin{itemize}
    \item[(i)] the maps $p_d|_{C_0^+}$ and $p_d|_{C_0^-}$ both cover $1$ time the disk $\DD_{10Mr_1}$, and  $p_d$ maps $G_0=\DD_{r_1/10}\setminus\overline C_0$ compactly inside $\bG_1^{(d)}$;
    \item[(ii)] for  all $j=2,\ldots, d-1$  and for   $q_{j}\leq  q< q_{j+1}$ the map  $p_d|_{C^{(d)}_{q}}$ covers $j$ times the annulus $\bC^{(d)}_{q+1}$;
 \item[(iii)]  for $q\geq q_{d}$  the map  $p_d|_{C^{(d)}_{q}}$ covers $d$ times the annulus $\bC^{(d)}_{q+1}$;
 \item[(iv)] for all $j=3,\dots, d$
 the map  $p_d|_{A_j}$  covers once the disk $\DD_{10Mr_{q_j}}$;
  \item[(v)] for all $q\geq 1$, the polynomial $p_d$  maps the gap
 $G^{(d)}_q$ compactly inside the gap $\bG^{(d)}_{q+1};$
 \item[(vi)] $p_d$ has $d-1$ simple critical points, each contained in a gap:
 $$
 x_1^{(d)}=0\in G^{(d)}_{0},x_2^{(d)}\in G^{(d)}_{q_3-1},\dots, x_{d-1}^{(d)}\in G^{(d)}_{q_{d}-1},
 $$
 and for all $2\leq j\leq d-2$,
 \begin{equation}\label{criticmoves}
 |x_j^{(d)}-x_{j}^{(d-1)}|\leq 2^{-d}{\rm dist}(x_j^{(j+1)},\partial G^{(d)}_{q_{j+1}-1}).
 \end{equation}
\end{itemize}

Assume we constructed  $a_d$, $q_d$, and  $(r_q^{(d)})_q$ up to an integer $d\geq 3$ (the case from $2$ to $3$ needs to be treated separately but  is analogous). Notice that we only have to choose $a_{d+1}$ and $q_{d+1}$, as the sequence $(r_q^{(d+1)})_{q\geq 1}$ will be then uniquely determined.
Set
\begin{equation}\label{eq:relation}
\rho_{d+1}:=\frac{10\cdot 2d}{2d+1}\frac{a_d}{a_{d+1}}
\end{equation}
and we think of it as a quantity depending on $a_{d+1}$ that has yet to be determined.
Denote
$$
A_{d+1}:=\{\rho_{d+1}/10<|z|<10\rho_{d+1}\}.
$$

Let $\bar q\geq q_d$ be such  that  $r^{(d)}_{\bar q}\geq \bar x$, where $\bar x$ is the constant given by Lemma \ref{lem:covering properties middle piece}.
By Remark \ref{rem:stabilityofcoverings} there exists $\varepsilon>0$ such that for all $a_{d+1}<\varepsilon$ the following properties $P(d+1)$  hold:
\begin{itemize}
\item  (i) is satisfied;
\item (ii) is satisfied up to $\bar q-1$;
\item (iv) is satisfied  for $j=2,\dots,d$;
\item (v) is satisfied  up to $\bar q-1$;
\end{itemize}
By the Implicit Function Theorem, since the critical points $(x_i^{(d)})$ of $p_d$ are simple, for small $a_{d+1}$  the critical points  $(x_i^{(d+1)})_{2\leq i\leq d-1}$ of $p_{d+1}$  can be written as analytic functions of the parameter $a_{d+1}$,  tending to the points $(x_i^{(d)})$ as $a_{d+1}\ra0$. Thus up to taking a smaller $\epsilon$, \eqref{criticmoves} is easily satisfied, and by Lemma \ref{lem:new critical points}  the new critical point
$x_d^{(d+1)}$ is contained in $G^{(d+1)}_{q_{d+1}-1}$, and thus (vi) is satisfied.

Up to taking a smaller $\epsilon$, we can also assume that $\epsilon\leq (d+1)^{-d-1}$,
\begin{equation}\label{disjointannuli}
\varepsilon^{d-1}10^2M^2<\left(\frac{a_d 10\cdot 2d}{2d+1}   \right)^d,
\end{equation}
and that
for all $a_{d+1}<\epsilon$ we have
\begin{equation}\label{nthroot}
\frac{\rho_{d+1}}{10^3}\geq \sqrt[d]{\frac{\rho_{d+1}}{a_{d}}},
\end{equation}
\begin{equation}\label{ratiolarge}
\frac{a_d\rho_{d+1}^{d-1}}{M10^{d+1}}\geq4d+2.
\end{equation}

Let $q_{d+1}> \bar q $ be such that we have
$$
\rho_{d+1}:=r^{(d)}_{q_{d+1}}\geq \frac{10\cdot 2d}{2d+1}\frac{a_d}{\varepsilon}.$$
Then define
$$
a_{d+1}= \frac{10\cdot 2d}{2d+1}\frac{a_d}{\rho_{d+1}}\leq \varepsilon.
$$

Now that $q_{d+1}$ is defined we can prove the remaining properties $P(d+1)$.
First of all, notice that by \eqref{disjointannuli} the sequence of  new radii $(r_q^{(d+1)})_{q\geq q_{d+1}}$ is strictly increasing and the new annuli $(\bC^{(d+1)}_q)_{q\geq q_{d+1}}$ have  pairwise disjoint closures.

 Properties (ii) and (v) for all $\bar q\leq q\leq q_{d+1}$  follow immediately from Lemma \ref{lem:covering properties middle piece} using \eqref{nthroot}, which is telling us that for every $q<q_{d+1}$ we have
$$
r_{q}\leq \frac{\rho_{d+1}}{10^3}.
$$
We proceed to prove (iii), and note that (v) for $q\ge q_{d+1}$ follows directly. To simplify notations, we now write $r_q:=r_{q}^{(d+1)}.$
 Fix $q\geq q_{d+1}$ and write $\gamma$ and $\Gamma$  for respectively  the inner and outer boundary  of the annulus
$C_{q}^{(d+1)}$.
We will show that the curve $p_{d+1}\circ \Gamma$ is contained in the complement of the disk of radius $10 M r_{q+1}$ and has winding number $d+1$ around any point in such disk.
Clearly the image of $\Gamma$  via $a_{d+1}z^{d+1}$ is a circle of radius $a_{d+1}10^{d+1}r_q^{d+1}$.
The ratio
$$
\frac{10 M r_{q+1}}{a_{d+1}10^{d+1}r_q^{d+1}}=\frac{M}{10^d}
$$
is small, thus the claim follows from \eqref{stimacruciale} and Rouch\'e's Theorem.
To obtain (iii), it will now be enough to show that the curve $p_{d+1}\circ \gamma$ is contained in
the disk of radius $\frac{a_{d+1}r_q^{d+1}}{10M}$, which is
 the bounded connected  component of the complement of the closure of $\bC^{(d+1)}_{q+1}$.
Clearly the image of $\gamma$ via $a_{d+1}z^{d+1}$ is a circle of radius $a_{d+1}\frac{r_q^{d+1}}{10^{d+1}}$. The ratio
$$
\left(\frac{a_{d+1}r_q^{d+1}}{10M}\right)/\left(a_{d+1}\frac{r_q^{d+1}}{10^{d+1}}\right) = \frac{10^d}{M}
$$
is large, thus the claim
 follows from \eqref{stimacruciale2}. This proves (iii).

We now show  (iv) for $j=d+1$. Write again $\gamma$ and $\Gamma$  respectively for the inner and outer boundary of the special annulus
$A_{d+1}$. It follows from the above  that the curve $p_{d+1}\circ \Gamma$ is contained in the complement of the disk $\overline{\D}_{10M\rho_{d+1}}$ and has winding number $d+1$ around any point in such disk.
We now show that the curve $p_{d+1}\circ \gamma$ is also contained in the complement of the disk $\overline{\D}_{10M\rho_{n+1}}$, but has winding number $d$ around any point in such disk.
Clearly the image of $\gamma$ via $a_{d}z^{d}$ is a circle of radius $a_{d} \rho_{d+1}^{d}/10^{d}$.
By \eqref{ratiolarge} we have
$$
\frac{a_{d}\frac{\rho_{d+1}^{d}}{10^{d}}}{10M\rho_{d+1}}\geq 4d+2,
$$
hence the claim follows from \eqref{stimacruciale3} and Rouch\'e's Theorem.

This ends the inductive construction. To conclude the proof of Theorem \ref{thm:Construction Theorem 1D} we define
$$
f(z) = a_0 + z^2 + a_3 z^3 + \ldots.
$$
Since for all $d\geq 2$ the inductive construction implies that $a_d<d^{-d}$, the function $f$ is entire. For all fixed $q\geq 1$ the sequence $(r_q^{(d)})_{d\geq 2}$ is eventually constant and defines the radius $r_q$.
 Points (i) to (iv) of Theorem \ref{thm:Construction Theorem 1D} follow from points (i) to (v) of the induction passing to the limit up to choosing a smaller $4<M<6$. Thanks to \eqref{criticmoves}, for all $j\geq 1$ the sequence $(x_j^{(d)})$ converges to a point in the gap $G_{q_{j+1}-1},$ which is a critical point for $f$. By Hurwitz's Theorem the map $f$ has no other critical point.
\end{proof}

The following property  can be ensured during the previous construction, and will be needed in order to embed the Markov process to $\C$.

\begin{lemma}[Expansion on annuli]\label{lem:large derivative on annuli}
    The coefficients $(a_d)$ can be chosen such that for all $q \geq 0$, the  derivative of $f$ has modulus at least $2$ on all of $C_q$.
\end{lemma}
    \begin{proof}
    For all $q\geq 1$ denote $\gamma_q$ and $\Gamma_q$ for the inner and outer boundary of $C_q$ respectively.
    We show by induction that for all $d\geq 3$ we have that  $|p_d'(z)|>2$ for all $$z\in \overline{C_0}\cup\dots\cup\overline{C_{q_{d}-1}}\cup \{|z|\geq \rho_d/10\}.$$
     Assume that this holds for $d\geq 3$.
    Notice that there are no critical points of $p_{d+1}$ in the set 
    $$\{\rho_d/10\leq |z|\leq10  r_{{q_{d+1}}-1}\}\cup \{|z|\geq \rho_{d+1}/10\}.$$
    Hence by the Minimum principle the result follows if we show that, up to choosing $a_{d+1}$ small enough, we can ensure that
     $|p_{d+1}'|>2$ on the circles $\gamma_{q_d}, \Gamma_{q_{d+1}-1}, \gamma_{q_{d+1}}$.
    This is obvious for $\gamma_{q_d}$ since $p_{d+1}$ is an arbitrarily small perturbation of $p_d$ on the compact set $\gamma_{q_d}$.
    We use the same notations as in Lemma \ref{lem:new critical points}. 
For all $z\in \gamma_{q_{d+1}}$ we have 
$$
|g(z)|\geq |g(x_+)|=(\rho_{d+1}/10)^{d-1}\frac{da_d}{2d+1},
$$ 
and
$$
|(p_{d+1}'-g)(z)|\leq 2 a_0|z|^{d-2}\leq 2 a_0(\rho_{d+1}/10)^{d-2},
$$
where  we used the fact that  ${\max_{2\leq j\leq d-1}}ja_j<a_0$ provided $a_j<\frac{1}{j}$. Hence if $a_{d+1}$ is small enough
$$|p'_{d+1}(z)|{\violet\geq }|g(z)|-|(p_{d+1}'-g)(z)|>2.$$

Analogously,  for all $z\in \Gamma_{q_{d+1}-1}$, denoting $s:=r_{{q_{d+1}}-1}=\sqrt[d]{\frac{10\cdot 2d}{(2d+1)a_{d+1}}}$,
$$|g(z)|\geq (d+1)a_{d+1}s^{d-1}(-s-x_0)\simeq (a_{d+1})^{\frac{1-d}{d}},$$  and 
$$
|(p_{d+1}'-g)(z)|\leq 2{ a_0}|z|^{d-2}\leq 2{ a_0}s^{d-2}\simeq (a_{d+1})^{\frac{2-d}{d}}.
$$

 The  case $d=3$ needs to be treated separately (with a similar proof), noticing that   the polynomial    $p_2$ defined in  Lemma~\ref{lem:seed} satisfies
 $$|p_2'(z)|\geq 2\sqrt{5Mr_1},\quad \forall \,z \in \overline C_0\cup \{|z|\geq r_1/10\},$$ which can be made arbitrarily large if $r_1$ is chosen large enough.
 \end{proof}

\subsection{Partitioning of annuli into cells}\label{sec:Bishop cuts}
In order to show that the  modified Markov process described in Remark \ref{modificationbishop} embeds in the dynamical system $(\C,f)$, we will now subdivide each annulus $C_q$ into finitely many cells. The set of all cells is naturally identified with the state space $S$ of the  Markov process, and given two cells $C$ and $C'$, the map $f|_C$ covers once $C'$ if and only if the transition probability between the corresponding states $s$ and $s'$ is nonzero. 

Assume that the coefficients  $a_n$ have all been determined so that  $f$ is an entire function satisfying Theorem~\ref{thm:Construction Theorem 1D} and Lemma~\ref{lem:large derivative on annuli}.
Recall that for $d \ge 2$ and $q_d \leq  q < q_{d+1}$ the restriction of the map $f$ to $C_q$ covers the next annulus $C_{q+1}$ exactly $d$ times. 

\begin{definition}
By a \emph{cut} in $\overline C_q$ we mean a 1-dimensional embedded submanifold with boundary $\sigma$ contained in $\overline C_q$, diffeomorphic to $[0,1]$, such that 
$$
\sigma\cap \partial C_q=\partial \sigma.
$$
  The \emph{main cut} $\Sigma_q \subset \overline C_q$ is defined as $\overline C_q \cap \mathbb R_+$.
\end{definition}

Our goal in this section is to remove from each annulus $C_q$ with $q_d \le q < q_{d+s}$ exactly $d$ disjoint cuts, each connecting the inner and outer boundaries of $C_q$, obtaining $d$ open subsets of $C_q$ which we will call \emph{cells}, and which will be denoted by  $\mathrm{Cell}_q(i)$ for $i = 1, \ldots d$. This will be done  in such a way  that the restriction of $f$ to each cell $\mathrm{Cell}_q(i)$ covers each cell $\mathrm{Cell}_{q+1}(j)$ exactly once.
We first define the cells for non-special $q$.

\begin{remark}
We note that $\Sigma_{q+1} \subset f(\Sigma_q)$ since $f$ has positive coefficients. This is the main motivation for working with positive coefficients.
\end{remark}
Let $q$ be non-special, that is  either $q=1$ (and in this case we set $d=2$)   or 
   $q_d < q < q_{d+1}$ where $d \ge 2$.
    Let $z \in  {\bf C}_{q+1}\cap \R_+$.  Since there are no critical points in $C_q$, the point $z$ has exactly $d$ preimages in $C_q$, denoted $z_1, \ldots, z_d$. 
   For all $j=1\dots  d$, let $\sigma_j$ be the connected component containing $z_j$ of the preimage $\overline C_q \cap f^{-1}(\mathbb R_+)$. Notice that one of these connected components is the main cut $\Sigma_q$.
    Since $\overline C_q$ does not contain any critical points, $f|_{\overline C_q}$ is a local biholomorphism, and thus  $\sigma_j$ is a  cut in $ \overline C_q$.

   For all $j=1\dots  d$, the restriction of $f$ to $\sigma_j$ must be injective, hence there are exactly $d$ such curves.   Since the inner component of the boundary of $C_q$ is mapped to points of norm smaller than $|z|$, and the outer boundary of $C_q$ is mapped to points of larger norm, it follows that $\sigma_j$ must connect the inner and  outer boundaries of $C_q$.

\begin{definition}
The connected components of the complement in $C_q$ of the $d$ cuts will be denoted by $\mathrm{Cell}_q(1), \ldots, \mathrm{Cell}_q(d)$, i.e.
$$
C_q \setminus \bigcup_{i=1}^d \sigma_i = \bigcup_{i=1}^d \mathrm{Cell}_q(i).
$$
\end{definition}

\begin{lemma}
Let $q\geq 1$ be non-special  and let $1\leq i \le d$. The restriction of $f$ to $\mathrm{Cell}_q(i)$ covers the domain ${\bf C}_{q+1} \setminus \R_+$ once. In particular every   $\mathrm{Cell}_{q+1}(j)$ with $1\leq j\leq d$ is covered once.
\end{lemma}
\begin{proof}
Let $m_i$ be the winding number of the curve $f(\partial \mathrm{Cell}_q(i))$ around the domain ${\bf C}_{q+1} \setminus \R_+$. 
Since every  point $w \in {\bf C}_{q+1}$ has exactly $d$ preimages in $C_q$, it is enough to show that no $m_i$ can be 0, that is, at least a point in $ \mathrm{Cell}_q(i)$ is mapped in ${\bf C}_{q+1} \setminus \R_+$.
But this is clear since every cut contains a point whose image is in  ${\bf C}_{q+1}\cap \R_+$ and $f$ is an open map.

\end{proof}
Let us now deal with the special annuli $A_d = C_{q_d}$, $d\geq 3$. 
\begin{remark}
Recall that   $f|_{A_d}$ covers once the disk $D_d$ whose boundary is given by the outer boundary of the annulus  ${\bf C}_{q_d}$.
 Let us write $\Delta_d$ for the inverse image of $D_d$ in the annulus $A_d$. Notice that $\Delta_d$ necessarily intersects $\R_{-}$. Indeed, by the Argument principle we know that $A_d$ contains a unique zero $w_d$, and since the Taylor coefficients of $f$ are real nonnegative, we have $w_d\in \R_{-}$.
 \end{remark}

While the situation is somewhat similar to the non-special case, it is now necessary to distinguish between even and odd values of $d$. The difference between the odd and even degrees is illustrated in  Figure \ref{fig: annuli}. 

Consider an arbitrary point $z \in {\bf C}_{{q_d}+1}\cap \R_+$. Then $z$ has exactly $d$ preimages $z_1, \ldots, z_d$ in $A_d$, each lying on a different connected component  $\sigma_j$ of $\overline A_d \cap f^{-1}(\mathbb R_+)$. 
Notice that $f$ restricted to $\overline{A_d}\cap \R_-$ is real valued and strictly monotone, hence $\overline{A_d}\cap \R_-$ contains at most one preimage of the point $z$.
Since one preimage is necessarily contained in $\overline{A_d}\cap \R_+$, and the non-real preimages appear in conjugate pairs, it follows that if $d$ is odd there is no preimage of $z$ in  $\overline{A_d}\cap \R_-$, while if $d$ is even there is one, let us say $z_1$.
As an immediate consequence we see that if $d$ is odd then $f$ restricted to $\overline{A_d}\cap \R_-$ is strictly increasing, while if $d$ is even it is strictly decreasing. Notice that if $d$ is even the connected component  $\sigma_1$ containing $z_1$ is  the interval $[-10\rho_d, w_1]$.

\begin{lemma}
Each connected component $\sigma_j$ (except for $\sigma_1$ in the even case) is a cut connecting the inner and outer boundaries of $A_d$ which does not intersect  the sets $ \R_-$ and $ \overline{\Delta_d}$.
\end{lemma}
\begin{proof}
We only show that $\sigma_j\cap \overline{\Delta_d}=\emptyset$ as the other claims are straightforward. If $\sigma_j$ intersects $\partial \Delta_j$ in a point $u$, then  $f(u)\in \R$, and thus $f(\bar u)=f(u)$. Since $f|{\partial \Delta_j}$ is injective, it follows that $u\in \R_-$, which is impossible.
\end{proof}

\begin{definition}[Odd case]
    For $d$ odd the complement
    $$
    A_d \setminus \bigcup_{i=1}^{d} \sigma_i
    $$
    consists of $d$ connected components, denoted by $\mathrm{Cell}_{q_d}(1), \ldots, \mathrm{Cell}_{q_d}(d)$. We let $\mathrm{Cell}_{q_d}(1)$ be the unique component  intersecting $\R_-\cup\overline{\Delta_d}$.
\end{definition}

 \begin{figure}[h]
 \begin{subfigure}[h]{0.48\linewidth}
 \includegraphics[width=\linewidth]{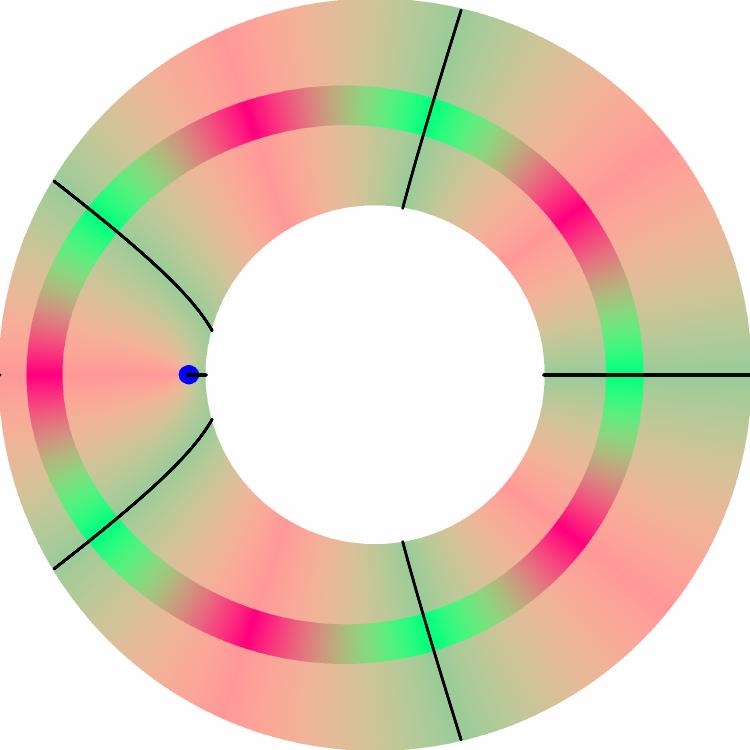}
 \caption{The annulus $A_5$}
 \end{subfigure}
 \hfill
 \begin{subfigure}[h]{0.48\linewidth}
 \includegraphics[width=\linewidth]{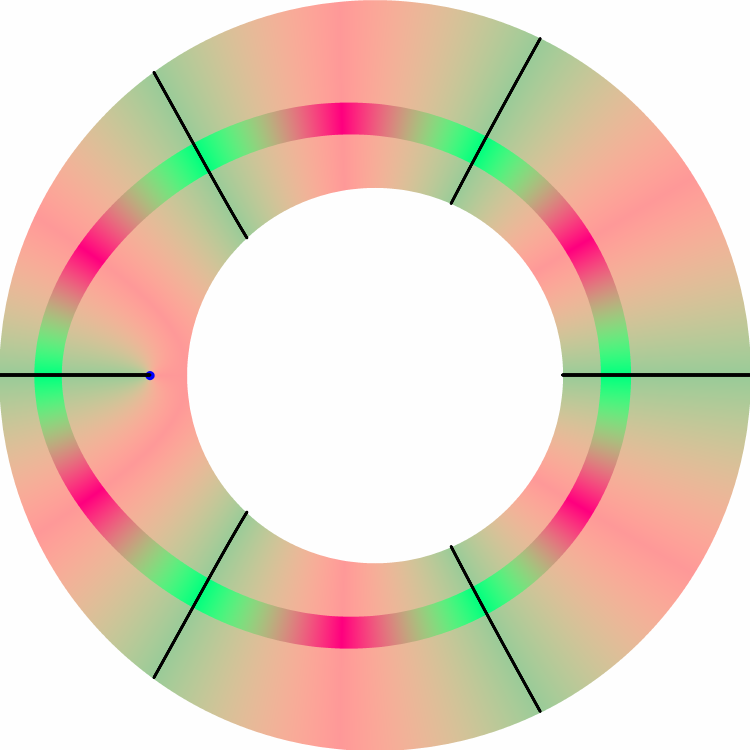}
 \caption{The annulus $A_6$}
 \end{subfigure}
 \caption{Illustration of the difference between the cuts in the annulus $A_n$ for $n$ even or odd. Working with the actual maps constructed in this section is not realistic, as the coefficients decrease too quickly and the annuli grow too quickly for an accurate computation to be feasible. The figure on the left is drawn by computing the image under the map $p_5:z\mapsto z^4+z^5/10$ of the annulus $A(0,9,20)$. The black curves are the preimages in the annulus of the positive real axis. The small blue disk that intersects the negative real axis is given by points for which $|p_5(z)|< 500$. For other points the green and red hues indicates the argument of the image. Points with clear colors are those whose images land inside the annulus $A(0,50000,100000)$. The illustration on the right is drawn similarly, by computing the images of points in the annulus $A(10,20)$ for the map $p_6:z\mapsto z^5 + z^6/12$. Points with clear colors are those whose images land in the annulus $A(0,500000, 1000000)$.}
\label{fig: annuli}
\end{figure}

When $d$ is even we  replace the connected component $\sigma_1$ (which is not a cut), with a cut $\widetilde \sigma_1$ which consists of three pieces. In order to define this cut, observe that $\mathbb R_-$ intersects $\partial \Delta_d$ in exactly two points, say $a < b < 0$. The cut $\widetilde \sigma_1$ consists of the closed interval on the negative real axis that connects the outer boundary of $A_d$ to $a$, a curve contained in the boundary of $\Delta_d$ that connects $a$ to $b$, and a closed interval in $\mathbb R_-$ that connects $b$ to the inner boundary of $A_d$.

\begin{definition}[Even case]
 For $d$ even the complement
    $$
    A_d \setminus \left(\widetilde \sigma_1 \cup \bigcup_{i=2}^{d} \sigma_i \right)
    $$
    consists of $d$ connected components, denoted by $\mathrm{Cell}_{q_d}(1), \ldots, \mathrm{Cell}_{q_d}(d)$. 
    We let $\mathrm{Cell}_{q_d}(1)$ be the component that contains $\Delta_d$.
\end{definition}

For both even and odd $d$ we obtain the following conclusion:

\begin{lemma}
    For all $d \ge 3$ the restriction of $f$ to $\mathrm{Cell}_{q_d}(i)$  covers the domain ${\bf C}_{q_d+1} \setminus \Gamma_{q_d+1}$ once.
     In particular every  $\mathrm{Cell}_{q_d+1}(j)$ with $1\leq j\leq d$ is covered once.
     Moreover, the  restriction of $f$ to $\mathrm{Cell}_{q_d}(1)$ covers every  $\mathrm{Cell}_{q}(j)$ with $q\leq q_d$ 
     once.
\end{lemma}
 Finally, define $\mathrm{Cell}_{0}(1)=C_0^-$ and $\mathrm{Cell}_{0}(2)=C_0^+$.

\subsection{Embedding of the symbolic dynamical system,  and proof of Theorems~\ref{Theorem: Main1} and \ref{Theorem: Main2} for Baker--Bishop maps.}\label{sec:Bishop measure}
In this subsection we will prove the following result.
\begin{theorem}\label{thm:invariant measures BB}
    There exist infinitely many distinct invariant ergodic measures supported on $J(f)$, and for which the entropy of $f$ equals infinity.
\end{theorem}
 We then discuss equidistribution of preimages and periodic points in order to conclude the proof of Theorems~\ref{Theorem: Main1} and \ref{Theorem: Main2} for Baker--Bishop maps.
Let $\QQ$ be the set of  special integers $(q_d)_{d\geq 3}$ given by Theorem~\ref{thm:Construction Theorem 1D}.
Let $S$ be the state space of the model Markov process described in Section \ref{sec:Markov}. Modify the Markov process as shown  in Remark \ref{modificationbishop}, that is, set
$$
E'=E\cup \{(s,t) \text{ such that $s=(q_d,1)$ for some $d\geq2$ and $t=(i,j)$ with $i\leq q_d$}\}.
$$
 Let 
$\mu$ be the unique stationary measure on $S$ and let $\mathbb{P}^\mu$ be the associated $\sigma$-invariant probability measure on $S^\N$.
We call an element $(\theta_n)$ of $S^\N$ an \emph{itinerary}.
\begin{remark}\label{allowableitineraries}
Recall that the support of the measure $\mathbb{P}^\mu$ coincides with the subset $\Theta$ of allowable itineraries. 
It is easy to see that  $(\theta_n)=(i_n,j_n)\in \Theta$ if and only if for all $n\geq 0$  one of the following holds:
    \begin{itemize}
        \item[(i)] both $i_n = i_{n+1} = 0$;
        \item[(ii)] $i_{n+1} =  i_n +1$;
        \item[(iii)] $i_n \in \QQ$, $j_n = 1$ and $i_{n+1} \le i_n$.
    \end{itemize}
    \end{remark}
    We denote by $\Theta_0$ the subset of  non-escaping allowable itineraries. Notice that $\Theta_0$  is completely invariant for the map $\sigma|_\Theta$, and that by Remark \ref{remark: full measure} we have $\mathbb{P}^\mu(\Theta_0)=1$.
    \begin{definition}
    Given  an itinerary $(\theta_n)=(i_n,j_n)\in \Theta_0$, 
     we say that  $z \in \C$ \emph{follows the itinerary} $(\theta_n)$ if  $$f^n(z)\in \mathrm{Cell}_{i_n}(j_n)$$ for all $n\in\N$.
    \end{definition}

\begin{lemma}\label{lemma:shrinking}
    Let $(\theta_n) = (i_n, j_n) \in \Theta_0$. Then there exists a unique point $z \in \mathbb C$ that follows the itinerary $(\theta_n)$.
\end{lemma}
    \begin{proof}
        We consider the nested sequence of open sets $U_0 \supset U_1 \supset \cdots$ given by
        $$
        U_m = \bigcap_{n=0}^m f^{-n} (\mathrm{Cell}_{i_n}(j_n)),
        $$
        for all $m\geq 0$.
      We now show that $\cap_{m\geq 0}U_m$ is not empty.
        Since the itinerary $(\theta_n)$ is non-escaping, it must return to a given cell infinitely often. Without loss of generality we may assume that this 
        cell is  $\mathrm{Cell}_{i_0}(j_0)$. Let thus $(n_s)$ be an increasing sequence of integers such that
        $$\mathrm{Cell}_{i_{n_s}}(j_{n_s})=\mathrm{Cell}_{i_0}(j_0),\quad \forall s\geq 0.$$
          It follows that for every $s\geq 0$, at least one of item (i) or (iii) in Remark \ref{allowableitineraries} occurs at a time $n_s\leq n< n_{s+1}$.
Then 
$$ f(\partial \mathrm{Cell}_{i_n}(j_n))\cap \overline{\mathrm{Cell}_{i_{n+1}}(j_{n+1}))}=\varnothing,$$
           which implies that 
        $$
        \bigcap_{\ell=n_s}^{n_{s+1}} f^{n_s - \ell} (\mathrm{Cell}_{i_{\ell}}(j_\ell))
        $$
        is relatively compact in $\mathrm{Cell}_{i_{n_s}}(j_{n_s})$, hence $\cap_{m\geq 0}U_m\neq\emptyset$.
        By Lemma~\ref{lem:large derivative on annuli} we know that $|f^\prime|>2$ on each cell, which implies that 
        the  diameter of $U_m$ goes to 0, and thus $\cap_{m\geq 0}U_m$ is a point.
    \end{proof}
   
\begin{definition}
Set $$
\Omega = \{z_0 \in \mathbb C :f^n(z_0) \in \bigcup_{q} C_q \quad \forall n \in \mathbb N \},
$$
and let  $\Omega_0$ denote the subset of  points in $\Omega$ which are non-escaping.
Let $\varphi\colon \Theta_0\to \C$ be the map sending  $(\theta_n)\in \Theta_0$ to the unique point $z\in \C$ which follows the itinerary $(\theta_n)$.
\end{definition}

\begin{prop}\label{cor:homeo}
   The map  $\phi$ is a  homeomorphism between $\Theta_0$ and  $\Omega_0$.
\end{prop}

\begin{proof}
It is clear that the map $\varphi$ is injective. 
We now show that its image is $\Omega_0$. Let $z_0\in \Omega_0$.
Since the orbit of every point on the positive real axis diverges, it follows that the orbit of a non-escaping point can never land on the positive real axis.
If a point in a cut is not mapped to the positive real axis (which can happen only for cuts in special annuli $A_d$ with $d$ even), then it is mapped to the outer boundary of the annulus ${\bf C}_{q_d}$, which is contained in the gap $G_{q_d}$.
Hence for all $n\geq 0$ the point  $z_n:=f^n(z_0)$ belongs to a cell, which we denote
$\mathrm{Cell}_{i_{n}}(j_n)$, and thus $z_0\in \varphi(\Theta_0)$.
  The map $\varphi$ is thus  a bijection between  $\Theta_0$ and $\Omega_0$. It remains to be checked that both $\varphi$ and $\varphi^{-1}$ are continuous.

 Let  $(\theta^{(k)}_n)_k=(i^{(k)}_n,j^{(k)}_n)$ be a sequence of itineraries  converging as $k\to \infty$ to an itinerary $(\theta_n)=(i_n,j_n)$ in the product topology. This means that the itineraries stabilize:
 for every $N\geq 0$ there exists $K\geq 0$ such that if $k\geq K$ then  $\theta^{(k)}_n=\theta_k$ for all $n\leq N$.
 As discussed in the proof of Lemma \ref{lemma:shrinking}, the nested sequence of finite pullbacks of cells has diameter converging to 0. It follows that, as $k\to\infty$,  we have  $\varphi(\theta^{(k)}_n)\to \varphi(\theta_n)$, proving that $\varphi$ is continuous.

    On the other hand, suppose that $z_0^{(k)}$ is a sequence of points in $\Omega_0$ converging as $k\to\infty$ to a  point $z_0 \in \Omega_0$.  Let  $(\theta_n)=(i_n, j_n)\in \Theta_0$ be the itinerary followed by $z_0$,  and for fixed $m$ consider the open neighborhood of $z_0$
    $$
    U_m = \bigcap_{n=0}^m f^{-n} (\mathrm{Cell}_{i_n}(j_n)).
    $$
    For large  $k$ the point $z_0^{(k)}$ must also lie in $U_m$. It follows that the  itineraries followed by the  points $z_0^{(k)}$ stabilize, and thus converge to $(\theta_n)$ in the product topology. 
    \end{proof}

As a consequence, the symbolic dynamical system  $(S^\N, \sigma)$ embeds in  the complex dynamical system $(\mathbb C,f)$ via the map $\varphi$.
Hence the measure $\mathbb{P}^\mu$ can be pushed forward to an ergodic invariant measure $\widetilde \mu:=\varphi_*\mathbb{P}^\mu$ on $\mathbb C$  with infinite entropy, whose support equals $\overline{\Omega_0}$.

\begin{prop}\label{prop:Julia}
    The Julia set $J(f)$ equals the set of points whose forward images lie in the union of the cells. In particular
    $$
    J(f) = \Omega = \overline{\Omega_0}.
    $$
\end{prop}
\begin{proof}
    Since $f^n\rightarrow \infty$ on the complement of the cells, uniformly on compact subsets, it follows that $J(f)$ must be contained in $\Omega$. 
     If $z\in \Omega$, then $|(f^n)'(z)|\to \infty$. Hence if $z$ is contained in the Fatou set, it follows that $f$ diverges uniformly on a neighborhood of $z$, but this is not possible since
by Lemma \ref{lemma:shrinking} there exists a point $w \in \Omega_0$ in any neighborhood of  $z$.
\end{proof}

 Theorem~\ref{thm:invariant measures BB} is a direct consequence of Proposition~\ref{cor:homeo} and Proposition~\ref{prop:Julia}

\begin{remark}
    We recall that there was a lot of freedom in the construction of the Markov process of infinite entropy. By changing finitely many non-zero transition probabilities slightly, the ergodicity of the process can be preserved, and the entropy can be guaranteed to remain infinite. Since each change in the collection of transition probabilities induces a different invariant measure on the space of sequences, we can obtain infinitely many different examples of measures on $\mathbb C$ with support $J(f)$, all having infinite entropy.
\end{remark}

\begin{remark}\label{lemmameasurepullback}
Let $U\subset {\rm Cell}_{i_1}(j_1)$ be an open subset such that $f|_U$ is injective and $f(U)\subset {\rm Cell}_{i_2}(j_2)$. 
Then, if we denote
 $s:= (i_1, j_1)$ and $t := (i_2, j_2)$, 
 it follows from \eqref{eq:measure of preimages integral} and Corollary \ref{Pmustationary} that  
 $$\widetilde \mu(U)=Q_{ts}\cdot \widetilde \mu(f(U)).$$
\end{remark}
We now discuss equidistribution.
\begin{prop}
Let $w\in \Omega_0$. Let $\theta$ be the itinerary followed by $w$. Then the sequence of probability measures
$$\varphi_*Q^n_\theta=\sum_{z\in f^{-n}(w)}Q^n(\theta,\varphi^{-1}(z))\delta_z$$ converges weakly to $\tilde \mu$ as $n\to\infty$. 

\end{prop}
 \begin{proof}
Let $n\geq 1$. Notice that every itinerary in $\sigma^{-n}(\theta)$ is non-escaping.
The probability measure $Q^n_\theta$ on $S^\N$ defined in Section \ref{sectionequidistribution} is a weighted sum of the measures $\delta_\alpha$ with $\alpha \in \sigma^{-n}(\theta)\cap \Theta_0$.
It follows that the measures $Q^n_\theta$  converge weakly to $ \mathbb{P}^\mu$  also  when  restricted   to the subset $\Theta_0$, and
$\varphi_*Q^n_\theta\to \widetilde \mu$ weakly.
\end{proof}
Similarly one obtains equidistribution  of periodic points.
Let $\nu_n$ be the finite measure on $\N^N$ introduced in Subsection \ref{equiperiod}.
\begin{prop}
The sequence of probability measures
$$\varphi_*\nu_n=\sum_{z: f^n(z)=z}Q^n(\varphi^{-1}(z),\varphi^{-1}(z))\delta_z$$ converges weakly to $\tilde \mu$ as $n\to\infty$. 
\end{prop}

 This  completes the proofs of Theorems \ref{Theorem: Main1} and  \ref{Theorem: Main2} for the subclass of Baker--Bishop maps.


\section{Invariant measures for disjoint-type maps}\label{sec:class B}
In this section we prove Theorems \ref{Theorem: Main1} and \ref{Theorem: Main2} for disjoint-type transcendental functions of finite order.  The setup is classical (\cite{EL92}, \cite{BK07},\cite{RRRS}). 
\subsection{Tracts and fundamental domains}
Let $f\colon \C\to \C$ be a transcendental  function. Let us define the set of singular values   $S(f)$  as  the closure of the set of critical and asymptotic values.
 The restriction
 $f: \C\setminus f^{-1}(S(f))\ra \C\setminus S(f)$ is a covering.

\begin{defn}[Disjoint-type]
A transcendental  function is called of  \emph{disjoint type} if there exists a holomorphic disk $D\Subset \C$ with real-analytic boundary,
 containing  $S(f)$, for which $f(D)\Subset D$.
\end{defn}

For disjoint-type  functions, forward invariance ensures that  $D$  contains the postsingular set  $\ov{\bigcup_{i\in\N}f^i(S(f))}$. We recall (\cite[Lemma 3.1]{BK07}) that $f$ is  of disjoint type if and only if it is  hyperbolic (i.e. its postsingular set is a compact subset of the Fatou set) and the Fatou set is connected. This happens if and only if  $f$ has one attracting fixed point and the set of  singular values is compactly contained in  its attracting basin.
For disjoint-type  functions, it is known  that the Julia set equals the set of points whose iterates never enter $D$ (see\cite[Proposition 3.2]{RempeArclike}).

 The preimage of $\C\setminus\ov {D}$ under $f$ consists of finitely or countably many unbounded holomorphic disks which are called \emph{tracts}. For any such tract $T$, the restriction $f|_T:T\ra \C\setminus \ov D$ is a universal cover, hence equivalent to the exponential map in the following sense: there exist  biholomorphisms  $\phi: T\ra\H  :=\{z\in\C: \re z>0\}$ and $\psi: \mathbb C \setminus \overline{D} \ra \C\setminus \overline \D$ such that $f|_{T}=\psi^{-1}\circ\exp\circ\phi$. For $f$ of disjoint type, the closures of the tracts are  disjoint from $\ov D$. 
By the Denjoy--Carleman--Ahlfors Theorem, disjoint-type functions of finite order have finitely many tracts.
The assumption  that $f$ has finitely many tracts is made to simplify the exposition.  The assumption that $f$ is a finite composition of functions of finite order is technical and simplifies the proofs in Subsection \ref{sec:proof of class B covering of cells}.

\begin{definition}[Fundamental domains and cells]\label{funddom}
 Let $\TT$ be the union of the tracts. Let $\delta\subset \C\setminus (\ov{\TT\cup D})$ be an analytic curve connecting $D$ to infinity. Such a curve is not difficult to find  (see e.g. \cite[Lemma 2.1]{BF15}). Then for every tract $T$ the open set $T\setminus f^{-1}(\delta)$ has
  countably many connected components called \emph{fundamental domains}, and for each fundamental domain  $F$,  the map $f:F\ra \C\setminus( \ov{D}\cup \delta)$ is a biholomorphism.
Let $r\in\R_+$ be sufficiently large such that $\D_r\supset D$, and let $F$ be a fundamental domain. We define the \emph{cell} $C(F,r)$ as
$$
C(F,r):= f^{-1}(\D_r)\cap F.
$$
\end{definition}

Notice that each cell $C(F,r)$ is bounded. The  cell $C(F,r)$ is disconnected if the curve $\delta$ and the circle $C_r$ of radius $r$ intersect in more than one point. However, there are at most finitely many intersections for each $r$, so each cell consists of finitely many connected components. Also, by definition
 $$
f(C(F,r))=\D_r\setminus( \ov{D}\cup \delta).
$$

 An illustration of the tracts, fundamental domains and cells for the classical example $f(z) = \cos(z)/2$ is provided in Figure~\ref{fig:disjoint}. 

\subsection{Logarithmic coordinates}

Let $f\colon \C\to \C$ be of disjoint type.  Let $D\Subset  \C$ be a holomorphic  disk with real analytic boundary containing $S(f)$ such that $f(D)\Subset D$.
Up to  conjugacy  by an affine map
 we can assume that $0\in D$,  that $D$ is contained in  the unit disk  $\D$, and that $1\not\in \overline D\cup \delta$. Let $\widetilde H:={\rm exp}^{-1}(\C\setminus\overline D)$. We have that  ${\rm exp}\colon \widetilde H\to \C\setminus \overline D$ is a universal covering, and that $\widetilde H$ contains  the right half plane $\H$.  
The preimage $f^{-1}(\C\setminus\overline D)$ is the union of the tracts $\bigcup_i T_i$. Every tract is simply connected and does not contain 0.
 Hence for every connected component $\widetilde T$ of the preimage ${\rm exp}^{-1}(f^{-1}(\C\setminus\overline D))$ there exists $i$ such that  $${\rm exp}\colon \widetilde T\to T_i$$ is a biholomorphism. It follows that  $f\circ {\rm exp}\colon \widetilde T\to \C\setminus \overline D$ is a universal covering, and thus there exists a biholomorphism $\widetilde f\colon \widetilde T\to \widetilde H$ which satisfies
 $${\rm exp}\circ \widetilde f=f \circ {\rm exp}.$$ 
In this way we obtain a holomorphic function $\widetilde f$, defined on $\exp^{-1}(f^{-1}(\C\setminus\overline D))$, which we normalize in such a way that 
$$
\widetilde f(z+2\pi i)=\widetilde f(z).
$$ 
We call the function  $\widetilde f$ the \emph{lift} of $f$, and we call every   connected component $\widetilde T$ of its domain ${\rm exp}^{-1}(f^{-1}(\C\setminus\overline D))$   a \emph{tract}  for $\widetilde f$. 
By construction $\widetilde f|_{\widetilde T}\colon \widetilde T\to \widetilde H$ is a biholomorphism, whose inverse  we denote $\varphi_{\widetilde T}\colon \widetilde H\to \widetilde T$.
Notice that every tract $\widetilde T$ is contained in $\widetilde H$, and the closure of each tract in the sphere $\widehat\C$ only intersects the closure of $\widetilde H$ in the sphere in the point $\infty$. 
\begin{definition}[Fundamental domains and cells for $\widetilde f$]
The open set $\widetilde H \setminus\exp^{-1}(\delta)$ has countably many connected components that we call \emph{strips}. Every strip $ S$ is mapped by the exponential map biholomorphically to $\C\setminus (\overline D\cup \delta).$
Given a tract $\widetilde T$ for $\widetilde f$, the open set $$\widetilde T\setminus \widetilde f^{-1}(\exp^{-1}(\delta))=\widetilde T\setminus(\exp^{-1}(f^{-1}(\delta))$$ has countably many connected components called
\emph{fundamental domains} for $\widetilde f$.  
Each fundamental domain $\widetilde F$ for $\widetilde f$ is mapped biholomorphically by $\widetilde f$ to one of the strips, and is mapped  biholomorphically by the exponential map to one of the fundamental domains for $f$.

 For any fundamental domain $\widetilde F$ let $\widetilde T$ be the tract containing $\widetilde F$, let $ S$ be the strip such that $ S=\widetilde f(\widetilde F)$, and  let  $F$ be  the fundamental domain for $f$ such that  $\exp (\widetilde F)=F$.  Let $r\in\R_+$ be sufficiently large such that $\D_r\supset D$. We define the \emph{cell} $C(\widetilde F,r)$ as
$$
 C(\widetilde F,r):= \widetilde F \cap \exp^{-1}(C(F,r)).
$$
Notice that by construction 
\begin{equation}\label{eq:cells as preim}
   C(\widetilde F,r)= \widetilde F \cap  \widetilde f^{-1}(\{\Re z<  \log r\}) = \widetilde T \cap\widetilde f^{-1}(\{\Re z<  \log r\}\cap  S).
\end{equation}
\end{definition}
\begin{remark}[Labeling of fundamental domains]
Denote by $ S_0$ the strip containing $0$, and for all $k\in \Z$ denote $ S_k:=  S_0+k2\pi i$. Let $T$ be a tract for $f$.  For all $k\in \Z$ we denote by $F_k(T)$ the fundamental domain contained in $T$ such that 
$$\widetilde f(\exp^{-1}(F_k(T)))= S_k.$$
Clearly in this way every fundamental domain $F_i(T)$  is adjacent to
$F_{i-1}(T)$  and to $F_{i+1}(T)$. 
\end{remark}


\subsection{Proof of Theorems \ref{Theorem: Main1} and \ref{Theorem: Main2}}

We will need the two following results.
\begin{prop}[Covering properties of cells]\label{prop:class B covering of cells logarithmic version}
 Let $\widetilde f$ be the lift of a function $f$ of disjoint type which is a finite composition of functions of finite order.
There exists $\alpha>0$ such that the following holds.  For every tract $T$ for $f$ there exist $t_{ T}>0$ such that for all $t\geq t_{ T}$,  for all   $|i|\leq \alpha\log t$, and
 for every fundamental domain $\widetilde F$ for $\widetilde f$,      there is a fundamental domain $ \widehat F$ for $\widetilde f$ such that
$ \exp( \widehat F)=F_i(T)$ and 
  $$\widetilde f(C(\widetilde F,t))\Supset C( \widehat F, e^{ t^\alpha} ).$$
\end{prop}

\begin{prop}[Size of cells]\label{lem: size of cells log plane} Let $\widetilde f$ be the lift of a function $f$ of disjoint type which is a finite composition of functions of finite order. Then there exists a   constant $k\geq 0$ such that
    for any fundamental domain $\widetilde F$, and for any constant  $\beta>1$, we have
\begin{equation}\label{eq:class B size cells log plane}
\diam_{\widetilde H} C(\widetilde F, \beta^n )\leq (\log\log \beta )(\log n)+k \text{ for all  $n\in\N$}.
\end{equation}
\end{prop}
Propositions~\ref{prop:class B covering of cells logarithmic version} and  \ref{lem: size of cells log plane} are proved in Subsection~\ref{sec:proof of class B covering of cells}. 
The following result  is implicitly contained in \cite{EL92}, \cite{BK07} and thus we omit the proof; compare with  \cite[Proposition 3.7]{RempeArclike}.

\begin{prop}[Hyperbolic contraction]\label{prop:strict contraction tracts}
    Let $\widetilde f$ be the lift of a function of disjoint type $f$. Then there exists $0<\lambda<1$ such that for any tract $\widetilde T$ and  for any $z,w\in\widetilde H$,
$$\dist_{\widetilde H}(\phi_{\widetilde T}(z),\phi_{\widetilde T}(w))\leq \lambda \dist_{\widetilde H}(z,w).$$
\end{prop}

From now on, to simplify the exposition,  we assume that  there are finitely many tracts $T_1\ldots T_p$. So far fundamental domains have been labeled by $p$ copies of $\Z$. It is convenient to relabel fundamental domains by $\N$,  in the following way.
 Reorder the fundamental domains as 
$$
F_0(T_1)\ldots F_0(T_p)F_1(T_1)\ldots F_1(T_p)F_{-1}(T_1)\ldots F_{-1}(T_p)F_2(T_1)\ldots F_2(T_p)\dots
$$
and label them as  $F(0), F(1), F(2)\ldots$   respecting the new ordering, that is,   $F(0)=F_0(T_1), F(p)=F_0(T_p), F(p+1)=F_1(T_1)\dots $. In this way fundamental domains are now labeled by indexes in $\N$.

\begin{remark}\label{remarkcovercells}
In view of Proposition~\ref{prop:class B covering of cells logarithmic version}, since $f$ has   finitely many tracts, there exists $t_0>0$  such that for all $t\geq t_0$ the following holds.
For every fundamental domain $\widetilde F$ for $\widetilde f$, and any $i\leq \alpha\log t$ there is a fundamental domain $\widehat F$ such that  $ \exp( \widehat F)=F(i)$
and   $$\widetilde f(C(\widetilde F,t))\Supset C( \widehat F, e^{ t^\alpha} ).$$
\end{remark}

We call a sequence  $x=(x_0,x_1,\ldots)\in\N^\N$ an \emph{itinerary}, and we say that a point $z\in \bigcup T_i$ \emph{follows the itinerary} $x$ if
$$
f^n(z)\in F(x_n) \text{ for all $n\in\N$.}
$$
The next proposition is similar in spirit  to Theorem B in \cite{BK07}, however the language used is different and the translation does not seem trivial. The growth condition is also different.
Recall (Definition \ref{def:quadratic growth}) that $\mathcal{Q}\subset \N^\N$ denotes the subset of sequences growing at most quadratically, which is completely invariant by the shift $\sigma\colon \N^\N\to \N^\N.$

\begin{prop}[Realization of itineraries]\label{prop:class B realization of itineraries}
Let $f$ be a  function of disjoint type which is  a finite composition of functions of finite order. Assume  that it has finitely many tracts.
Let $x=(x_0, x_1, \ldots)\in \mathcal{Q}$. Then if $\beta$ is sufficiently large (depending on $x$), there
 exists exactly one point $z_x\in \C$ which satisfies 
\begin{equation}\label{Realization}
f^n(z_x)\in C(F(x_n),\beta^{(n+1)^2}),\quad \forall n\in \N.
\end{equation}
In particular  $z_x$ follows the itinerary $x$.
\end{prop}

\begin{proof}
Let $\beta\geq {\rm max}\{1,t_0\}$ be large enough that
$$x_{n+1}\leq (n+1)^2\alpha\log\beta,\quad \forall\,n\in \N,$$
and
$$\beta^{(n+1)^2\alpha}\geq (n+2)^2\log\beta,\quad \forall\,n\in \N.
$$
Let $\widetilde  F_0$ be a fundamental domain for $\widetilde f$ such that ${\rm exp}(\widetilde F_0)=F(x_0).$
Thanks to Remark \ref{remarkcovercells}, we can choose fundamental domains $(\widetilde F_n)_{n\geq 1}$ for $\widetilde f$ such that 
${\rm exp}(\widetilde F_n)=F(x_n)$ for all $n\geq 1$ and
$$\widetilde f(C(\widetilde F_n,\beta^{(n+1)^2}))\Supset C(\widetilde F_{n+1},\beta^{(n+2)^2}),\quad \forall \, n\geq 0.$$
For all $n\geq 0$, set $$U_n:=\bigcap_{j=0}^n \widetilde f^{-j}(C(\widetilde F_n,\beta^{(n+1)^2})).$$
Then $U_0\Subset U_1\Subset U_2\dots$, and thus $\bigcap_{n\geq 0}U_n$ is not empty.

We now show that  $\bigcap_{n\geq 0}U_n$   contains a unique point. 
For all $n\geq 0$, denote by $\widetilde T_n$ be the tract for $\widetilde f$ containing the fundamental domain $\widetilde F_n$, and denote
$$
\phi_n:=\phi_{\widetilde T_0}\circ\phi_{\widetilde T_1}\ldots\circ \phi_{\widetilde T_{n-1}}.
$$
 Notice that   $\phi_n$ is defined on $\widetilde H$, and that  $$\varphi_n(C(\widetilde F_n,\beta^{(n+1)^2}))=U_n.$$
By Proposition \ref{prop:strict contraction tracts}, $\phi_n$ contracts the Poincar\'e distance on $\widetilde H$ by a factor $\lambda^n$.
Hence, by Proposition \ref{lem: size of cells log plane},
 $$
 \diam_{\widetilde H}\,\phi_n({C(\widetilde F_n,\beta^{(n+1)^2})})\leq \lambda^n(2(\log\log \beta)(\log(n+1))+k)\ra0 \text{ \  \ \ as $n\ra\infty$.}
 $$
\end{proof}

Let $\varphi\colon \mathcal Q\to \C$ be the injective map which sends every itinerary $x\in \mathcal Q$ to the  point $z_x\in \C$ given by the previous proposition.
Its image $\varphi(Q)$ is the subset of $z\in \C$ such that there exists $\beta\geq 0$ and $x\in \mathcal{Q}$ satisfying
$$
f^n(z)\in C(F(x_n),\beta^{(n+1)^2}),\quad\forall\, n\in\N.
$$ 
Notice that $\varphi(Q)$ is completely invariant under $f$. The map $\varphi\colon \mathcal Q\to \C$ satisfies 
\begin{equation}\label{eq: phi is a conjugacy}
\phi\circ\sigma=f\circ \phi \text{ for all $x\in\N^\N$}.   
\end{equation}
We will now use $\phi$ to define a measure $\tilde \mu $ on $\phi(\QQ)$.

\begin{remark}
The inverse map $\varphi^{-1}\colon \varphi(\mathcal{Q})\to Q$ is clearly continuous thanks to the continuity of the maps $f^n$. 
However, the map $\varphi\colon \mathcal Q\to \C$ is not continuous. 
Given a sequence $(t_k)$ of non-negative integers consider the sequences $x^{(k)}\in\QQ$ defined by $x^{(k)}_j=0$ if $j\neq k$,  $x^{(k)}_k=t_k$. Then $x^{(k)}$ converges to the zero sequence $\ov{0}$,
which has  image $\phi(\ov{0})=z_{\ov{0}}\in\C$.  However, since
 every compact set intersects at most finitely many fundamental domains (see Lemma 2.1 in \cite{BeniniRempe}),
 there exists  a sequence $(t_k)$ such that $z_k:=z_{x^{(k)}}\ra\infty$, contradicting continuity.

%
%
%
\end{remark}

For  all $B\in\N$ let  
$$
\QQ_B:=\{x=(x_0,x_1,\ldots)\in \N^\N \text{ such that $x_n\leq B(n+1)^2$ for all $n\in\N$}\}.
$$
Notice that every $\QQ_B$ is closed, that $\QQ_B\subset\QQ_{B+1}$, and that $\QQ=\bigcup_B \QQ_B$.
Since $\mathbb{P}^\mu(\mathcal{Q})=1$ by Proposition~\ref{prop:modify process to prescribe support}, it follows that for $B>0$ large enough, $\mathbb{P}^\mu(\mathcal{Q}_B)>0$.
\begin{lem}\label{lem: phi on QB}

For all $B\in\N$, the map $\phi_{|_{\QQ_B}}$ is continuous.     
\end{lem}
\begin{proof}
Let $B>0$.
Then the  constant $\beta$ in Proposition \ref{prop:class B realization of itineraries} can be chosen to be the same for all $x\in \mathcal{Q}_B$.
In the product topology a sequence of itineraries $(x^{(k)})$ converges to $x^*$ if and only if for every $n\in\N$ there exists $k_n$ such that $x^{(k)}_j =x^*_j$ for all $j\leq n$ whenever $k\geq k_n$. Let $z_k:=z_{x^{(k)}}$, $z_*:=z_{x^*}$.
By the proof of Proposition \ref{prop:class B realization of itineraries} for all $n\geq 0$ there exists $k_n$ such that 
for all $k\geq k_n$ the points $z_k$ and $z_*$ belong to the open set
$$\exp(U_n):=\exp(\cap_{j=0}^n \widetilde f^{-j}(C(\widetilde F_j,\beta^{(j+1)^2}))).$$
The result follows since  $\diam_{\widetilde H}(U_n)\to 0$ as $n\to \infty$.

\end{proof}

It immediately follows that $\varphi\colon \QQ\to \C$ is Borel measurable. 
Since $\varphi^{-1}\colon \varphi(\mathcal{Q})\to Q$ is continuous, it follows that  $\varphi\colon \QQ\to \varphi( \QQ)$ is a Borel isomorphism.
Denote by $\tilde \mu$ the push-forward probability measure $$\tilde \mu:=\varphi_*(\mathbb{P}^\mu).$$
It is clear that $\tilde \mu$ is  $f$-invariant, mixing (hence ergodic), and has infinite entropy. 
\begin{prop}\label{supportcontains}
The support of the measure $\tilde \mu$ is  $\overline{\varphi(\QQ)}=J(f)$.
\end{prop}
\begin{proof}
It is clear that the support of $\tilde \mu$ is contained in $\overline{\varphi(\QQ)}$. 
We now show that it contains $\overline{\varphi(\QQ)}$. Since, by the previous lemma, the map $\phi_{|_{\QQ_B}}$ is continuous for all $B>0$,  it is enough to show that, if $B>0$ is large enough, then for every cylinder set $C=C_{y_0,\dots, y_k}\subset \N^\N$ 
such that $y_i\leq B(i+1)^2$ for all $0\leq i\leq k$,
we have that 
$$
\mathbb{P}^\mu(C\cap \QQ_B)>0.
$$
Let $B>0$ large enough so that $\mathbb{P}^\mu(\QQ_B)>0$.
Define a continuous branch $g\colon \N^\N\to \N^\N$ of $\sigma^{-(k+1)}$ by
$$g(z_0,z_1,\dots)=(y_0,\dots, y_k, z_0, z_1, \dots).$$
Notice that $C\cap \QQ_B\supset g(\mathcal{Q}_B)$, and recall that the support $\Theta$ of $\P^\mu$ is the whole  $\N^\N$.  By \eqref{endsummer}, we have
$$\mathbb{P}^\mu(C\cap \QQ_B)\geq \mathbb{P}^\mu(g(\QQ_B))=\int_{\QQ_B}Q^n(w,g(w))d\mathbb{P}^\mu(w),$$ which is strictly positive since $\mathbb{P}^\mu(\QQ_B)>0$ and $Q^n(w,g(w))>0$ for all $w\in \N^\N$.

 We now show that $\overline{\varphi(\QQ)}=J(f)$.
Since $f$ is of disjoint type, the Julia set consists of the points whose iterates always belong to the tracts, so $\ov{\phi(\QQ)}\subset \ov{J(f)}=J(f)$.
Since $\overline{\varphi(\QQ)}$ is closed, completely invariant, and contains at least three points, it contains the Julia set $J(f)$.  It follows that the support of $\tilde \mu$ equals the Julia set.
\end{proof}

We now prove equidistribution of preimages. 

\begin{lemma}\label{lem:convergence MP}
Let $(P_{st})$ be an ergodic Markov process on a countable state space $S$. Let $\mu$ be its unique stationary measure. Let $G\subset S^\N$ be a Borel subset with $\mathbb{P}^\mu(G)>0$. Let  $\eta_n$ denote the probability measure on $S$ defined by $\eta_n(A):=\mathbb{P}^\mu(x_n\in A \; | \; x \in G)$. Then $\eta_n$ converges to $\mu$ weakly as $n\to\infty$.
\end{lemma}

\begin{proof}
It follows from  \cite{BF} and the Reverse Martingale Convergence Theorem.
\end{proof}

 Recall the definition of $Q^n_x$ from Section~\ref{sec:Equidistribution preimages Markov process}

\begin{lemma}\label{othertopology}
Let $x=(x_0,x_1,\dots)\in \QQ$. Then if $B$ is large enough so that  $\mathbb{P}^\mu(\QQ_B)>0$, and if $C=C_{y_0,\dots, y_k}\subset \N^\N$ is a cylinder set, we have that
\begin{equation}\label{othertopologyeq}
Q^n_x(C\cap\QQ_B)\stackrel{n\to\infty}\longrightarrow \mathbb{P}^\mu(C\cap\QQ_B).
\end{equation}
\end{lemma}
\begin{proof}
If $C\cap\QQ_B=\varnothing$ the statement is trivially true. Assume thus that $C\cap\QQ_B\neq \varnothing$, that is,  $y_i\leq B(i+1)^2$ for all $0\leq i\leq k$.
If $n>k$, then we have
$$Q^n_x(C\cap \QQ_B)=\mathbb{P}^\mu(\alpha\in  C\cap \QQ_B|\alpha_n=x_0).$$
Hence, by Bayes' theorem,
$$Q^n_x(C\cap \QQ_B)=\frac{\mathbb{P}^\mu(\alpha_n=x_0|C\cap \QQ_B)\mathbb{P}^\mu(C\cap \QQ_B)}{\mathbb{P}^\mu(\alpha_n=x_0)}.$$ Notice that the conditional probability $\mathbb{P}^\mu(\alpha_n=x_0|C\cap \QQ_B)$ is well defined since by the proof of Proposition \ref{supportcontains} we have $\mathbb{P}^\mu(C\cap \QQ_B)>0$.
We have that $\mathbb{P}^\mu(\alpha_n=x_0)=\mu(x_0)$, hence it is sufficient to prove that
\begin{equation}\label{storm}
\mathbb{P}^\mu(\alpha_n=x_0|C\cap \QQ_B)\stackrel{n\to\infty}\longrightarrow\mu(x_0).
\end{equation}
The result follows by Lemma \ref{lem:convergence MP} setting $G:=C\cap \QQ_B.$

\end{proof}

\begin{prop}\label{autumn}
Let $w\in \varphi(\QQ)$. Let $x$ be the itinerary followed by $w$. Then the sequence of probability measures
$$\varphi_*Q^n_x=\sum_{z\in f^{-n}(w)}Q^n(x,\varphi^{-1}(z))\delta_z$$ converges weakly to $\tilde \mu$ as $n\to\infty$. 

\end{prop}
\begin{proof}
Let $B$ large enough so that Lemma \ref{othertopology} holds. We claim that the sequence of finite measures
$(Q_x^n|_{\QQ_B})_{n\geq 1}$ converges weakly to the  measure  $\mathbb{P}^\mu|_{\QQ_B}$.
Lemma \ref{othertopology} would immediately imply this if the measures were probability measures, see e.g.  \cite[Theorem 2.2]{Bil99}. However the proof can be adapted to our situation since there is no mass loss in the limit. Indeed, since $\QQ_B$ is closed, and $(Q_x^n)$ converges weakly to 
$\mathbb{P}^\mu$ by Proposition \ref{equiprop}, we have that 
$$\mathbb{P}^\mu(\QQ_B)\geq \limsup_{n\to\infty}Q_x^n(\QQ_B).$$

Since $\varphi|_{\QQ_B}\colon \QQ_B\to \C$ is continuous, it follows that 
$$(\varphi_* Q_x^n)|_{\varphi(\QQ_B)}\stackrel{n\to\infty}\longrightarrow\tilde \mu|_{\varphi(\QQ_B)}$$ weakly. 
We now show that this implies that 
$$
\varphi_*Q^n_x\stackrel{n\to\infty}\longrightarrow\tilde \mu
$$
 weakly. Let $\varepsilon>0$, and let $h\in C_b(\C).$ Let $M>0$ such that $|h(z)|\leq M$ for all $z\in \C$. Let $B>0$ be large enough such that $\tilde \mu(\varphi(\QQ_B))\geq 1-\epsilon/(6M).$ It follows that if $n$ is large enough, we have that 
$$\varphi_* Q_x^n(\varphi(\QQ_B))\geq 1-\frac{\epsilon}{3M}.$$
Thus, if $n$ is large enough,
\begin{align*}\left|\int_\C h d\varphi_* Q_x^n-\int_\C h d\tilde\mu\right|&\leq \left|\int_{\varphi(\QQ_B)} h d\varphi_* Q_x^n-\int_{\varphi(\QQ_B)} h d\tilde\mu\right|+\left|\int_{\C\setminus \varphi(\QQ_B)} h d\varphi_* Q_x^n-\int_{\C\setminus\varphi(\QQ_B)} h d\tilde\mu\right|\\
&\leq \epsilon/2+M(\varphi_* Q_x^n(\C\setminus\varphi(\QQ_B))+\tilde \mu(\C\setminus\varphi(\QQ_B)))\leq \epsilon.
\end{align*}
\end{proof}
Finally, we prove equidistribution of periodic points. Let $\nu_n$ be the finite measure on $\N^N$ introduced in Subsection \ref{equiperiod}.

\begin{lemma}\label{othertopology2}
Let $B>0$ large enough so that  $\mathbb{P}^\mu(\QQ_B)>0$, and let  $C=C_{y_0,\dots, y_k}\subset \N^\N$ be a cylinder set, then
\begin{equation}\label{othertopologyeq}
\nu_n(C\cap\QQ_B)\stackrel{n\to\infty}\longrightarrow \mathbb{P}^\mu(C\cap\QQ_B).
\end{equation}
\end{lemma}

\begin{proof}
If $C\cap\QQ_B=\varnothing$ the statement is trivially true. Assume thus that $C\cap\QQ_B\neq \varnothing$, that is,  $y_i\leq B(i+1)^2$ for all $0\leq i\leq k$.
If $n>k$, then we have
$$\nu_n(C\cap \QQ_B)=\mathbb{P}^\mu(\alpha\in  C\cap \QQ_B|\alpha_n=y_0).$$
The proof continues   as the proof of Lemma \ref{othertopology}.

\end{proof}

\begin{prop}
The sequence of probability measures
$$\varphi_*\nu_n=\sum_{z: f^n(z)=z}Q^n(\varphi^{-1}(z),\varphi^{-1}(z))\delta_z$$ converges weakly to $\tilde \mu$ as $n\to\infty$. 

\end{prop}
\begin{proof}
The proof is analogous to the proof of Proposition \ref{autumn}, using  Lemma \ref{othertopology2} and Theorem \ref{lido}.

\end{proof}

\begin{figure}[h]
  \includegraphics[width=12cm]{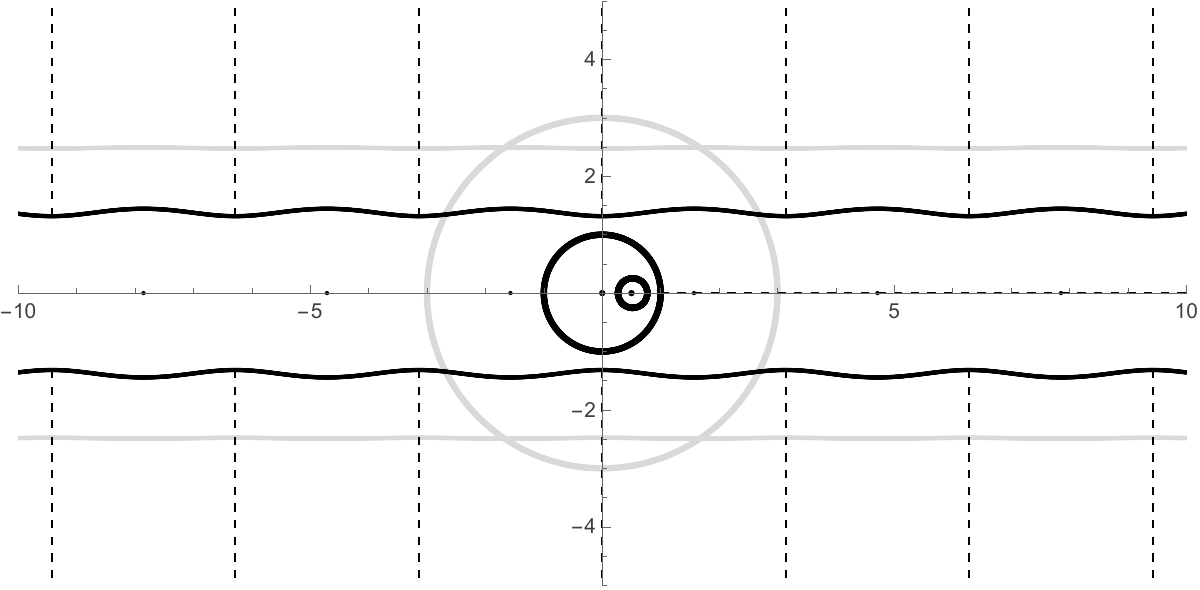}
  \caption{An illustration of important subsets for the disjoint-type map $\cos(z)/2$. Illustrated in black is the unit circle, its forward image, which is compactly contained in the unit disk, and its inverse image, which is the union of two horizontal curves, partitioning the complex plane into three topological disks: the middle sector, which is the inverse image of the unit disk, and the two tracts above and below. The curve $\delta$ is chosen on the positive real axis. The dashed lines illustrate $\delta$ and the inverse images of $\delta$. These inverse images partition each of the tracts into infinitely many fundamental domains $F_i$. Depicted in light gray are the circle of radius 3 and its preimage, which consists of two components, one in each of the tracts. These two curves partition each fundamental domain $F$ in two components; the bounded component is the cell $C(F,3)$.}
  \label{fig:disjoint}
\end{figure}


\subsection{Proofs of Propositions~\ref{prop:class B covering of cells logarithmic version} and  \ref{lem: size of cells log plane}} \label{sec:proof of class B covering of cells}

The following condition is introduced in  \cite{RRRS}. 
\begin{definition}
Let $\Omega\subset\C$ be a domain which is bounded to the left. 
The domain $\Omega$ satisfies the \emph{bounded slope condition} if 
there exist constants $M_1\geq 0, M_2\geq 0$ such that  for all $z,w\in \Omega$,
\begin{equation}\label{eq:bounded slope}
|\Im z-\Im w|\leq M_1 \max\{\Re z, \Re w\} +M_2.
\end{equation}
Notice that if $\Omega$ satisfies the bounded slope condition, then for all $t\in \R$ the domain $\Omega+it$ also satisfies the bounded slope condition, with the same constants $M_1,M_2$.
\end{definition}

 The following result is a combination of  Theorem 5.6 and Corollary 5.8 in  \cite{RRRS} (see also \cite[Proposition 2.5]{BF15}). The two results in \cite{RRRS} are stated for a function with bounded set of singular values; we state it in the more restrictive disjoint type setting. 
 
 \begin{prop}[\cite{RRRS}]
Let $f$ be a disjoint type function which is a finite composition of functions of finite order. Let $\widetilde f$ be a lift for $f$. Then
every tract $\widetilde T$ for $\widetilde f$ satisfies the bounded slope condition.
\end{prop}
\begin{remark}\label{prop:bounded slope}
It is easy to see that, under the assumptions of the previous proposition, every strip $ S$ satisfies the bounded slope condition.
\end{remark}

\begin{remark}\label{enlarge}

Let $D'\Subset D$ be a holomorphic disk  with real-analytic boundary,
 containing  $S(f)\cup \{0\}$, for which $f(D')\Subset D'$. Define  $\widetilde H':={\exp}^{-1}(\mathbb C \setminus \overline{D'})$, and notice that $\overline{\widetilde H} \subset  \widetilde H'$. If $\widetilde T$ is a tract for $\widetilde f$, the inverse  map $\varphi_{\widetilde T}\colon \widetilde H\to \widetilde T$ extends to a biholomorphism 
 $$\varphi_{\widetilde T'}\colon \widetilde H'\to \widetilde T',$$ where $\widetilde T'$ is a simply connected open set containing $\widetilde T$ (the domain $\widetilde T'$  is actually a tract for a lift of $f$ which extends $\widetilde f$ and is defined on the larger open set $\exp^{-1}(f^{-1}(\C\setminus D')) )$.  We can choose $D'$ in such a way that  every $\widetilde T'$ is contained in $ \widetilde H$.

\end{remark}

\begin{proof}[Proof of Proposition \ref{prop:class B covering of cells logarithmic version}]
Let $M_1,M_2\geq 0 $ be the constants given by the bounded slope condition for $S_0$, which also apply to any other strip $S_i$. 
Let $\widetilde F$ be any fundamental domain for $\widetilde f$ and   let $\widetilde S=\widetilde f (\widetilde F)$.
Let $F_i(T)$ be a fundamental domain for $f$ and let  $\widehat F$ the fundamental domain for $\widetilde f$ contained in $\widetilde S$ such that $\exp(\widehat F)=F_i(T).$  Recall that by the ordering of the fundamental domains,   $\widetilde{f} (\widehat{F}) = S_i$.

Let $\alpha>0$ to be determined later. By definition 
\begin{align}
\label{eq:C1}\widetilde f ( C(\widetilde F,t))&=\widetilde S \cap \{\Re z< \log t\}, \\
  \label{eq:C2}  \widetilde f ( C(\widehat F,e^{t^\alpha }))&= S_i \cap \{\Re z<  t^\alpha\}.
\end{align}
Consider the   tract $\widehat T\subset \widetilde S$  such that $\exp(\widehat T)=T$.
Let $\phi_{\widehat T}\colon \widetilde H\to \widehat T$ be the inverse of $\widetilde f$.

In view of \eqref{eq:C1} and \eqref{eq:C2},   in order to show that
$$
C(\widehat F, e^{t^\alpha} )\Subset \widetilde f(C(\widetilde F,t))  \text{ for $|i|\leq \alpha\log t$},
$$
it is enough to show that
$$
\varphi_{\widehat T} ( S_i \cap \{\Re z< t^\alpha\})\Subset \widetilde S\cap \{\Re z< \log t\}    \text{ for $|i|\leq \alpha\log t$}.
$$

We first show that there exists $t_T\geq 0$ such that for all $t\geq t_T$
\begin{equation}\label{eq:class B condition 1}\varphi_{\widehat T} ( S_i \cap \{1\leq \Re z< t^\alpha\})\Subset \widetilde S\cap \{\Re z< \log t\}    \text{ for $|i|\leq \alpha\log t$}.
\end{equation}
Let $t$ be large enough so that $\varphi_{\widehat T}(1)<\log t$. Let $z\in S_i \cap \{1\leq \Re z< t^\alpha\}$ where $|i|\leq \alpha\log t$. We now estimate the real part of $z$. Since $\H\subset \widetilde H$ we have, thanks to the bounded slope condition,
\begin{align*}
\dist_{\widetilde T}(\phi_{\widetilde T}(1), \phi_{\widetilde T}(z))&=\dist_{\widetilde H}(1,z)\leq  \dist_\H(1, z) \leq \int_{1}^{\Re z}\frac{1}{t}dt+\left|\int_{0}^{\Im z}\frac{1}{\Re z}dt\right|=
\log \Re z  +\frac{|\Im z|}{\Re z}\leq\\
&\leq \log {\Re z}+\frac{M_1 \Re z+M_2+2\pi \alpha \log t }{\Re z}\leq M_1+M_2+(2\pi+1) \alpha \log t .
\end{align*}

Since points in tracts have distance at most $\pi$ from the boundary, we have $|x-y|\leq 2\pi \dist_{\widetilde T}(x,y)$  for all $x,y\in \widetilde T$,  and thus
$$
|\Re \phi_{\widetilde T}(1)-\re \phi_{\widetilde T}(z)|\leq |\phi_{\widetilde T}(1)-\phi_{\widetilde T}(z)|\leq 2\pi \dist_{\widetilde T}(\phi_{\widetilde T}(1), \phi_{\widetilde T}(z)).
$$
It follows that 
$$
\re \phi_{\widetilde T}(z)\leq \re \phi_{\widetilde T}(1)+2\pi (M_1+M_2)+2\pi(2\pi+1) \alpha \log t,
$$
and if $\alpha<\frac{1}{2\pi(1+2\pi)}$, then the right-hand side  is smaller than $\log t$ if $t_T$ is large enough, and this proves \eqref{eq:class B condition 1}.

We now show that   there exists $t_T\geq 0$ such that for all $t\geq t_T$
\begin{equation}\label{eq:secondcase}\varphi_{\widehat T} ( S_i \cap \{ \Re z\leq 1\})\Subset \widetilde S\cap \{\Re z< \log t\}    \text{ for $|i|\leq \alpha\log t$}.
\end{equation}
Let $t$ be large enough so that $\varphi_{\widehat T}(1)<\log t$. Let $z\in S_i \cap \{\Re z\leq 1\}$ where $|i|\leq \alpha\log t$. We now estimate the real part of $z$. Let $\widetilde T', \widetilde H',\varphi_{\widetilde T'}$ as defined in Remark \ref{enlarge}. Fix  a base point $z_0\in S_0 \cap \{\Re z=1\}$, and denote 
$$q:=\diam_{\widetilde H'}(S_0\cap \{ \Re z\leq 1\})<+\infty.$$
We have $$\dist_{\widetilde T'}(\phi_{\widetilde T}(z_0), \phi_{\widetilde T}(z))=\dist_{\widetilde H'}(z_0,z)\leq q\alpha \log t,$$
and thus
$$
|\Re \phi_{\widetilde T'}(z_0)-\re \phi_{\widetilde T'}(z)|\leq |\phi_{\widetilde T'}(z_0)-\phi_{\widetilde T'}(z)|\leq 2\pi \dist_{\widetilde T'}(\phi_{\widetilde T'}(z_0), \phi_{\widetilde T'}(z))\leq 2\pi q\alpha\log t.
$$
It follows that 
$\re \phi_{\widetilde T'}(z)\leq \Re \phi_{\widetilde T'}(z_0)+ 2\pi q\alpha\log t,$ and 
 if $\alpha<\frac{1}{2\pi q}$, then the right-hand side  is smaller than $\log t$ if $t_T$ is large enough, and this proves \eqref{eq:secondcase}.
\end{proof}

%

\begin{proof}[Proof of Proposition \ref{lem: size of cells log plane}]
Let $M_1,M_2\geq 0 $ be the constants given by the bounded slope condition for $S_0$, which apply to   any strip $S_i$. 
 Let $\widetilde S=\widetilde f(\widetilde F)$, and $\widetilde T$ be the tract containing $\widetilde F$.  Let $\beta>1$,  $n\in\N$.
The cell 
$$
C(\widetilde F, \beta^n)=\varphi_{\widetilde T}(\widetilde S\cap \{\Re z\leq n\log \beta\})
$$
is the union of  the two sets
$$
X_1:=\varphi_{\widetilde T}(\widetilde S\cap \{\Re z\leq 1\}),\quad X_2:=\varphi_{\widetilde T}(\widetilde S\cap \{1\leq \Re z\leq n\log \beta\}).
$$
 
  Since $X_1\cap X_2\neq\emptyset$,
$$
   \diam_{\widetilde H} C(\widetilde F, \beta^n)  \leq\diam_{\widetilde H}(X_1)+\diam_{\widetilde H}(X_2).
$$

We first estimate $\diam_{\widetilde H}(X_2)$. We have
$$
\diam_{\widetilde H}(X_2) \leq \diam_{\widetilde T} (X_2)=\diam_{\widetilde H}( \widetilde S\cap \{1\leq \Re z\leq n\log \beta\}).
$$

Fix  a base point $z_0\in \widetilde S\cap \{\re z=1\}$, and let $z\in  \widetilde S\cap \{1\leq \Re z\leq n\log \beta\}$.
Since $\H\subset \widetilde H$ we have, thanks to the bounded slope condition,
\begin{align*}
\dist_{\widetilde H}(z_0,z)\leq \dist_\H(z_0,z)&\leq\int_{1}^{\re z} \frac{dt}{t}+\left|\int_{\im z_0}^{\im z}\frac{dt}{\re z}\right|\leq \log \re z+ \frac{|\Im z-\Im z_0|}{\re z}\\
&\leq \log \re z+\frac{M_1 \re z+M_2}{\re z}\\
&\leq (\log\log \beta )(\log n)+ M_1+M_2.
\end{align*}

We now estimate $\diam_{\widetilde H}(X_1)$.
Let $\widetilde T', \widetilde H',\varphi_{\widetilde T'}$ as defined in Remark \ref{enlarge}.
We have 
$$
\diam_{\widetilde H}(X_1)\leq   \diam_{\widetilde T'} (X_1)=\diam_{\widetilde H'}( \widetilde S\cap \{ \Re z\leq 1\}):=B.
$$
Notice that the constant $B$ does not depend on the strip $\widetilde S$, and thus does not depend on the fundamental domain $\widetilde F$. The estimates
$$
    \diam_{\widetilde H} C(\widetilde F, \beta^n)  \leq\diam_{\widetilde H}(X_1)+\diam_{\widetilde H}(X_2)\leq (\log\log \beta )(\log n)+M_1+M_2+B.
$$
complete the proof.
\end{proof}


\section{Invariant measures for elementary strongly polynomial-like maps} \label{sec: strongly polynomial-like}

\subsection{Transfer operator for rational functions.}

Let us recall the method that Lyubich used in \cite{Lyubich83} to obtain equidistribution for rational functions, and discuss how we will adapt this method to the transcendental setting.

Let $f$ be a rational function of degree $d\ge 2$ acting on the Riemann sphere $\widehat{\mathbb C}$, and let us assume for the sake of simplicity that the exceptional set $\mathcal{E}(f)$ is empty. The \emph{transfer operator} $A = A_f$, acting on the Banach space of continuous functions $C(\widehat{\mathbb C})$ equipped with the sup-norm, is then defined by
\[
A(\varphi)(w) = \sum_{z \in f^{-1}(w)} \frac{1}{d} \varphi(z). 
\]

\begin{prop}[Lyubich]
For every $\varphi \in C(\widehat{\mathbb C})$ the family $(A^n(\varphi))_{n \in \mathbb N}$ is equicontinuous. 
\end{prop}

By Arzel\`a-Ascoli it follows that the transfer operator $A$ is almost periodic: every subsequence of $(A^n(\varphi))_{n \in \mathbb N}$ has a uniformly convergent subsequence. Lyubich then shows that whenever $A(\varphi) = \lambda \cdot \varphi$ with $|\lambda | = 1$, then $\lambda = 1$ and $\varphi$ must be a constant function. The existence of the invariant probability measure then follows from the following abstract functional analytic result.

\begin{thm}[Lyubich]
Let $B$ be a complex Banach space, and let $A: B \rightarrow B$ be an almost periodic operator. Suppose that the unitary spectrum of $A$ equals $\{1\}$, and that the corresponding eigenspace is simple and spanned by $h \in B$. Then there exists a unique functional $\mu \in B^*$ such that  $A^n(\varphi) \rightarrow \mu(\phi) h$ as $n \rightarrow \infty$ for every $\varphi \in B$.
\end{thm}

In the transcendental setting we will work with a similar transfer operator $A = A_f$, acting on the complex Banach space $C_b(\mathbb C)$ of bounded continuous complex-valued functions on $\mathbb C$,  equipped with the sup norm $\|\cdot\|_{\mathbb C}$. Following ideas that are similar in spirit as those used by Lyubich, we will prove that for each $\varphi \in C_b(\mathbb C)$ the family $(A^n(\varphi))$ is equicontinuous. However, Arzel\`a-Ascoli only implies that every subsequence of  $(A^n(\varphi))_{n \in \mathbb N}$ has a subsequence that converges locally uniformly, and thus we will not be able to use the above functional analytic result. 
In fact, it is clear that we cannot expect the sequence $(A^n(\varphi) )$ to converge uniformly to a nonzero constant: indeed, if $\varphi$ has compact support, then so does $A^n(\varphi)$ for every $n \in \mathbb N$, hence the functions $A^n(\varphi)$ cannot converge uniformly to any non-zero constant.

Instead of relying on an abstract functional analytic result, we will prove directly in Theorem \ref{thm: non-uniform Lyubich} that the equicontinuity of the family $(A^n(\varphi))_{n \in \mathbb N}$ and the tightness  of the sequence of probability measures $((A^*)^n \delta_w)$ imply  that for every $\varphi\in C_b(\C)$ the sequence  $(A^n(\varphi))$ converges locally uniformly to a constant function $\varphi_\infty$. 
Again by tightness the functional $\mu \in C_b(\mathbb C)^*$ that assigns  to a function $\varphi$ the value $\varphi_\infty$ is a probability measure on $\C$.

There will be no condition that corresponds to the description of the unitary spectrum. However, we will work in the concrete setting of the transfer operator, and we will prove in Lemma \ref{lemma: stable implies constant} and Proposition \ref{prop: limits are constants} that any subsequential limit of  the sequence $(A^n(\varphi))$ is constant. The reader familiar with the proofs from \cite{Lyubich83} will recognize many of the steps in these proofs.

Let us end this preamble by discussing the family of test functions $C_b(\mathbb C)$.
 Due to the tightness assumption, the sequence $((A^*)^n \delta_w)$ converges weakly to a probability measure $\mu$ if and only if  $$\int_\C\varphi\, d(A^*)^n \delta_w\to \int_\C\varphi \,d\mu$$ for all $\varphi$ in the space $C_0(\C)$ of continuous functions vanishing at $\infty$.  Hence one could work instead with the Banach space  $C_0(\mathbb C)$, but working in this smaller Banach space does not affect any of the proofs.


\subsection{Transfer operator for entire functions}
Let $f\colon \C\to \C$ be an entire function. We limit this discussion to the case where $f$ is surjective, but note that it can easily be adapted to a transcendental function with an omitted value.

\begin{definition}[Weight function]
A \emph{weight function} is a function $Q(w,z)$ with values in $[0,1]$ defined for every $(z,w)\in \C^2$ such that $f(z)=w$, which  for all $w\in \C$ satisfies $$\sum_{z\in f^{-1}(w)}Q(w,z)=1.$$
Notice that here and throughout the section we sum over preimages counting with multiplicities. With slight abuse of notation we refer to the function $Q\colon \C\to [0,1]$ defined by 
$$
Q(z):=Q(f(z),z),
$$
 also as a weight function, which clearly carries the same information as the function $(z,w)\mapsto Q(w,z)$.

Let $Q$ be a weight function. For every $w\in \C$ we define a probability measure
$Q_w$ with support contained in $f^{-1}(z)$ by
$$
Q_w:=\sum_{z\in f^{-1}(w)}Q(z)\delta_z.
$$ 
Moreover, for every function $\varphi\in C_b(\C)$, we define a function
\begin{equation}\label{eq: transfer operator}
A(\varphi)\colon \C\to \C,\quad A (\phi)(w) := \sum_{z\in f^{-1}(w)} Q(z) \phi(z)=\int_\C \varphi dQ_w,
\end{equation}
and notice that $\|A(\varphi)\|_\C\leq \|\varphi\|_\C.$
\end{definition}
Note that for a polynomial of degree $d$  we can take the constant function $Q(z) = \frac{1}{d}$.

\begin{lemma}
Let $Q$ be a weight function. The following are equivalent:
\begin{enumerate}
\item for every $\varphi\in C_b(\C)$ the function $A(\varphi)$ is continuous,
\item the map $w\mapsto Q_w$ from $\C$ to $(C_b(\C))^*$ endowed with the weak-$*$ topology is continuous.
\end{enumerate}
\end{lemma}

\begin{proof}
Both properties are equivalent to the fact that if $w_k\to w$ and $\varphi\in C_b(\C)$ then $$\int_\C\varphi \, dQ_{w_k}\to\int_\C\varphi\,dQ_{w}.$$
\end{proof}
\begin{remark}
Notice that (1) implies that  $A\colon C_b(\C)\to C_b(\C)$ defined in (\ref{eq: transfer operator}) is  a bounded continuous operator with $\|A\|\leq 1$, called the \emph{transfer operator} associated with $Q$. 
Also, (2) can be restated by saying  that  $(Q_w)_{w\in\C}$ is the transition kernel of a Feller Markov process on $\C$.
\end{remark}

\begin{definition}
We say that the weight function $Q$ is \emph{weakly continuous} if one of the two equivalent conditions from the previous lemma is satisfied. 
\end{definition}
We now show that if a weight function $Q$ is continuous as a function from $\C$ to $[0,1]$, then it is weakly continuous and thus 
its associated transfer operator is well defined.
\begin{lemma}\label{lem: transfer operator well defined}
If the weight function $Q\colon \C\to [0,1]$ is continuous, then it is weakly continuous.
\end{lemma}

\begin{proof}
Let $\varphi \in C_b(\mathbb C)$, $w \in \mathbb C$ and $\epsilon>0$, and write $\|\varphi\|_\C = M$. Without loss of generality we may assume that $M>0$. We now show that  there exists $\delta>0$ such that  if $|w-w'|\leq \delta$, then 
$$
\left|\int_\C\varphi \, dQ_{w'}-\int_\C\varphi\,dQ_{w}\right|=| \sum_{z' \in f^{-1}(w')} Q(z')\phi(z')- \sum_{z\in f^{-1}(w)} Q(z)\phi(z)|\leq \epsilon,
$$
which proves the result.

Since the probability measure $Q_w$ is supported on the countable set $f^{-1}(w)$, we can select a finite subset $S \subset f^{-1}(w)$
for which
 $$
Q_w(S) >   1-\frac{\epsilon}{4M},
$$
and thus
$$
|\sum_{z\in f^{-1}(w)\setminus S} Q(z)\phi(z)|\leq \epsilon/4.
$$

For $w'$ sufficiently close to $w$ we can find a corresponding subset $S' \subset f^{-1}(w')$ such that each $z' \in S^\prime$ is close to a point $z\in S$. Note that there may be multiple points in $S^\prime$ corresponding to a single point $z \in S$ if $z$ is a critical point of $f$. Since $Q$ and $\varphi$ are continuous, we can choose $\delta>0$ sufficiently small such that for $|w-w'|\leq \delta$,
\[
|\sum_{z^\prime \in S^\prime} Q(z^\prime) \varphi(z^\prime) - \sum_{z \in S} Q(z) \varphi(z)| \leq \frac{\epsilon}{2}.
\]
Thanks to the continuity of $Q$, by reducing $\delta$ if necessary we can guarantee that 
\[
Q_{w'}( S^\prime) >   1-\frac{\epsilon}{4M},
\]
 and thus
\[
|\sum_{z^\prime \in f^{-1}\{w^\prime\} \setminus S^\prime} Q(z^\prime) \varphi(z^\prime) - \sum_{z \in f^{-1}\{w\} \setminus S} Q(z) \varphi(z)| \leq \frac{\epsilon}{2},
\]
which completes the proof.
\end{proof}

\begin{remark}
The converse of the previous lemma does not hold. It is easy to see that a weakly continuous weight function $Q$ has to be continuous at every regular point of $\varphi$, but it could be discontinuous at critical points. 
\end{remark}
In the rest of this subsection, we merely assume that $Q$ is weakly continuous. However, in the examples that we will consider in the remaining of the section the weight function $Q$ will actually be continuous.  

%
%
%
%
\begin{remark}
The transpose operator $A^*\colon (C_b(\C))^*\to  (C_b(\C))^*$ sends the Dirac measure $\delta_w$ to the measure $Q_w$. For all $n\geq 1$ denote by
\begin{equation} \label{eq:etan}
Q^n_w:=(A^*)^n(\delta_w)=\sum_{z\in f^{-n}(w)}Q^n(z)\delta_z,
\end{equation}
 where 
\begin{equation}\label{nweights}Q^n(z):=\prod_{j=0}^{n-1}Q(f^j(z)).
\end{equation}
Notice that, for every probability measure $\nu$ on $\C$ and for all measurable $B\subset \C$, 
$$(A^*)^n(\nu)(B)=\int_\C Q^n_w(B)d\nu(w).$$
Hence a probability measure $\nu$ satisfies $A^*\nu=\nu$ if and only if
\begin{equation}\label{eq:measure of preimages integral2}
\nu(B)=\int_\C Q_w(B)d\nu(w)=\int_\C\left( \sum_{z\in f^{-1}(w)\cap B} Q(z)\right) d\nu(w).
\end{equation}
In particular, if $g\colon B\to \C$ is a holomorphic branch of $f^{-1}$ defined on an open set $B\subset \C$, then
\begin{equation}\label{integralweight}
\nu(g(B))=\int_BQ(g(w))d\nu(w) .
\end{equation}
\end{remark}

%
%
%
%
%
\subsection{Weighted equidistribution of preimages}
We  introduce  conditions on a weakly continuous weight function $Q$ that will imply weighted equidistribution of the preimages of $f$, that is the existence of a probability measure $\mu$ on $\C$ such that $Q_w^n\to \mu$ weakly for all $w\in \C$.
\begin{defn}[Normality]
    We say that a weakly continuous weight function $Q$ is \emph{normal} if for every $\phi \in C_b(\C)$ the family $(A^n (\phi))_{n \in \mathbb N}$ is equicontinuous on $\mathbb C$.
\end{defn}
\begin{defn}[Tightness]
    We say that a weakly continuous weight function $Q$ is \emph{tight} if for every $\epsilon>0$ there exists a compact $K \subset \mathbb C$ such that for every $w \in \mathbb C$ there exists $n_0 \in \mathbb N$ such that for $n \ge n_0$ one has that
        \[
        Q^n_w(K) = \sum_{z \in f^{-n}(w)\cap K} Q^n(z) \ge 1-\epsilon.
        \]
\end{defn}

\begin{defn}[Irreducibility]\label{defirr}
    We say that a weakly  continuous weight function $Q$ is \emph{irreducible} if one of the two following conditions holds:
    \begin{enumerate}
        \item there exists a repelling  fixed point $p$  with  $0<Q(p)<1$ such that for every $r>0$ and  every compact $K \subset \mathbb C$ there exists an $n \in \mathbb N$ such that for every $w \in K$ there exists $z\in D(p,r)$ with $f^n(z) = w$ and $Q^n(z) > 0$,
    \item
 there exist   two distinct repelling  periodic points $p_1,p_2$  with coprime period and with  $Q(p_i)>0$ such that for every $i=1,2$, every $r>0$, and  every compact $K \subset \mathbb C$ there exists an $n \in \mathbb N$ such that for every $w \in K$ there exists $z\in D(p_i,r)$ with $f^n(z) = w$ and $Q^n(z) > 0$.

    \end{enumerate}

\end{defn}
In the proofs that follow we will assume to be in case (1) of the previous definition for simplicity of exposition, but the proofs can be easily adapted to case (2).
\begin{remark}
We recall that for general transcendental functions there may not exist any repelling fixed point,  however, by 
\cite[Theorem 1]{BerPer} every transcendental function is known to have infinitely many repelling cycles of any period at least $2$.
\end{remark}


\begin{remark}
If the weight function $Q$  is nowhere zero, thanks to the previous remark the weight function $Q$ is irreducible if and only if the function $f$ has  no exceptional point.
\end{remark}

\begin{defn}[$Q$-Julia set]
    The $Q$-Julia set of $f$, denoted by $J_Q(f)$, is defined as the set of all points $z_0 \in \mathbb C$ which have the property that for every $r>0$ the set 
    \[
    \Sigma(z_0,r):=\{w\in \mathbb C \mid \exists z\in D(z_0,r), n \in \mathbb N: f^n(z) = w, Q^n(z) > 0\}
    \]
    avoids at most one point in $\C$.
  
  
\end{defn}

\begin{remark}
Let us say that a point $z$ is $Q$-\emph{exceptional} if it is fixed and $Q(z)=1$. Clearly an exceptional point is $Q$-exceptional. 
 Let $z_0\in\C$ and $r>0$. If $ \Sigma(z_0,r)$ avoids exactly one point $z$, then $z$ is $Q$-exceptional. 
 Notice that if $Q$ is irreducible, then no point of $\C$ can be $Q$-exceptional, and thus for all $z_0\in \C, r>0$ we have
$ \Sigma(z_0,r)=\C$.
\end{remark}

\begin{remark}
The $Q$-Julia set is closed and   forward invariant, but is backward invariant only with respect to preimages with positive weight.
Notice moreover that if the weight function $Q$ is nowhere zero, then $J_Q(f) = J(f)$.
\end{remark}
We  now prove weighted equidistribution of preimages  in the case of a nowhere zero weight function.
\begin{thm}[Weighted equidistribution of preimages]\label{thm: non-uniform Lyubich}
 Assume $f$ has no exceptional points.
  Let $Q$ be a weakly continuous weight function which is normal, tight, and nowhere zero. 
Then for every $\phi \in C_b(\C)$ the sequence of functions $(A^n (\phi))$ converges locally uniformly to a constant function $\varphi_\infty$, and the functional in $(C_b(\C))^*$ defined by
$ \varphi\mapsto \varphi_\infty$ is a Borel probability measure $\mu$ on $\mathbb C$ whose support equals the Julia set $J(f)$.

Moreover, the probability measure $\mu$ is $A^*$-invariant (and thus $f$-invariant) and mixing,  and if $\nu$ is an arbitrary Borel probability measure on $\C$, then the sequence of probability measures  
$((A^*)^n(\nu))$ converges weakly to $\mu$.  In particular, for all $w\in \C$ we have that $(Q^n_w)$ converges  weakly to $\mu$.

\end{thm}

We need several preliminary results. It follows by Arzela--Ascoli that for every $\phi \in C_b(\C)$ and every sequence in the family $(A^n (\phi))_{n \in \mathbb N}$ there exist locally uniform subsequential limit functions. We refer to those simply as subsequential limits of $\phi$.

\begin{defn} Let $Q$ be a  weakly continuous normal weight function. A bounded continuous function $\phi \in C_b(\C)$ is said to be \emph{stable} for $A$ if for any subsequential limit $\psi$ of $\phi$ we have
\[
\|\psi\|_{\mathbb C} = \|\phi\|_{\mathbb C}.
\] 
\end{defn}
We note that $\|\psi\|_{\mathbb C} \le \|\phi\|_{\mathbb C}$ is always satisfied, since $A$ is a bounded linear operator with norm $1$, and stability implies that $\|A^n (\phi)\|_{\mathbb C} = \|\phi\|_{\mathbb C}$ for every $n \in \mathbb N$.

\begin{lemma} Let $Q$ be a weakly continuous weight function  that  is  normal and tight.    If $\psi$ is a subsequential limit of $\phi\in C_b(\C)$, then $\psi$ is stable.
\end{lemma}
\begin{proof}
    Let us suppose that $A^{n_j}( \phi)$ converges locally uniformly to $\psi$, and that $A^{m_l} (\psi)$ converges locally uniformly to  $\eta$. Suppose for the purpose of a contradiction that $\|\psi\|_{\mathbb C} > \|\eta\|_{\mathbb C}$. Let $\epsilon>0$ sufficiently small such that
    \[
    \epsilon \|\phi\|_{\mathbb C} + (1-\epsilon) \|\eta\|_{\mathbb C} < \|\psi\|_{\mathbb C}.
    \]
    Since $Q$ is assumed to be tight, we can find a compact set $K$ such that for each $w \in \mathbb C$ there exists $n_0(w)$ such that 
    \[
    Q^n_w(K) \ge 1-\epsilon, \quad \forall\, n\geq n_0(w),
    \]
    and such that 
    \[
    \epsilon \|\phi\|_{\mathbb C} + (1-\epsilon) \|\eta\|_{K} < \|\psi\|_{\mathbb C}.
    \]
Since $(A^{m_l}( \psi))$ converges uniformly on $K$ to the function $\eta$, we can choose $l\in \mathbb N$ sufficiently large such that
    \[
    \epsilon \|\phi\|_{\mathbb C} + (1-\epsilon) \|A^{m_l}( \psi)\|_K < \|\psi\|_{\mathbb C}.
    \]
    Since $(A^{n_j}(\phi))$ converges locally uniformly to $\psi$ as $j \rightarrow \infty$, the sequence $(A^{n_j+m_l}(\phi))$ converges locally uniformly to $A^{m_l}(\psi)$ as $j \rightarrow \infty$, for $l$ fixed. It follows that we can choose $j$ sufficiently large such that 
    \[
    \epsilon \|\phi\|_{\mathbb C} + (1-\epsilon) \|A^{n_j+m_l} (\phi)\|_K < \|\psi\|_{\mathbb C}.
    \]
    Since $\|A^{n_j+m_l}( \phi)\|_{\mathbb C} \le \|\phi\|_{\mathbb C}$ this implies
    \[
    \epsilon \|A^{n_j+m_l} (\phi)\|_{\mathbb C} + (1-\epsilon) \|A^{n_j+m_l}( \phi)\|_K < \|\psi\|_{\mathbb C}.
    \]
    
Let $w\in \C$, and  let $j'\geq 0$ large enough such that $n_{j'}\geq n_j+m_l+n_0(w)$. Then by tightness  we have
    $$|A^{n_{j'}}(\varphi)(w)|=\left|\int_\C  A^{n_j+m_l}(\varphi)       dQ_w^{n_{j'}-n_j-m_l}\right|\leq \epsilon \|A^{n_j+m_l} (\phi)\|_{\mathbb C} + (1-\epsilon) \|A^{n_j+m_l} (\phi)\|_K,$$
    which contradicts the assumption that the sequence $(A^{n_{j}} (\phi))$ converges locally uniformly to $\psi$.
\end{proof}

\begin{lemma} \label{lemma: stable implies constant} Let $Q$ be a weakly continuous weight function that is normal, tight, and irreducible.
    The function $\phi \in C_b(\C)$ is stable if and only if $\phi$ is constant on $J_Q(f)$ and obtains its maximal modulus on $J_Q(f)$.
\end{lemma}
\begin{proof}    
    The backwards implication is immediate, we will prove the forward implication.
    Let us first prove that the modulus of $\phi$ must be maximal at each point in $J_Q(f)$. For the purpose of a contradiction, assume that there is a point $z_0 \in J_Q(f)$ such that $|\phi(z_0)| < \|\phi\|_{\mathbb C}$. Since $\phi$ continuous, there exists a disk $D(z_0,r)$ such that $\|\phi\|_{D(z_0,r)} < \|\phi\|_{\mathbb C}$. 
    Let $p$ be the repelling fixed point of $f$ given by  the definition of irreducibility.
    By the definition of the $Q$-Julia set, we can find a point $z \in D(z_0,r)$ such that $f^n(z) = p$ and $Q^n(z) > 0$.
    It follows that $|A^n(\phi)(p)|< \|\phi\|_{\mathbb C}$, and therefore, there exists $s>0$ such that $|A^n(\phi)|_{D(p,s)}< \|\phi\|_{\mathbb C}$.

    Let $\epsilon \in (0,1)$, and let $K$ be the compact subset given by tightness relative to the constant $\epsilon$. By the irreducibility assumption, there exists  $m \geq 0$ such that for all $w\in K$ there exists $z\in D(p,s)$ such that $f^n(z)=w$ and $Q^n(z)>0$. It follows that $\|A^{n+m}(\phi)\|_K < \|\phi\|_{\mathbb C}$. By tightness, for every $w \in \mathbb C$ and $j\geq n+m+n_0(w)$ we have that 
    \[
    |A^j(\phi)(w)|=\left|\int_\C A^{n+m}(\varphi)dQ^{j-n+m}_w\right| \le (1-\epsilon) \|A^{n+m}(\phi)\|_K  + \epsilon \|\phi\|_{\mathbb C} .
    \]
 It follows that for any subsequential limit $\psi$ we have that
    \[
    \|\psi\|_{\mathbb C} \le (1-\epsilon) \|A^{n+m}(\phi)\|_K  + \epsilon \|\phi\|_{\mathbb C} < \|\phi\|_{\mathbb C},
    \]
    which gives the desired contradiction.

Let now $\phi\neq 0$ be a stable function. The modulus of   $\phi$ in $J_Q(f)$ is constant and non-zero. 
    Notice that the function $A^m(\phi)$ is stable for each $m \in \mathbb N$, and hence the modulus of $A^m(\phi)$ must be  equal to $\|\phi\|_\C$ at all points of $J_Q(f)$. Let $w \in J_Q(f)$. Then all preimages $z \in f^{-m}(w)$ with $Q^m(z) > 0$  have the same value
    $\varphi(z)$, since otherwise the modulus of $A^m(\phi)(w)$ would be strictly smaller than $\|\phi\|_\C$ due to cancellation.
      Since $Q(p)>0$, it follows that every preimage $z\in  f^{-m}(p)$ with $Q^m(z) > 0$ satisfies $\varphi(z)=\varphi(p).$
     By irreducibility every point of $J_Q(f)$ is accumulated by preimages $z \in f^{-m}(p)$ with $Q^m(z) > 0$, hence $\varphi $
      is constantly equal to $\varphi(p)$ on $J_Q(f)$.
    
\end{proof}

\begin{prop} \label{prop: limits are constants} 
 Assume $f$ has no exceptional points. Let $Q$ be a weakly continuous weight function  that is  normal, tight, and nowhere zero.
    If $\psi$ is a subsequential limit of $\phi$, then $\psi$ is constant.
\end{prop}
\begin{proof}
Since $Q$ is nowhere zero, we have that $J_Q(f)=J(f).$
    Let $M = \|\phi\|_{\mathbb C}$, and note that without loss of generality we may assume that $M \neq 0$. Let us say that $\psi$ is the local uniform limit of the subsequence $(A^{n_j}(\phi))$.
    Since $\psi$ is known to be stable, $\psi$ must be equal to a constant $c$ on $J(f)$. We now show that, if $w$ is in the Fatou set of $f$, then the sequence $(A^{n_j}(\phi)(w))$ converges to $c$. Let $\epsilon>0$ and let 
     $K$ be a compact subset of $\mathbb C$ given by  the definition of tightness such that the mass of preimages $z \in f^{-n}(w)$ in $K$ is at least $1 - \epsilon/(6M)$, for $n \ge n_0(w)$.
    By local uniform convergence,
    there exists an open neighborhood $U$ of $J(f)\cap K$ such that for $j$ large enough we have
    $$\|A^{n_j}(\phi)-c\|_U\leq \frac{\epsilon}{3}.$$
    For all $n_l\geq n_j+n_0(w)$ we have 
    \begin{align*}|A^{n_l}(\varphi)(w)-c|&\leq \int_\C|A^{n_j}(\varphi)-c|dQ_w^{n_l-n_j}\\
    &\leq \int_U|A^{n_j}(\varphi)-c|dQ_w^{n_l-n_j}+\int_{\C\setminus K}|A^{n_j}(\varphi)-c|dQ_w^{n_l-n_j}+\int_{K\setminus U}|A^{n_j}(\varphi)-c|dQ_w^{n_l-n_j}\\
    &\leq \frac{\epsilon}{3}+\frac{\epsilon}{3}+\int_{K\setminus U}|A^{n_j}(\varphi)-c|dQ_w^{n_l-n_j}.
    \end{align*}
To conclude it is enough to show that for $n_l$ large enough we have that $\int_{K\setminus U}|A^{n_j}(\varphi)-c|dQ_w^{n_l-n_j}\leq \epsilon/3.$ 
    We will actually show that $Q_w^n(K\setminus U)\stackrel{n\to+\infty}\longrightarrow 0$.
    Note that  $K \setminus U$ consists of finitely many compact subsets $(L_i)$ of finitely many Fatou components. Clearly it is enough to show that $Q_w^n(L_i)\stackrel{n\to+\infty}\longrightarrow 0$ for all $i$.
    Observe that the preimages of $w$ can only lie in $L_i$ for sufficiently large $n$ if $w$ lies in a periodic Fatou component, and moreover the Fatou component must either be a Siegel disk or $w$ is an attracting periodic point, since in all other cases the backward orbits of $w$ eventually leave any compact subset.
In either case we can increase the compact subset $L_i$ if necessary so that it contains $w$  and it is forward invariant, hence if backwards orbits of $w$ at some time leave $L_i$ they can never return.
    Since $f$ has no exceptional points, it is irreducible.
   Hence for every $z\in L_i$ some backward  orbits of $z$ get close to the repelling fixed point $p$,
and thus by the compactness and  forward invariance of $L_i$ we have that 
     there exists $0<a<1$ and $k\geq 0$ such that $Q_z^k(L_i)\leq a$ for all $z\in L_i$. But then we have, for $m\geq 1$,
     $$Q_w^{km}(L_i)\leq a^m.$$

    \end{proof}

\begin{proof}[Proof of Theorem \ref{thm: non-uniform Lyubich}]
Let $\varphi\in C_b(\C)$.
Assume by contradiction that $\varphi$ admits two distinct subsequential limits $\varphi_\infty\neq \psi_\infty$, which by the previous proposition are constant functions.
Notice that  $A(\varphi_\infty)=\varphi_\infty$ and $A(\psi_\infty)=\psi_\infty$. Set $\varphi_0:=\varphi-\varphi_\infty$, and notice that there exists a subsequence $(n_j)$ such that the sequence $(A^{n_j}(\varphi_0))$ converges to 0 locally uniformly, and a subsequence $(m_l)$ 
such that the sequence $(A^{m_l}(\varphi_0))$ converges to  the constant $ \varphi_\infty- \psi_\infty\neq 0$ locally uniformly.
Let $\epsilon:=\frac{1}{2}|\varphi_\infty- \psi_\infty |>0$.
 Let $K\subset \C$ be the compact subset given by tightness relative to the constant $\frac{\epsilon}{2\|\varphi_0\|_\C}$.
 Let $w\in \C$ and let  $N\geq 0$ be such that for all $n\geq N$ we have $Q^n_w(K)\geq 1-\frac{\epsilon}{2\|\varphi_0\|_\C}$. Let $j\geq 0$ be such that 
$\|A^{n_j}(\varphi_0)\|_K\leq \epsilon/2.$
For all $l\geq 0$ such that  $m_l\geq n_j$ we have that 
$$A^{m_l}(\varphi_0)(w)=\int_\C A^{n_j}(\varphi_0)\,dQ_w^{m_l-n_j}.$$ Hence for all $l$ large enough so that $m_l\geq N+n_j$ we have that 
$|A^{m_l}(\varphi_0)(w)|\leq \epsilon$, which contradicts the fact that $A^{m_l}(\varphi_0)(w)\to \varphi_\infty- \psi_\infty.$

We claim that the whole sequence
$(A^n(\varphi_0))$ converges to 0 locally uniformly. Let $H\subset \C$ be a compact subset and let $\epsilon>0$. Let $K\subset \C$ be the compact subset given by tightness relative to the constant $\frac{\epsilon}{2\|\varphi_0\|_\C}$, and let $N$ be such that for all $n\geq N$ and for all $w\in H$ we have $Q^n_w(K)\geq 1-\frac{\epsilon}{2\|\varphi_0\|_\C}$. Let $j\geq 0$ be such that 
$\|A^{n_j}(\varphi_0)\|_K\leq \epsilon/2.$
For all $n\geq n_j$ we have that 
$$A^n(\varphi_0)(w)=\int_\C A^{n_j}(\varphi_0)dQ_w^{n-n_j}.$$ Hence for all $w\in H$ and for all $n\geq N+n_j$ we have that 
$$|A^n(\varphi_0)(w)|\leq \epsilon\|\varphi_0\|_\C+\epsilon/2\leq \epsilon,$$ and the claim is proved.
It follows that the sequence of functions $(A^n (\phi))$ converges locally uniformly to the constant function $\varphi_\infty$.
Let $\mu\in (C_b(\C))^*$ be the linear functional defined by $\varphi\mapsto \varphi_\infty$. Clearly $A^*(\mu)=\mu$, and for all $w\in \C$ the sequence
of probability measures $(Q^n_w)$ converges to $\mu$ in the weak-$*$ topology:
$$\int_\C \varphi d Q^n_{w}=A^n(\varphi)(w)\stackrel{n\to\infty} \longrightarrow \varphi_\infty= \mu(\varphi).$$
Notice that $(C_b(\C))^*$ contains more linear functionals than just finite Borel measures. However, since the sequence of probability measures $(Q^n_{w})$ is tight,  it follows from Prokhorov's Theorem that 
there exists a subsequence converging weakly to a probability measure $\mu'$. It follows that $\mu=\mu'$, and thus $\mu$ is a probability measure. 
The fact that $\mu$ is $f$-invariant follows from the fact that  $\mu$ is $A^*$-invariant and
$A(\varphi \circ f) = \varphi$ for all $\varphi\in C_b(\C)$. 

If $\nu$ is a probability measure on $\C$ and  $\varphi\in C_b(\C)$, then by dominated convergence
$$\int_\C \varphi\, d(A^*)^n(\nu)=\int_\C A^n(\varphi)d\nu\stackrel{n\to\infty}\longrightarrow\int_\C \mu(\varphi) d\nu=\mu(\varphi),$$ 
thus the sequence of probability measures $((A^*)^n(\nu))$ converges weakly to $\mu$.

If $\varphi, \psi\in C_b(\C)$, then by dominated convergence
$$
\int_\C (\varphi \circ f^n)  \psi \, d\mu  = \int_\C A^n \left((\varphi \circ f^n)  \psi\right) \, d\mu
=  \int_\C \varphi  A^n(\psi) \, d\mu  \stackrel{n \rightarrow \infty}{\longrightarrow} \int_\C \varphi \, d\mu  \int_\C \psi \, d\mu,
$$
where the second equality follows from the definition of the transfer operator $A$. Hence the measure $\mu$ is mixing for $f$.

We now show that the support of $\mu$ must be contained in the Julia set. Let $z_0$ lie in the Fatou set of $f$. If $r>0$ is small enough such that $D(z_0,r)$ is contained in the Fatou set, then all forward images $f^n(D(z_0,r))$ are contained in the Fatou set. Let $w\in J(f)$. 
Since $\mu$ is the weak limit of the sequence $Q^n_w=\sum_{z\in f^{-n}(w)}Q^n(z)\delta_z$, it follows that $\mu(D(z_0,r))=0$.


To show that ${\rm Supp}(\mu)=J(f)$ it is therefore sufficient to prove that an arbitrary point $z_0$ in the Julia set must be contained in the support, i.e. the measure of any disk $D(z_0,r)$ with $r>0$ must be positive. 
Since $f$ has no exceptional points, it is irreducible.
Let $p$ be the repelling fixed point given by the definition of irreducibility.
 Since $z_0\in J(f)$,  the family  $(f^n: D(z_0,r)\rightarrow \mathbb C)$ is not normal, and thus its image cannot  avoid both $p$ and a (distinct) preimage of $p$.   It follows that there exists $n_0\geq 0$ such that $p \in f^{n_0}(D(z_0,r))$. 
 Let $K\subset \C$ be a compact set such that $\mu(K)>0$.
 By irreducibility there exists an $n_1 \ge n_0$ such that $K \subset f^{n_1}(D(z_0,r))$.
Since $Q^{n_1}(z)>0$ for all $z\in \C$,
 it follows that  
     $$ \mu(D(z_0,r))=\int_\C Q^{n_1}_w(D(z_0,r))d\mu(w)>0.$$

\end{proof}

\begin{corollary}
 Assume $f$ has no exceptional points.
  Let $Q$ be a weakly continuous weight function that is  normal, tight, and nowhere zero. Then 
  the integer $n_0$ in the definition of tightness can be chosen uniform on compact subsets of $\C$, that is, for every $\epsilon>0$ there exists a compact subset $K\subset \C$ such that for every compact subset $L\subset \C$ there exists $n_0\geq 0$ such that for all $n\geq n_0$ we have that $Q^n_w(K)\geq 1-\epsilon.$

\end{corollary}
\begin{proof}
Let $\epsilon >0$  and let  $H\subset \C$ be the compact given by definition of tightness for the constant $\epsilon/2$. 
Let $K\subset\C$ be a compact subset which contains $H$ in its interior. Let $\varphi\colon \C\to [0,1]$  be a continuous function which is  identically equal to 1 on $H$ and identically 0 in $\C\setminus K$.
By Theorem \ref{thm: non-uniform Lyubich} we have that the sequence $A^n(\varphi)$ converges locally uniformly to a constant function $\varphi_\infty$.  Notice that $\varphi_\infty\geq 1-\epsilon/2$, indeed if $w\in \C$ we have, for $n$ large enough,
$$A^n(\varphi)(w)=\int_\C\varphi\, dQ^n_w\geq Q^n_w(H)\geq 1-\epsilon/2.$$
Let $L\subset \C$ be a compact subset and let $n_0\geq 0$ large enough such that for all $n\geq n_0$ we have
$$\|A^n(\varphi)-\varphi_\infty\|_L\leq \frac{\epsilon}{2}.$$ For every $w\in L$ and for every $n\geq n_0$ we have that 
$$1-\epsilon\leq A^n(\varphi)(w)=\int_\C\varphi\, dQ^n_w\leq \sum_{z\in f^{-n}(w)\cap K}Q^n(z)=Q^n_w(K).$$
\end{proof}

%
%
%

If the weight function $Q$ is allowed to have zeros, the previous proof can be easily adapted by working in the space $J_Q(f)$ instead of $\C$, yielding the following result.
\begin{thm}[Weight functions with zeros]

  Let $Q$ be a weakly continuous weight function which is normal, tight, and irreducible. 
Then for every $\phi \in C_b(\C)$ the sequence of functions $(A^n (\phi))$ converges locally uniformly on $J_Q(f)$ to a  function $\varphi_\infty$ , and the functional in $(C_b(J_Q(f)))^*$ defined by
$ \varphi\mapsto \varphi_\infty$ is a Borel probability measure $\mu$ on $J_Q(f)$ with full support.

Moreover, the probability measure $\mu$ is $A^*$-invariant (and thus $f$-invariant) and mixing,  and if $\nu$ is an arbitrary Borel probability measure on $J_Q(f)$, then the sequence of probability measures  
$((A^*)^n(\nu))$ converges weakly to $\mu$.  In particular, for all $w\in J_Q(f)$ we have that $(Q^n_w)$ converges  weakly to $\mu$.

\end{thm}

\subsection{Elementary strongly polynomial-like maps.} We now introduce the family of elementary strongly polynomial-like maps that we will study in the remaining of the section.
Let $f\colon \C\to \C$ be a transcendental function. After conjugating with an affine function we may assume that the first order coefficient vanishes, and  that the second coefficient equals $1$, giving a power series expansion of the form
$$
f(z) = a_0 + z^2 + a_3 z^4 + \ldots.
$$
We say that $f$ is an \emph{elementary strongly polynomial-like map} if  none of the coefficients vanish, and  their norms decrease sufficiently rapidly, i.e. given the coefficients $a_0, a_3, \ldots, a_d$, we may assume that norm of the coefficient $a_{d+1}$ is (non-zero and)  sufficiently small. 

Contrary to our treatment of the Baker--Bishop maps, we do not assume that the coefficients are real, nor do we make any assumptions on the critical orbits, which may escape, may be recurrent, and may even be periodic. The dynamical behavior of elementary strongly polynomial-like  maps is therefore much more varied than that of Baker--Bishop maps. For example, there are no restrictions on the multipliers of the periodic orbits that  may occur. 

As a consequence, we  will not aim to obtain a symbolic description of the dynamics, nor construct the ergodic measures via an ergodic Markov process. Instead, we will construct the invariant measure via weighted equidistribution, similarly to the construction of the measure of maximal entropy for rational functions due to Lyubich \cite{Lyubich83}. However, the inspiration for the choice of weights still comes from a symbolic dynamical system: we will associate to our map $f$ a Baker--Bishop map $g$, and the weights of the equidistribution will be defined as (a continuous interpolation of) the weights of the map $g$, that is the transition probabilities of the time-reversed Markov process $Q_{ts}$ inducing the ergodic probability measure for $g$.

The central idea in this section is that the estimates for our strongly polynomial-like map $f$ are derived from estimates for the associated Baker--Bishop map $g$.

%
We define the  domain $C_0$ exactly as we did for Baker--Bishop maps in Lemma~\ref{lem:seed}. In particular $C_0$ is bounded by two disjoint real analytic curves  $\Gamma^\pm_0:=\partial C_0^\pm$.  We note that the definition of $C_0$ only depends on the coefficient $a_0$.
 For all $d\geq 2$ we  set
 $$
\rho_{d+1}:=\frac{10\cdot2d}{2d+1}\frac{|a_{d}|}{|a_{d+1}|}.
$$
 and for all $d\geq 3$ we define the \emph{special annulus}  $$A_d := \{\rho_{d}/10<|z|<10\rho_{d}\}.$$
  For all $d\geq 3$, let $\gamma_d, \Gamma_d$ denote respectively the interior and exterior components of the  boundary of the annulus $A_d$.
 Notice that, in contrast with what we did for Baker--Bishop maps, we do not define the family of annuli $(C_q)_{q\geq 1}$.
 
Let us define  
$$
R_2:=\D_{\frac{\rho_3}{10}}\setminus\overline{ C_0},
$$
 and, for all $d\geq 3$,
$$
R_d:=\{10\rho_{d}< |z|< \frac{\rho_{d+1}}{10}\}.
$$
The domains $R_d$ for $d\geq2$ are called the   \emph{degree-$d$  region}.
Finally, we denote $D_2:=C_0$ and for all $d\geq 3$ we set
 $$
    D_d := \{z \; : \; |z| < 10 \rho_d \}.
    $$


\begin{remark}
    By letting the norms of the coefficients $(a_j)$ decrease sufficiently fast, we may assume that the radii $\rho_d$ are increasing, and in fact, increase arbitrarily rapidly.
\end{remark}
 The proofs of the following results are very similar to the proofs of the analogous results for Baker--Bishop maps, so we omit them.

\begin{lemma}\label{lem:requirements on coefficients}
    By letting the norms of the coefficients $(a_j)$ decrease sufficiently fast, we may assume that for each $d \ge 3$ and for $z$
 in $A_d$ the dominating terms  are $a_{d-1}z^{d-1}$ and $a_dz^d$.
    \begin{enumerate}
        \item[$\bullet$] \emph{Critical points:} The critical points of $f$ are   $x_1=0$ 
        and  points $x_d$ for $d \ge 2$, where
        $$
       x_d \sim \frac{- d a_{d}}{(d+1) a_{d+1}}
        $$
        In particular $|x_d| < \rho_{d+1}/10$,  that is, $x_d\in R_{d}$, close to the inner boundary of $A_{d+1}$.
        All critical points have multiplicity 2, that is $f'(x_d)=0$ and $f''(x_d)\neq 0$.
         \item[$\bullet$]\emph{Winding numbers:}  for all $d\geq 3$
  the curves $f(\gamma_d)$ and $f(\Gamma_d)$   are both contained in the region $R_d$ and have  winding number $d-1$  and $d$ respectively around any point in $D_d$.
        \item[$\bullet$]\emph{Images:}
  For all $d\geq 2$ we have $$f(R_d)\subset R_d\cup \overline{A_{d+1}}\cup R_{d+1},$$ and for all $d\geq 3$ we have
$$f(A_d)\subset \overline{D_d}\cup R_d.$$
Moreover, $$f(C_0)\subset \overline{C_0}\cup R_2.$$
        \item[$\bullet$]\emph{Critical values:} The critical values $v_d := f(x_d)$ satisfy
     $v_1=a_0\in R_2$ and 
$$10\rho_{d+1}+2<|v_d|<|x_{d+1}|-2$$ for all $d\geq 3$, thus $v_d\in R_{d+1}$.
    \end{enumerate}
\end{lemma}

\begin{corollary}\label{cor: delta in annulus}
    For each $d \ge 3$ there exists a unique simply connected domain $\Delta_d$ contained in the annulus
    $A_d$ which is mapped univalently by $f$ onto the disk $D_d$.
   All points in $A_d \setminus \Delta_d$ are mapped in the complement of $D_d$.
\end{corollary}

\begin{remark}
Notice that there are no Picard exceptional values, so that every point has infinitely many preimages. Also,  there are no asymptotic values, hence the set of singular values is the set of critical values $(f(x_d))$, for $d\geq1$.

\end{remark}

\begin{definition}
Let $d\geq 2$.
 For all $z\in \overline {R_d}$ we define $k(z)$ (also denoted $k_f(z)$) as the minimum integer such that $|f^k(z)|\geq  \frac{\rho_{d+1}}{10}.$
 \end{definition}

\begin{prop}[Almost uniform escape time]\label{prop:k_d}
If the norm of the coefficients decrease sufficiently fast, 
then 
\begin{equation}\label{diff1}
\max_{z\in  \overline{R_d}}\{k(z)\}\leq \min_{z\in \overline{D_d}}\{k(z)\}+1<+\infty.
\end{equation}

\end{prop}

\begin{proof} 
The argument is analogous to the argument in   Section \ref{sec: Bishop maps}, so we only briefly sketch it.
Assume that we have chosen the coefficients $a_0,a_3,\dots, a_d$ in such a way that  \eqref{diff1}, denoting $p_d(z)=a_0+a_3z^3+\dots +a_{d}z^{d}$,
$$\max_{z\in  \overline{R_d}}\{k_{p_d}(z)\}\leq \min_{z\in  \overline{D_d}}\{k_{p_d}(z)\}+1<+\infty, \quad \forall \,2\leq j\leq d.$$

Clearly by choosing $|a_{d+1}|$ small enough we can ensure
that, for all $2\leq j\leq d$,  $$\max_{z\in  \overline{R_j}}\{k_{p_{d+1}}(z)\}=\max_{z\in  \overline{R_j}}\{k_{p_{d}}(z)\},\quad \quad \min_{z\in  \overline{D_d}}\{k_{p_{d+1}}(z)\}= \min_{z\in  \overline{D_d}}\{k_{p_d}(z)\}.$$

Inductively define a sequence of radii $(r_q)_{q\geq 1}$ setting
$r_1:=\rho_d$, and $r_{q+1}=|a_d|r_{q}^d$ for all $q\geq 1$. We have $C_1=A_d$.
For all $q\geq 1$ consider the annuli  $C_q:=\{ \frac{r_q}{10}<|z|<10 r_q\}$ and the gaps $G_q:=\{ 10 r_q<|z|<\frac{r_{q+1}}{10}\}$.
Then $p_d$ sends each gap $G_q$ compactly inside the next gap $G_{q+1}$ and  every annulus $C_q$  covers exactly $d$ times the next annulus $C_{q+1}$, sending   the inner boundary of $C_q$ in the gap $G_q$ and the outer boundary in the gap $G_{q+1}$.
If  $|a_{d+1}|$ is small enough,  then arguing as in Section \ref{sec: Bishop maps} we get  
$$\max_{z\in  \overline{R_{d+1}}}\{k_{p_{d+1}}(z)\}\leq \min_{z\in  \overline{D_{d+1}}}\{k_{p_{d+1}}(z)\}+1<+\infty.$$
\end{proof}

\begin{definition}[Escape time from the $d$-degree region]\label{def:escape time} We define 
 $k_2:=\max_{z\in \overline{R_2}}k(z),$ and for all $d\geq 3$,
 $$k_d:=1+\max_{z\in \overline{R_d}}k(z).$$
\end{definition}

The role of the constants $k_d$ is similar to the role that these constants play for the Baker--Bishop-maps, except that it does not take exactly $k_d$ steps to pass from the degree $d$ region to the degree $d+1$ region, but it could take a slightly shorter time. However, the difference between the maximal number of steps and the minimal number of steps is bounded.
A key idea in what follows is that by letting the integers $k_d$ be sufficiently large, the impact of this bounded difference can be made as small as we wish.

\begin{definition}[Associated Baker--Bishop map]\label{defassBB}
If the coefficients $(a_d)$ of $f$ decrease fast enough, then the sequence $(k_d)$ grows fast enough to ensure  the existence (see Remark \ref{importantremark} (1)) of a Baker--Bishop map $g\colon \C\to \C$ such that, for all $d\geq 2$,
$$k_d=q_{d+1}-q_d.$$ 
Although $g$ is not unique, we chose one, and we say that the Baker--Bishop map $g$ is {\sl associated} with $f$.

\end{definition}
\subsection{Defining the weight function}\label{sec: inverse probabilities}

 Let $(P_{st})$  be the model Markov process  defined by   Proposition \ref{prop:existence}  for the sequence of escape times $(k_d)$ associated with the map $f$.  For Baker--Bishop maps we observed that the  transition probabilities $(Q_{ts})$ of the time-reversed Markov process
 can be used as weights to obtain weighted equidistribution of preimages.
The central idea in this section is that we can  use the same transition probabilities $(Q_{ts})$  to define   weights  for the function $f$, obtaining  weighted equidistribution of preimages even though we do not establish a conjugacy between $f$ and the symbolic dynamical system associated to the Markov process. The fact that we can still use the probabilities $(Q_{ts})$ without defining the  cells depends on Lemma \ref{computationtrans} which shows that $Q_{(q+1,j)(q,i)}$
 is independent of $i,j$, and only depends on 
 the integer $d$ such that $(q,i)\in S_d$.
 
  Recall  that the complex plane is partitioned as follows
$$
\C=\overline{C_0}\cup\bigcup_{d\geq3} \overline {A_d}\cup \bigcup_{d\geq2} R_d.
$$
Taking inspiration from the weights of the associated Baker--Bishop map,  we now define the weights of the preimages of points via the map $f$.
Let $c_2 , c_3, \ldots$ be a sequence of strictly positive real numbers satisfying
$$
\sum_{d\geq 2} c_d = \frac{1}{2},\quad \quad \sum_{d\geq 2} c_d \log d=+\infty.
$$ 
Recall the definition of the sequence $(c'_d)_{d\geq 2}$ given in Proposition \ref{prop:existence}.
\begin{definition}[Weight function]\label{def:weights summary}\, Let $w\in\C$ and $z\in f^{-1}(w)$.  In the sequel, preimages are counted with multiplicity. 
\begin{enumerate}
\item If $w \in C_0$, then $w$ has two preimages in $C_0$ and a preimage in $ A_d$ for each  $d \ge 3$. We define
\begin{align*}
Q(w,z)&=\frac{1}{2}-\frac{c_2}{k_2} &\text{if $z\in C_0$,}\\
Q(w,z)&=\alpha_d=\frac{2c'_{d-1}}{k_{d-1}}-\frac{2c'_d}{k_d}   &\text{if $z\in A_d$, $d\geq3$.}
\end{align*} 
 \item The map $f|_{R_{2}\cup C_0}$ covers $2$ times the region $R_2$, thus
if  $w \in R_2$, 
it has  $2$ preimages in $R_2\cup C_0$. Moreover $w$ has  
 one preimage in each  annulus $A_D$ for $D> 2$.
 We define
\begin{align*}
Q(w,z)&=\frac{1}{2} &\text{if $z\in R_2\cup C_0$},\\
Q(w,z)&=0  &\text{ if $z\in A_D$, $D> 2$}.
\end{align*}  
 \item 
 Let $d\geq3$. The map $f|_{R_{d-1}}$ covers $d-1$ times the special annulus $A_d$. Thus
 if $w\in A_d$, then $w$ has  $d-1$  preimages    in $R_{d-1}$. Moreover $w$ has  one preimage in each annulus $A_D$ for $D\geq d$. 
We define
\begin{align*}
Q(w,z)&=\frac{1}{d-1} &\text{if $z\in R_{d-1}$,}\\
Q(w,z)&=0  &\text{ if $z\in A_D$, $D\geq d$}.
\end{align*}  
\item 
Let $d\geq3$. The map $f|_{R_{d-1}\cup A_d \cup R_d}$ covers $d$ times the region $R_d$. Thus
if $w \in R_d$,  then $w$ has $d$ preimages   in $R_{d-1}\cup A_d\cup R_d$. Moreover, $w$ has 
 one preimage in each  annulus $A_D$ for $D> d$.
We define
\begin{align*}
Q(w,z)&=\frac{1}{d} &\text{if $z\in R_d\cup A_d$},\\
Q(w,z)&=0  &\text{ if $z\in A_D$, $D> d$}.
\end{align*}  

 \end{enumerate}
\end{definition}

The weights given with the preimages of $w$ are therefore  constant, except on countably many curves:   the outer boundaries of each annulus $A_d$ (where there is a sudden jump in the weights from $1/(d-1)$ to $1/d$,  and one  inverse image receives nonzero weight), and on the boundary of $C_0$, where the weights jump from two preimages of weight $1/2$ outside, to infinity many preimages of points inside that all receive positive weights. In order to prove equidistribution we will modify these weights near each of the curves where the discontinuities occur.

\subsection{Continuous interpolation of the weight function} \label{subsection: smoothening}

 We now show that the weight function can be made continuous by interpolation.
Notice that  the only singular values of $f$  are critical values, and each of them  has exactly one critical preimage and infinitely many regular preimages.
Recall that  for any simply connected open set  $D\subset \C$  which does not contain singular values for $f$ we have that its preimage  $f^{-1}(D)$  consists of infinitely many simply connected domains $U_i$ which are mapped conformally to $D$ under $f$, that is,  there exists a countable number of conformal maps $z_i: D\ra U_i$ such that $f\circ z_i=\Id$ on $D$. 

Let  $d\geq 3$. 
By Lemma~\ref{lem:requirements on coefficients},    the  distance between the critical value $v_d$ and  the outer boundary of $A_d$ is strictly more than $2$.  Arguing as in Corollary~\ref{cor: delta in annulus} one shows that  then there exists
a holomorphic disk   $\widetilde \Delta_d\Subset A_d$  such that $f\colon \widetilde \Delta_d\to \D_{10\rho_d+2}$ is a biholomorphism which extends as a homeomorphism of the closure.

Define  
$$
B_d:=\{10\rho_d<|z|<10\rho_d+1\}\subset R_d.
$$
For the case $d=2$, we define $B_2$ as the subset of $C_0$ of points whose distance from $\partial C_0$ is strictly smaller than 1. Clearly $B_2$ is homeomorphic to the disjoint union of two annuli.

 We now modify the weight function $Q(w,z)$ when $w\in \bigcup_{d\geq 2}\overline{B_d}$. Fix $d\geq 3$.
For all $w\in \overline B_d$ there are exactly  $d-1$ distinct  preimages $z_1,\dots, z_{d-1}$ in  $\D_{\rho_d/10}$  for which $Q(w_0,z_i)=\frac{1}{d-1}$, and one preimage $z_d$ in $\widetilde \Delta_d$ for which $Q(w_0,z_d)=0$. (We ignore all other preimages.) 
For $w\in \overline B_d$ let $$t(w):=|w|-10\rho_d\in [0,1],$$ 
 then we redefine the weights setting
\begin{align}\label{eq:weights on Bn1}
Q(w,z_i)&= (1-t)\frac{1}{d-1}+t\frac{1}{d}  \text{\ \ \ for $i=1,\ldots ,d-1$,}\\
Q(w,z_d)&=t\frac{1}{d}.
\end{align}
This clearly coincides with the weights from Definition~\ref{def:weights summary} for $w$ in the outer boundary of $B_d$ as well as for points in the outer boundary of $A_d$.
Notice the weight of the preimages $z_1,\dots, z_{d-1}$ is the same, in particular it is independent on how we number them.

 Finally, we modify the weights on $\overline{B_2}$. For all $w\in \overline B_2$ there are exactly  $2$ distinct  preimages $z_1,z_2$ in  $C_0$  for which $Q(w,z_i)=\frac{1}{2}-\frac{c_2}{k_2}$, and for all $d\geq 3$ one preimage $z_d$ in $\widetilde \Delta_d$ for which $Q(w,z_d)=\alpha_d$. 
For $w\in \overline B_2$ let $$t(w):=d_{\partial C_0}(w)\in[0,1],$$ where  $d_{\partial C_0}$ denotes the distance from the boundary of $C_0$.
We redefine the weights by 
\begin{align}\label{eq:weights on B2}
Q(w,z_i)&= (1-t)\frac{1}{2}+t(\frac{1}{2}-\frac{c_2}{k_2})=1-\frac{tc_2}{k_2} \text{\ \ \ for $i=1,2$,}\\
Q(w,z_d)&=t\alpha_d\text{\ \ \ for $d\geq 3$.}
\end{align}


\begin{lemma}\label{lemma: smoothened weights}
 Let $Q$ be the interpolated weight function.
Let  $w_0$ be a regular value, and let  $r>0$ be small enough (in particular, $D(w_0,r)$ does not contain critical values). 
Let $(z_i)_{i\in I}$ be the family of the inverse branches of $f$ on $D(w_0,r)$.
Then for  all $w\in D(w_0,r)$ we have
\begin{equation}\label{eq: Lipshitz weights}
\sum_{i\in I}|Q(w,z_i(w)) - Q(w_0,z_i(w_0))| \le |w - w_0|.
\end{equation}
\end{lemma}

\begin{proof}
We only have to prove the lemma for points in $\cup_{d\geq 2}\ov B_d$.  
 Let $d\geq 3$, and let $w_0\in B_d$.
 Let $r>0$ be small enough so that $D(w_0,r)\subset B_d$.
 Let $z_d$ be the branch of $f^{-1}$ with image in $\widetilde \Delta_d$, and let $z_1,\dots, z_{d-1}$ be the branches with image in $\D_{\rho_d/10}$.
 Then $ |t(w)-t(w_0)|\leq |w-w_0|$, and thus for all $i=1,\ldots ,d$ we have
 $$|Q(w_0,z_i(w_0))-Q(w,z_i(w))|\leq \frac{1}{d}|w-w_0|,$$ which yields the result.
 The other cases are similar. 
\end{proof}

\begin{corollary}\label{prop: smoothened weights2}
The interpolated weight function $Q\colon \C\to [0,1]$ is continuous, and thus weakly continuous.
\end{corollary}
\begin{proof}
The  function $Q\colon \C\to [0,1]$ is locally constant except at points $z\in f^{-1}(\bigcup_{d\geq 2}\overline{B_d})$, in which case $f(z)$ is a regular value and thus continuity follows  from the previous lemma
\end{proof}

\begin{remark}\label{remarkconnect}
Let $d\geq 3$. Let $z_0, z_1,\dots, z_n$ be a piece of backward orbit, that is $f(z_{j+1})=z_j$ for all $0\leq j<n$. Assume that this backward orbit has positive weight, that is
$Q(z_j,z_{j+1})>0$ for all   $0\leq j<n$. It follows from Proposition \ref{prop:k_d} that if  $z_0, z_1,\dots, z_n$ is contained in $\{10\rho_{d}< |z|\leq 10\rho_{d+1}\}$, then 
$n\leq k_d-1$. 
Similarly, if $z_0, z_1,\dots, z_n$ is contained in $\{10\rho_{d}< |z|\leq 10\rho_{i}\}$ with $i\geq d+1$, then
$$n\leq  k_d+\dots +k_{i-1}-1.$$

Similarly, if $d=2$ and $z_0, z_1,\dots, z_n$ is a piece of backward orbit with positive weight contained in $\{|z|\leq 10\rho_i\}\setminus C_0$, where $i\geq 3$, then
$$n\leq  k_2+\dots +k_{i-1}.$$
\end{remark}

\subsection{Tightness of the weight function}

From now on $Q$ will denote the interpolated weight function. Before treating tightness of $Q$ we observe that $Q$ is irreducible:

\begin{lemma}[Irreducibility of $Q$]
There exists a repelling  fixed point $p$  with  $0<Q(p)<1$ such that for every $r>0$ and  every compact $K \subset \mathbb C$ there exists an $n \in \mathbb N$ such that for every $w \in K$ there exists $z \in D(p,r)$ with $f^n(z) = w$ and $Q^n(z) > 0$,
\end{lemma}
\begin{proof}
Let $p$ be one of the repelling fixed points in $C_0$, and let $r>0$ be sufficiently small. Since the exceptional set of $f$ is empty, the domains $f^n(D(p,r))$ form a nested sequence of increasing domains whose union equals the entire complex plane. The claim does not follow immediately since $Q$ may vanish along some backwards orbits.

We now construct recursively an increasing sequence of domains $(U_n)_{n\geq 0}$ such that 
\begin{enumerate}
\item $U_0=D(p,r)$,
\item $f(U_n)\supset U_{n+1}$ for all $n\geq 0$,
\item $Q(f(z),z)>0$ if $z\in U_n$ and $f(z)\in U_{n+1}\setminus U_n$,
\item $\bigcup_n U_n=\C$,
\end{enumerate}
and from this the result immediately follows  noticing that $f(U_0)$ also contains $U_0$ and $Q(f(z), z)$ is strictly positive when $z$ and $f(z)$ both lie in $C_0$.

To ensure condition (3), it is enough to show that if $U_n$ intersects an annulus $A_d$, then it contains the whole disc $D(0,\rho_d/10)$.
Let $N \in \mathbb N$ be the smallest integer for which $f^N(D(p,r))$ contains an open disk centered at $0$ containing $C_0$.
Set $U_0=D(p,r)$, and define $(U_n)$ recursively with  the following rules:

\begin{itemize}
\item[(a)] $U_{n+1} = f(U_n)$ for $0\leq n < N-1$.
\item[(b)] If $n \in \mathbb N$ is the first integer for which $f(U_n)$ contains the open disk $D(0,\rho_d/10)$ for some $d \ge 2$, then $U_{n+1} = D(0,\rho_d/10)$.
\item[(c)] In all other cases $U_{n+1}$ is the largest open disk centered at the origin that is contained in $f(U_n)$. 
\end{itemize}

 \end{proof}

\begin{prop}\label{prop:tightness}
The weight function $Q$ is tight.
\end{prop}

\begin{proof} 
The idea is to compare the measure of the complement of $\overline D_d$ with the corresponding measure for an associated Baker--Bishop map $g$ (Definition \ref{defassBB}).
 Let $(P_{st})$  be the Markov process defined by   Proposition \ref{prop:existence}  for the sequence $k_2,k_3,\dots$ associated with the map $f$, and  let $\nu$ be its unique  stationary measure. Let $(Q_{ts})$ be the time-reversed Markov process, and recall that its transition probabilities are computed in  Lemma \ref{computationtrans}. 
  Notice that $\nu$ is a stationary measure also for $(Q_{ts})$.
 For $d\geq 3$ recall that   $\alpha_d$ denotes the transition probability from one of the two base-states to the special state $(q_d,1)$, that is
$$
 \alpha_d:=Q_{(0,1)(q_d,1)}=Q_{(0,2)(q_d,1)}.
 $$

For all $d\geq 3$ let $\widetilde S_d$ be the subset of the state space $S$ defined as
$$
\widetilde S_d := \bigcup_{q_d<q\leq q_{d+1}}\mathscr{C}_q.
$$
We denote $\widetilde S_2$ the subset $\widetilde S_d := \bigcup_{1\leq q\leq q_{3}}\mathscr{C}_q.$

Notice that $S_0, \widetilde S_2, \widetilde S_3,\widetilde S_4,\dots$ is a partition of $S$, and that  every column in $\widetilde S_d$ has the same mass.
 By stationarity of the measure $\nu$ we have that
 $$\nu(\widetilde S_2)=\nu(S_0)(k_{2}+1)\sum_{j\geq 3}\alpha_{j},$$ and for all $i\geq 3$,
$$\nu(\widetilde S_i)= \nu(S_0)k_{i}\sum_{j\geq i+1}\alpha_{j}.$$ 
Recall that  $\nu(\cup_{i=2}^{\infty} \widetilde S_i)=1-\nu(S_0)=1/2$. Hence 
\begin{equation}\label{importantexpansion}1=\frac{\nu(\cup_{i=2}^{\infty} \widetilde S_i)}{\nu(S_0)}= \sum_{i\geq 3}(1+\sum_{j=2}^{i-1}k_j)\alpha_i,
\end{equation}
and if $d\geq 3$,
$$\frac{\nu(\cup_{i=d}^{\infty} \widetilde S_i)}{\nu(S_0)}=  \sum_{i\geq d+1}(\sum_{j=d}^{i-1}k_j)\alpha_i.$$

For all $d\geq 3$ let   $K_d:=\overline{D_{d}}=\{z : |z| \leq  10\rho_{d} \}$. 
Let $d\geq 3$ large enough such that $w_0\in K_d$.
We claim that, for all $n\in\N$,
  \begin{equation}
\label{eq:weights of complement comparison}
Q^n_{w_0}(\C\setminus K_{d})\leq \frac{\nu(\cup_{i=d}^{\infty} \widetilde S_i)}{\nu(S_0)}.
\end{equation}
Notice that the result immediately follows from the claim since $\nu(\cup_{i=d}^{\infty} \widetilde S_i)\to 0$ as $d\to+\infty$.
Now we estimate $Q^n_{w_0}(\C\setminus K_{d})$. Notice that if $w\in C_0$ and $z\in A_i$ is such that $f(z)=w$, then by definition of the interpolation of the weights we have that $Q(w,z)\leq \alpha_i$. Let $z\in f^{-n}(w_0)$ be a point with nonzero weight $Q^n(w_0,z)$ in $\C\setminus K_d$.  Let $m\geq 1$ be the smallest integer such that $f^m(z)\in C_0$ (notice that such an $m$ exists since $w_0\in K_d$). Then clearly $f^{m-1}(z)$ belongs to some special annulus $A_i$ with $i\geq d+1$.
 Since the weight of $z$ is nonzero, it follows that the piece of orbit
$$z, f(z), \dots, f^{m-1}(z)$$ is contained in  $\{ 10\rho_d < |z|\leq 10\rho_i\}$. By Proposition \ref{prop:k_d} (and Remark \ref{remarkconnect}) we have that 
$m\leq \sum_{j=d}^{i-1}k_j.$
Denote the integer $i$ by $\gamma(z)$. For all $i\geq d+1$ we define $\Gamma_i$ as the subset of  points $z$ in $f^{-n}(w_0)\cap (\C\setminus K_d)$ with nonzero weight $Q^n(w_0,z)$ which satisfy $\gamma(z)=i$.   Then we have
$$Q^n_{w_0}(\Gamma_i)\leq ( \sum_{j=d}^{i-1}k_j)\alpha_i,$$ and thus
$$Q^n_{w_0}(\C\setminus K_d)=\sum_{i\geq d+1}Q^n_{w_0}(\Gamma_i)\leq\sum_{i\geq d+1} (\sum_{j=d}^{i-1}k_j)\alpha_i\leq \frac{\nu(\cup_{i=d}^{\infty} \widetilde S_i)}{\nu(S_0)}.$$
Hence we have proved that for every $\epsilon>0$ there exists a compact subset $K_d$ such that 
$Q_{w_0}^n(K_d)\geq 1-\epsilon$ for all $w_0\in K_d$ and for all $n\geq 0$.
Assume now that $w_0\in \C\setminus K_d$. For every backward orbit with positive weight $(z_n^{(i)})_{n\geq 0}$ starting at $w_0$ let $n_i\geq 0$ be the smallest integer such that $w_i:=z_{n_i}^{(i)}\in K_d$.
 Notice that $n_i$ is uniformly bounded by an integer $N$.
Then if $n\geq N$, $$Q^n_{w_0}(K_d)=\sum_iQ^{n_i}(w_i)\cdot Q^{n-n_i}_{w_i}(K_d)\geq 1-\epsilon,$$ since $\sum_iQ^{n_i}(w_i)=1$ and 
$Q^{n-n_i}_{w_i}(K_d)\geq 1-\epsilon$ for all $i$.
\end{proof}

\begin{remark}
In the previous proof, the quantity $\nu(\cup_{i=d}^{\infty} \widetilde S_i)$ can be interpreted in terms of the associated Baker--Bishop map as follows. Denote the $g$-invariant measure $\varphi_*\mathbb{P}^\nu$ constructed in Section \ref{sec: Bishop maps} by $\mu_g$.
Then $$\nu(\cup_{i=d}^{\infty} \widetilde S_i)=\mu_g(K_d(g)),$$
where $K_d(g)$ is the subset of $\C$ corresponding to $K_d$ for the map $g$, that is
$K_d(g):=\{z : |z| \leq  10\rho_{d}(g) \}$.

\end{remark}

\subsection{Normality of the weight function}

In this subsection we prove the following result.
\begin{theorem}\label{almostperiodictheorem}
The weight function $Q$ is normal.
\end{theorem}
%
%
The first  step is  to prove that for any $w\in\C$ and any sufficiently small disk  $D$ centered at $w$,  most preimages of $D$ under $f^n$  are regular (see Proposition \ref{lemma:inverse branches}). Here most refers not to the number of preimages, which is  infinite, but to the weights of the preimages   of the base point $w$. In step 2  we will argue that in fact the diameter  of most preimages of $D$ shrinks exponentially fast (see Proposition~\ref{prop: exponentially shrinking}), a result that also holds for rational functions, see for example the ``Telescope Lemma'' in the paper~\cite{Prz1990} by Przytycki.

\subsubsection{Step 1. Most components are regular preimages.}


 \begin{prop}\label{lemma: uniform diameter bound}
 Assume that  $|a_n|$ decreases sufficiently fast. Then for every $w\in \C$ there exists $r>0$ such that for every $n\geq 1$ every connected component $U$  of $f^{-n}(D(w,r))$ satisfies 
$$\diam(U)\le 1/2.$$
\end{prop}

The proof of Proposition~\ref{lemma: uniform diameter bound} requires a  lemma.
Recall that $(x_i)_{i\geq 1}$ are the critical points of $f$ and $(v_i)=(f(x_i))$ are the critical values of $f$. 
\begin{lemma}\label{lem:expansive almost everywhere}
If   $|a_n|$ decreases sufficiently fast, we have that
\begin{equation}\label{eae}
|f'(z)|\geq 2,\quad \forall \,z\in \C\setminus( \bigcup_{i\ge 1}\D(x_i,1/6)),
\end{equation}
\end{lemma}
\begin{proof}
This is similar to the proof of Lemma \ref{lem:new critical points} and we leave it to the reader. It involves estimating the derivative  on the boundary of disks of radius 1/6 centered at the critical points  and using the minimum principle.
\end{proof}

\begin{lemma}
If   $|a_n|$ decreases sufficiently fast,  there exists a sequence $(\delta_i)_{i\geq 1}$ of strictly positive numbers such that:
\begin{enumerate} 
\item  for all $i\geq 1$ we have $$\frac{1}{2^{k_{i+1}-4}}\leq \delta_i ,$$
\item if $V\subset \C$ is a domain intersecting $f(\D(x_i,1/6))$ with $\diam V\leq  \delta_i$, then any connected component $U$ of $f^{-1}(V)$ intersecting $\D(x_i,1/6)$ has to be contained in $\D(x_i,1/4)$, and thus 
$$\diam(U)\le 1/2.$$
\end{enumerate}

\end{lemma}
\begin{proof}
We argue inductively.
Assume that we constructed $p_d=a_0+z^2+\dots+a_dz^d$  and $\delta_1,\dots,\delta_{d-2}$ in such a way that (1) holds for all $1\leq i\leq d-2$, and in such a way that for all critical points 
$x^{(d)}_1,\dots, x^{(d)}_{d-2}$ it holds that, 
if $V\subset \C$ is a domain intersecting $p_d(\D(x^{(d)}_i,1/6))$ with $\diam V\leq  \delta_i$, then any connected component of $p_d^{-1}(V)$ intersecting $\D(x^{(d)}_i,1/6)$ has to be relatively compact  in $\D(x^{(d)}_i,1/5)$.
 Let $x^{(d)}_{d-1}$ the critical point  of $p_d$ close to the inner boundary of the special annulus $A_d$. Let 
$\delta_{d-1}>0$ be such that if $V\subset \C$ is a domain intersecting $p_d(\D(x^{(d)}_{d-1},1/6))$ with $\diam V\leq \delta_{d-1}$, then any connected component of $p_d^{-1}(V)$ intersecting $\D(x^{(d)}_{d-1},1/6)$ has to be relatively compact   in $\D(x^{(d)}_{d-1},1/5)$.

Now choose $|a_{d+1}|$ small enough as to preserve the previous construction for the map $p_{d+1}$, and in such a way that  $\frac{1}{2^{k_d-4}}\leq\delta_{d-1} $.

\end{proof}

\begin{proof}[Proof of Proposition~\ref{lemma: uniform diameter bound}] 
If $w$ belongs to some $R_j$ set $r:={\rm min}\{\delta_j,1/2\}$. Otherwise set $r:=1/2$. Let $\XX:=\bigcup_{i\geq 1 } \D(x_i,1/6).$ 
Let $(w_n)$ be a backward orbit of $w$, set $W_0:=\D(w,r)$ and for all $n\geq 1$ let $W_n$ be the convex hull of  the connected component of $f^{-1}(\D(w,r))$ containing $w_n$.
Assume first that $W_n\cap f(\XX)=\varnothing $ for all $n\geq 1$. Then for every $n\geq 0$ there exists a  branch $\varphi_n$ of $f^{-1}$ defined on $W_n$ whose image contains $w_{n+1}$, and $|\varphi'_n|\leq 1/2$ on $W_n$.
Hence we get inductively that  for all $n\geq 1$ $$\diam W_n\le 1/2.$$
 
 Assume now that there exist   $n\geq 1$ such that $W_n\cap f(\XX)=\varnothing $, and denote such indexes by $n_1\leq n_2\leq n_3\dots$.
Let $i\geq 1$ and assume inductively that $$\diam W_{k}\le 1/2, \quad \forall\, k\leq n_i+1.$$ Clearly
we have $\diam W_{k}\le 1/2$ also for all $k\leq  {n_{i+1}}.$
Let $d\geq 2$ be the integer such that  $W_{n_{i+1}}\cap f(\D(x_{d-1},1/6))\neq \varnothing$.  
 Since by Proposition~\ref{prop:k_d} we have $$ \min_{z\in \overline{D_d}}\{k(z)\}\geq k_d-2,$$  it  follows that it takes the point $w_{n_{i+1}}$ at least $k_d-3$ forward iterates to exit $R_d$, that is 
 $$f^m(w_{n_{i+1}})=w_{n_{i+1}-m}\in R_d,\quad \forall\, 1\leq m\leq k_d-4.$$  
Hence  $$n_{i+1}-n_{i}+1\geq k_d-4.$$ It follows from (1) of the previous lemma that ${\rm diam}(W_{n_{i+1}})\leq \delta_{d-1}$, and then by (2) by the previous lemma we get ${\rm diam}(W_{n_{i+1}+1})\leq 1/2$.

\end{proof}

\begin{definition}[Regular preimages]
Let $w\in \C$ and let $D$ be a disc centered at $w$.
Let $n\geq 1$ and let $U$ be a connected component of $f^{-n}(D)$. We say that $U$ is  
a {\sl regular preimage} of $D$ under $f^n$ if  $f^n|_U\colon U\to D$ is a biholomorphism.
We define the {\sl weight} of $U$ as $Q^n_w(U)=Q^n(w,z),$ where $z=f^{-n}(w)\cap U$.
\end{definition}
Since $Q(w,z)\leq1/4$ for all $w,z$, a regular preimage of $D$ under $f^n$ has weight at most $1/4^n$.
 Recall that all singular values of $f$  are critical values, and that the set of singular values for $f^k$ is $S(f^k)=\bigcup_{0\leq j\leq k-1}f^j(S(f))$, where $S(f)$ is the set of singular values of $f$.

\begin{prop}\label{lemma:inverse branches}
Assume that $|a_n|$ decreases sufficiently fast. 
Then   for all $\epsilon>0$ there exists $k\geq 1$ such that, if $w\in\C$ and $r>0$ is small enough  that 
$$D(w,r) \cap S(f^k)=\emptyset,$$
 then  for all $n \in \mathbb N$ the total weight 
    of the regular preimages of $D(w,r)$ under $f^{n}$    is at least $1-\epsilon/2$.
\end{prop}

\begin{proof}
Assume that $|a_n|$ decreases sufficiently fast to ensure  $k_d\geq d+1$ for all $d\geq 2$.
By Proposition~\ref{lemma: uniform diameter bound} we can assume that, for all $n\geq1$, all connected components of $f^{-n}(D(w,r))$ have diameter bounded by 1/2. 
Recall that if   a regular preimage $U$ of $D(w,r)$ under $f^n$ contains no critical values, then 
 every connected component of $f^{-1}(U)$ is a regular preimage of  $D(w,r)$ as well (under $f^{n+1}$).  Notice also that for all $w$ and $z\in f^{-1}(w)$ we have that $Q(w,z)\leq 1/2,$ and thus by \eqref{integralweight}
the weight of  a regular preimage of $D(w,r)$ under $f^n$ is at most $1/2^n$.
Let $k$ be sufficiently large such  that  $$\sum_{m =  k}^\infty m/2^m\leq \epsilon/2.$$
For all $1\leq j\leq k$ we have that all connected components of $f^{-j}(D(w,r))$ are regular preimages. 
 Assume now that $w\in C_0$.   We claim that for any $n\geq1$, there are at most $n$ connected components of  $f^{-n}(D(w,r))$ which can contain singular values. 
Indeed, recall that for all $d\geq 3$, the critical value $v_{d-1}$ needs at least $k_d-3$ iterates to exit $R_d$.
Hence if $j\leq d-2\leq k_d-3$, then $f^{-j}(D(w,r))$ cannot contain $v_{d-1}$.
This means that  for all $n\geq 1$ the preimage $f^{-n}(D(w,r))$ can contain only $v_1, \dots, v_n$.

 Let now $n> k $.  Then the total weight 
    of the regular preimages of $D(w,r)$ under $f^{n}$    is at least $$1-\sum_{m=k}^{n-1} m/2^m\geq 1-\epsilon/2.$$
 The case $w\not\in C_0$ is similar.

\end{proof}

\subsubsection{Step 2. Most inverse branches decrease exponentially fast.}

Let us recall that the weights $Q(w,z)$ are locally constant outside the domains $B_d$. For $d \ge 3$ these are annuli of unitary width whose inner boundary coincides with the outer boundary of $A_d$, while $B_2$ is  the disjoint union of two annuli   whose outer  boundaries coincide with the outer boundaries of the two components of $C_0$. 

\begin{lemma}
By letting   $|a_n|$ decrease sufficiently fast we can ensure that
\begin{equation}\label{eq: area estimate}
\sum_{n \in \mathbb N} \frac{1}{2^n} \sum_{d \, : \, k_d \le n} \mathrm{Area} (B_d) < \infty.
\end{equation}
\end{lemma}
\begin{proof}

Observe that we can write \eqref{eq: area estimate} as
$$\sum_{n=k_2}^{k_3-1}\frac{1}{2^n}\mathrm{Area} (B_2)+ \sum_{n=k_3}^{k_4-1}\frac{1}{2^n}(\mathrm{Area} (B_2)+\mathrm{Area} (B_3))+\dots\leq 2\sum_{d\geq 2}\frac{1}{2^{k_d}}\sum_{j=2}^d\mathrm{Area} (B_j).$$
Notice that the area of the  domains $B_d$   depends only on $a_d$ and $a_{d-1}$, and that having fixed $a_0, \dots, a_d$ we can choose the next coefficient $a_{d+1}$ such that the integer $k_d$ is arbitrarily large, ensuring convergence of this series.

\end{proof}
Set $$
 \Lambda  := \bigcup_{d\ge 2} \overline{B_d}.
 $$ The next result shows that    
 most inverse branches intersecting $\Lambda$
are exponentially small.
\begin{prop}\label{prop: exponentially shrinking}
Assume that  $|a_n|$ decrease sufficiently fast.  Let $\delta>0$, let $\epsilon>0$, and let  $k\in \N$ given by  Proposition \ref{lemma:inverse branches} for $\epsilon$. If  $w \not\in S(f^k)$,
then
 there exist $r>0$ and subsets
 $\mathscr{S}_n\subset f^{-n}(w)$ such that for all $n\geq 1$:
    \begin{enumerate}
        \item the image $f(\mathscr{S}_{n+1})$ is a subset of $\mathscr{S}_n$;
        \item every $z\in  \mathscr{S}_n$ is contained in a  regular preimage $U_n(z)$ of $ D(w,r)$ under $f^n$;
        \item $Q^n_w(\mathscr{S}_n)\geq 1-\epsilon$;
        \item for each $z\in \mathscr{S}_n$, if $U_n(z)\cap\Lambda\neq\emptyset$, then the Euclidean  diameter of $U_n(z)$ is at most $\delta/\sqrt{2}^n$.
    \end{enumerate}
\end{prop}

\begin{proof} Let $s>0$ be such that   $D(w,s)\cap S(f^k)=\emptyset$.
By Proposition ~\ref{lemma:inverse branches}, for all $n\geq 1$ the total weight of the regular preimages of $D(w,s)$ under $f^n$ is at least $1 - \epsilon/2$. We will consider only these connected components of $f^{-n}(D(w,s))$ and disregard the other connected components.

   By Proposition~\ref{lemma: uniform diameter bound}, for each $d$ we may assume that the regular preimages intersecting $B_d$ must be contained in the neighborhood of $B_d$ of radius $1/2$, which we denote by $\mathcal{N}(B_d)$. We have  that   $
    \mathrm{Area} (\mathcal{N}(B_d)) < 2\mathrm{Area} (B_d).
    $
    It follows from equation~\eqref{eq: area estimate} that
    \begin{equation}\label{finiteseries}
    \sum_{n \geq 1} \frac{1}{2^n} \sum_{d \, : \, k_d \le n} \mathrm{Area} (\mathcal{N} (B_d)) < \infty.
    \end{equation}
    Let $C>0$, $n\in \mathbb N$.  Clearly  the  number of connected components of $f^{-n}(D(w,s))$  intersecting $B_1\cup \dots\cup B_d$, and whose  area is greater or equal than $C/2^n$,     is bounded by 
    $$
    \frac{2^n}{C} \cdot \sum_{j=1,\ldots,d} \mathrm{Area} (\mathcal{N}(B_j)).
    $$
       For all $n\geq 1$ let $\mathscr{B}_n$ denote the subset of $f^{-n}(w)$ of points which are contained in a regular preimage $U_n(z)$, such that $U_n(z)\cap \Lambda\neq \varnothing$ and  $\mathrm{Area} (U_n(z))\geq C/2^n$.
       We claim that if  $C=C(\epsilon)$ is sufficiently large, then $ \sum_{n\geq 1}Q_w^n(\mathscr{B}_n)\leq \epsilon/2$.
      Assume first that $w\in C_0$. Then for all $d\geq 3$ we have that $f^{-n}(D(w,s))$ can intersect $B_d$ only if $n\geq k_d$.
  Moreover  every   regular preimage of $D(w,s)$ under $f^n$ has weight at most $1/4^n$, and thus 
      \begin{equation}\label{estimate neighbor}
  \sum_{n\geq 1}Q_w^n(\mathscr{B}_n)\leq   \frac{1}{C}\sum_{n \geq 1} \frac{1}{2^n } \sum_{d \, : \, k_d \le n} \mathrm{Area} (\mathcal{N} (B_d))\leq \epsilon/2,
    \end{equation}
  where the last inequality follows using  \eqref{finiteseries} and choosing $C=C(\epsilon)$ sufficiently large.
If $w\not\in C_0$ the proof of the claim is similar.
    
   Define $\mathscr{S}_n$ as the set of all $z\in f^{-n}(w)$ contained in a regular preimage $U_n(z)$ such that, for all $0\leq j\leq n-1$, if $U_{n-j}(f^{j}(z))\cap \Lambda\neq \varnothing$, then $\mathrm{Area}(U_{n-j}(f^{j}(z)))<\frac{C}{2^{n-j}}.$
  Properties   (1) and (2) are immediate from the definition of $\mathscr{S}_n$. Property (3) follows from \eqref{estimate neighbor}.
  
   It remains to prove (4).
   Let $z\in \mathscr{S}_n$ be such that $U_n(z)\cap\Lambda\neq \varnothing$, and let $\varphi\colon D(w,s)\to U_n(z)$ be the corresponding inverse branch of $f^n$.  By Koebe's $1/4$-Theorem we have $$|\varphi'(w)|\leq \frac{4}{\sqrt{\pi}s}\sqrt{\mathrm{Area}(U_n(z))}\leq  \frac{4}{\sqrt{\pi}s}\frac{\sqrt{C}}{\sqrt{2}^n}. $$
   Let $0<r<s$. By Koebe's distortion Theorem,
   $${\rm diam}\,\varphi(D(w,r))\leq 2\frac{sr}{s^2-r^2}|\varphi'(w)|s\leq \delta/\sqrt{2}^n $$ for $r$ sufficiently small.

\end{proof}

\begin{proof}[Proof of Theorem \ref{almostperiodictheorem}]
Let $w \in \mathbb C$ and let $\epsilon>0$. We want to find $r>0$   such that 
$$
|A^n(\varphi)(w^\prime) - A^n(\varphi)(w)| < \epsilon \text{\ \ \ for $w'\in D(w,r)$,\ $n\geq 1$}.
$$
 Let $M>0$ be such that $|\varphi(z)|\leq  M$ for all $z\in \C$.

By Proposition \ref{lemma: uniform diameter bound} we can choose  $r>0$ small enough  such that for every $n \geq 1$ the diameter of every connected component of $f^{-n}(D(w,r))$ is less than $1/2$.
Let $k$ be given by Proposition \ref{lemma:inverse branches} for  $\epsilon'= \epsilon/(16M)$, and let  $\delta>0$ be such  that    $\sum_{j=0}^\infty\delta/\sqrt{2}^j\leq \epsilon/16M$. 
Let us first assume that $w$ does not lie in the critical set of $f^k$.
By Proposition  \ref{prop: exponentially shrinking}, up to taking a smaller $r$, we have that for every $n\geq 1$ the total weight of the regular preimages of $D(w,r)$ under $f^n$ which contain points of $\mathscr{S}_n$     is at least $1-\epsilon/(16M)$, and such components have diameter at most $\delta/\sqrt{2}^n$ whenever they intersect $\Lambda$. 
We will refer to such connected components of $f^{-n}(D(w,r))$ as \emph{good components}, while the other  connected components  will be called   \emph{bad components}.

  Let $w^\prime \in D(w,r)$.
By combining (4) in Proposition~\ref{prop: exponentially shrinking}   
with the Lipschitz estimate  \eqref{eq: Lipshitz weights}, it follows by induction that, for all $n\geq 0$,
\begin{equation}\label{inductiveinequality}
\sum_i |Q^n(w,z_i(w)) - Q^n(w^\prime,z_i(w^\prime) )| \leq  \sum_{j=0}^n \delta/\sqrt{2}^j.
\end{equation}
Let $K$ be the compact given by Proposition \ref{prop:tightness} for the constant $\epsilon/8M$.  Let $K'$ be the 1/2-neighborhood of the compact $K$.
Let $s>0$ be small enough such that if $x,y\in K'$ satisfy $\|x-y\|\leq s$, then $|\varphi(x)-\varphi(y)|\leq \epsilon/4$.

Let $(U_n(z_i))_{i\in I}$ be the good components of $f^{-n}(D(w,r))$, and for all $i\in I$ let $z_i'$ be the unique point in $U_n(z_i)\cap f^{-n}(w')$.
Every good component $U_n(z_i)$ is a holomorphic disc contained in a euclidean disc of radius $1/2$. By the monotonicity property of the Poincar\'e distance, we can assume that for all $i\in I$, $\|z_i-z_i'\|\leq s$  up to taking a smaller $r$  if necessary.

 We have
\begin{align*}
\left|\sum_{i\in I}Q^n(w,z_i)\varphi(z_i)-Q^n(w',z_i')\varphi(z_i')\right|&\leq\sum_{i\in I}|\varphi(z_i)||Q^n(w,z_i)-Q^n(w',z_i')|+ \sum_{i\in I}|\varphi(z_i)-\varphi(z_i')|Q^n(w,z_i)\\
&\leq \epsilon/4+\sum_{i\in I}|\varphi(z_i)-\varphi(z_i')|Q^n(w,z_i),
\end{align*}
where we used \eqref{inductiveinequality}.
Let $J\subset I$ be the set of indexes $i$ such that $z_i\in K$.
Then 
$$\sum_{i\in I}|\varphi(z_i)-\varphi(z_i')|Q^n(w,z_i)=\sum_{i\in J}|\varphi(z_i)-\varphi(z_i')|Q^n(w,z_i)+\sum_{i\in I\setminus J}|\varphi(z_i)-\varphi(z_i')|Q^n(w,z_i)\leq \epsilon/4+\epsilon/4=\epsilon/2.$$
Let $\mathscr{B}$ be the union of all bad components of $f^{-n}(D(w,r))$. We are left with estimating 
\begin{equation}\label{lastterm}\left|\sum_{z\in f^{-n}(w)\cap \mathscr{B}} Q^n(w,z)\varphi(z)-\sum_{z'\in f^{-n}(w')\cap \mathscr{B}} Q^n(w',z')\varphi(z')\right|.
\end{equation}
Since the total weight of the good components is at least $1-\epsilon/16M$ we have that $\sum_{z\in f^{-n}(w)\cap \mathscr{B}} Q^n(w,z)\leq \epsilon/16M$, and thanks to \eqref{inductiveinequality} we have that
$\sum_{z'\in f^{-n}(w')\cap \mathscr{B}} Q^n(w',z')\leq \epsilon/8M$. Hence \eqref{lastterm} is smaller that $\epsilon/4$, which proves equicontinuity.

Assume now that  $w$ is  in the critical set of $f^k$. One can reduce to the above case by a standard argument similar to the argument used for rational functions by Lyubich in \cite{Lyubich83}. By choosing $n_0$ sufficiently large we can guarantee, using the properties of the entire functions that we consider just as in the proof of Proposition \ref{lemma:inverse branches}, that the preimages $w'\in f^{-n_0}(w)$ that do not lie in the critical set of $f^k$ have weight arbitrarily close to $1$. For each such $w'$ we can therefore choose the radius of the disk such that the above argument holds. The connected components of $f^{-n_0}(D(w,r))$ that contain the points $w'$ will lie inside those disks if $r$ is sufficiently small. Thus the statement holds for arbitrary $w$.

\end{proof}


\subsection{Mass comparison with  Baker--Bishop maps}

Let $f$ be an elementary strongly polynomial-like map, and let $g$ be an  associated Baker--Bishop map. We now compare the $f$-invariant measure $\mu$ with the 
 $g$-invariant measure $\mu_g$ obtained without using Proposition \ref{prop:extension of the support}   to modify the support. (See Subsection \ref{subsection: full measure} below for the analogous result with modified supports). We denote by $C_0(g)$ and $K_d(g)$ the sets   $C_0$ and $K_d$ relative to    the map $g$.

\begin{prop}\label{lemma: measure of C_0}
\,
\[
\mu(\overline C_0) \ge \mu_g(\overline {C_0(g)})=\frac{1}{2}.
\]
\end{prop}
\begin{proof}
Let us suppose for the purpose of a contradiction  that
\begin{equation}\label{contr1}
\mu(\overline{C_0}) < \mu_g(\overline{C_0(g)}).
\end{equation}
Let $d\geq 3$, and let $H_d^{(0)}$ be the component of $f^{-1}(\overline{C_0})$ contained in the special annulus $A_d$.  Let $\varphi$ be the corresponding branch of $f^{-1}$.
Then $$\mu(H_d^{(0)})=\int_{\overline{C_0}}Q(w,\varphi(w))d\mu\leq \alpha_d \mu(\overline{C_0}).$$
For  all $j\leq 1$ define inductively 
$$H_d^{(j)}:=\{z\in f^{-1}(H_d^{(j-1)}):Q(f(z),z)>0\}\setminus C_0.$$
 By Proposition  \ref{prop:k_d} we have  $H_d^{(k)}=\varnothing$ for all $k\geq 1+k_2+\dots+k_{d-1}$.
Notice that for all $j\geq 1$ we have $\mu(H_d^{(j)})\leq \mu(H_d^{(0)})\leq \alpha_d\mu(\overline C_0).$ 
We claim that   
\begin{equation}\label{containsallmass}
\mu(\overline{C_0}\cup\bigcup_{d\geq 3}\bigcup_{k\geq 0}H_d^{(k)})=1.
\end{equation}
Let $w\in C_0$, and denote $Q^n_w:=(A^*)^n\delta_w$. Then for all $n$ and for all $j\geq 3$,
$$
Q_w^n(\bigcup_{ d\geq j}\bigcup_{k\geq 0}H_d^{(k)}))\leq \alpha_j(1+k_2+\dots+k_{j-1})+\alpha_{j+1}(1+k_2+\dots+k_j)+\dots,
$$
which tends to 0 as $j\to+\infty$ due to \eqref{importantexpansion}.
Hence for all $\varepsilon>0$ there exists $d_0$ such that for all $n$, $$Q_w^n( \overline {C_0}\cup\bigcup_{ 3\leq d\leq d_0}\bigcup_{k\geq 0}H_d^{(k)})\geq 1-\varepsilon.$$ Since 
the set $ \overline {C_0}\cup\bigcup_{ 3\leq d\leq d_0}\bigcup_{k\geq 0}H_d^{(k)}$ is closed, its $\mu$-measure equals at least  $1-\varepsilon$, which proves the claim.

But 
\begin{align*}
\mu(\bigcup_{d\geq 3}\bigcup_{k\geq 0} H_d^{(k)})\leq \sum_{d\geq 3}\alpha_d\mu(\overline{C_0})(1+k_2+\dots+k_{d-1})\leq \sum_{d\geq 3}\alpha_d\mu_g(\overline{C_0})(1+k_2+\dots+k_{d-1})=\frac{1}{2},
\end{align*}
where the last equality follows from  \eqref{importantexpansion}.
Together with \eqref{contr1}, this contradicts \eqref{containsallmass}.
\end{proof}


\subsection{Infinite entropy of the measure $\mu$}\label{sec:infinite entropy SPL}

\begin{remark}
Let $\nu$ be a probability measure on $\C$ satisfying $A^*\nu=\nu$, that is, stationary with respect to the Markov process $(Q_w)_{w\in \C}$.
Then the metric entropy of the measure $\nu$ with respect to the map $f$ is equal to the entropy rate of the Markov process  $(Q_w) $, that is
\begin{equation}\label{continuousentropy}
h_\nu(f)=-\int_\C\left(\sum_{z\in f^{-1}(w)} Q(z)\log (Q(z))  \right) d\nu(w).
\end{equation}
\end{remark}

\begin{prop}\label{prop:infinite entropy strongly poly-like}
We have that $h_{\mu}(f)=+\infty$.
\end{prop}

\begin{proof}
Let $d\geq 2$. Recall that $D_{d+1}$ is the open disk centered at the origin whose boundary coincides with the exterior boundary of the annulus $A_{d+1}$.
We consider the case $d\geq 3$, the proof for the case $d=2$ is analogous.
For each $d \ge 3$ the intersection $f^{-1}(D_{d+1}) \cap D_{d+1}$ consists of the holomorphic disk $\Delta_{d+1} \subset A_{d+1}$, which is mapped to $D_{d+1}$ univalently, and a holomorphic disk
$$
E \subset D_{d+1} \setminus A_{d+1},
$$
which is mapped to $D_{d+1}$ with degree $d$.
 We denote by $C_d(0)$ the doubly connected domain  $D_{d+1} \setminus E$. Let us denote by $C_d(1)$ its preimage in $D_{d+1}\setminus A_{d+1}$. Then $C_d(1)$ is mapped onto $C_d(0)$ and has identical mass thanks to \eqref{integralweight}. Continuing backwards we obtain  a finite family of doubly connected domains $C_d(0), C_d(1), \ldots, C_d(j)$  with the same mass which are  pairwise disjoint, all disjoint from  $D_{d-1}$, and such that  $f^{-1}(C_d(j))\cap D_{d-1}\neq \varnothing.$
 Let $z\in f^{-1}(C_d(j))\cap D_{d-1}$. Then $k(z)\leq j+2$, and thus
 $$k_d\leq \min_{z\in \overline{D_d}}\{k(z)\}+2\leq j+4,$$ that is, there are at least $k_d-3$ doubly connected domains $C_d(i)$.

 We want to estimate from below the mass of the domains $C_d(i)$. 
  By Proposition \ref{lemma: measure of C_0} we have that $m_0:=\mu(C_0)$ is strictly positive. 
  Since $A^*(\mu)=\mu$, we have that  $\mu(C_d(0))\geq m_0\alpha_{d+1} =m_0(\frac{2c'_{d}}{k_{d}}-\frac{2c'_{d+1}}{k_{d+1}}) $. Assume  that $k_{d+1}$ is chosen large enough such that $\frac{2c'_{d+1}}{k_{d+1}}\leq \frac{c'_{d}}{k_{d}}$, and thus
 $$\mu(C_d(0))\geq m_0\frac{c'_{d}}{k_{d}}.$$
 
 Notice that every point in $\bigcup_{i=0}^{j}C_d(i)$ has  $d$ distinct preimages with weight $1/d$ whenever $C_d(i)$ does not contain critical values.  Since there is only one critical value between $A_d$ and $A_{d+1}$, there are at least 
 $k_d-4$ doubly connected domains $C_d(i)$ which do not contain critical values. 
Hence using \eqref{continuousentropy} we have
$$
h_\mu(f)\geq \sum_{d\geq 2}(k_d-4)\cdot m_0\frac{c'_d}{k_d}\cdot \log d= m_0\sum_{d\geq 2}\left (\frac{k_d-4}{k_d}\right )c'_d\log d=+\infty.
$$

%
\end{proof}

\subsection{Obtaining full support}\label{subsection: full measure}

The purpose of this subsection is to modify the weight function $Q$ so that it is nowhere zero. 
We will refer to the previously introduced weight function $Q$ and the resulting measure $\mu$ as the \emph{model system}, and the system with non-zero weights as the \emph{modified system}.

Recall that for the model system the weight $Q(w,z)$ is already positive for any $w = f(z)$, unless $z$ lies in a special annulus $A_d$ and $w$ lies in $D_d \setminus C_0$. For the modified system we increase these probabilities recursively with $d$ by a constant depending only on $d$, with each consecutive increase in $Q(w,z)$ sufficiently small such that a number of properties remain to hold. When increasing $Q(w,z)$ for a given preimage $z \in f^{-1}(w)$, we decrease the weights $Q(w,z^\prime)$ of all other preimages by a constant multiplicative factor so that for given $w$ the weights $Q(w,z)$ still sum up to $1$. If the increased weights are consecutively chosen sufficiently small, then it is clear that all weights can be guaranteed to be non-zero.

In order to guarantee that the construction of the measure $\mu$ carries through, it needs to be checked that a number of properties hold. 

\medskip

{\bf Continuity of the modified weight function.}
By choosing the increased weights sufficiently small and then interpolating the weights linearly near the boundary of the domains where the weights have been changed, one can  guarantee that Lemma \ref{lemma: smoothened weights} remains to hold.

\medskip

{\bf Tightness of the modified weight function.} In Proposition \ref{prop:tightness}, tightness of the pullback measures was obtained by comparing the measures of the degree-$d$ regions for the model system with the measures of the corresponding regions for the corresponding model Baker--Bishop maps. The same argument can be applied here. Recall that for the Baker--Bishop map, we changed the transition probabilities $P_{st}$ of the corresponding Markov process slightly such that all transition probabilities $P_{st}$ corresponding to possible pairs $(z, f(z))$ become positive, yet the Markov process remains positively recurrent, i.e. the expected time for forward orbits to  return to the base state is always finite. Any choice of such transition probabilities leads to a collection of inverse transition probabilities $Q_{ts}$, which can be used as weights for the modified system given by the strongly polynomial-like maps we consider here. It follows that the probability that an inverse orbit travels from a degree $d$ region to a given larger degree region is then at least as large for the corresponding Baker--Bishop map as it is for the strongly polynomial-like map we consider.  We emphasize that in order for this inequality to point in the correct direction, it is important that the interpolation of the weights needs to be defined on the \emph{outside} of each of the special annuli.

The tightness of the sequence of pull-back measures now follows from the tightness of the measures for the corresponding Baker--Bishop maps.

\medskip

{\bf Normality of the modified weight function.} The proof of the equicontinuity of the pull-back measures $Q^n_w$ carries through exactly as in Theorem \ref{almostperiodictheorem}.

\medskip

{\bf Mass comparison with  Baker--Bishop maps.}  We claim that the inequality $\mu_f(C_0(f)) \ge \mu_g(C_0(g))$, proved in Proposition \ref{lemma: measure of C_0} for the model system, still holds for the modified system, as long as the same modifications of the weights are used for both the map $f$ and for the associated Baker--Bishop map $g$.

The argument by contradiction presented in Proposition \ref{lemma: measure of C_0} still holds in the modified setting, with slight changes due to the fact that not every allowable inverse orbit (which is now every inverse orbit) must return to the set $C_0$.  The essential step is to prove that equation \eqref{containsallmass}, that is 
$$
\mu(\overline{C_0} \cup\bigcup_{d\geq 3}\bigcup_{k\geq 0}H_d^{(k)})=1,
$$
still holds, despite the fact that for each $d\ge 3$ the sets $\bigcup_{k\geq 0}H_d^{(k)}$ are not closed, and therefore neither is their union. The proof of the equality follows that proof of Proposition \ref{lemma: measure of C_0}, now by showing that both
$$
\mu(\bigcup_{d\geq j}\bigcup_{k\geq 0}H_d^{(k)}) \rightarrow 0
$$
and for each $d \ge 3$
$$
\mu(\bigcup_{k\geq j}H_d^{(k)}) \rightarrow 0
$$
as $j \rightarrow \infty$. The latter convergence holds trivially for the model system since the sets $H_d^{(k)}$ are empty for $k$ sufficiently large, and remains to hold for the modified system when the changes in weights are sufficiently small. 

{\bf Infinite entropy.} By changing the weights sufficiently little, it can be proved that the measure of each special annulus $A_d$ changes arbitrarily little. In particular we can guarantee that the measure of each $A_d$ for the modified system is greater than half the measure of $A_d$ for the model system, uniformly over all $d$. This provides the same estimate for each degree $d$ region, which guarantees that the entropy estimate that is obtained is at least half the estimate obtained for the model system, and thus still infinite.

%

 This completes the proof of Theorem \ref{Theorem: Main1} for the class of elementary strongly polynomial-like transcendental functions.

\bibliographystyle{amsalpha}


\begin{thebibliography}{ABFP21}

\bibitem[AKM65]{AKM65}
R. L. Adler, A. G. Konheim and M. H. McAndrew, \emph{Topological entropy}, Trans. Amer. Math. Soc. {\bf 114} (1965), 309--319.

\bibitem[ABFP21]{henon3}
L. Arosio, A. M. Benini, J. E. Forn{\ae}ss, and H. Peters,
  \emph{Dynamics of transcendental {H}\'{e}non maps {III}: {I}nfinite entropy},
  J. Mod. Dyn. \textbf{17} (2021), 465--479.


\bibitem[Bak63]{Bak63}I. N. Baker,\emph{ Multiply connected domains of normality in iteration theory}, Math. Z. \textbf{81} (1963), 206--214.

\bibitem[Bak76]{Bak76}I. N. Baker, \emph{An entire function which has wandering domains}, J. Austral. Math. Soc. Ser. A, \textbf{22} (1976), 173--176.

\bibitem[BF14]{BF15}
A. M. Benini and N. Fagella, \emph{{A separation theorem for entire
  transcendental maps}}, Proceedings of the London Mathematical Society
  \textbf{110} (2014), no.~2, 291--324.

\bibitem[BFP20]{BFP1}
A. M. Benini, J. E. Forn{\ae}ss, and H. Peters, \emph{Infinite
  entropy for transcendental entire functions with an omitted value}, Acta
  Math. Vietnam. \textbf{45} (2020), no.~1, 49--52. 

\bibitem[BFP21]{BFP2}
\bysame, \emph{Entropy of transcendental entire functions}, Ergodic Theory
  Dynam. Systems \textbf{41} (2021), no.~2, 338--348. 

\bibitem[Bis18]{Bis18} C. J. Bishop,\emph{ A transcendental Julia set of dimension 1}, Invent. Math. \textbf{212} (2018), 407--460.


\bibitem[BK07]{BK07}
K. Baranski and B. Karpinska, \emph{Coding trees and boundaries of
  attracting basins for some entire maps}, Nonlinearity \textbf{20} (2007),
  391--415.

\bibitem[BLS93]{BLS93}
E. Bedford, M. Lyubich, and J. Smillie, \emph{Polynomial
  diffeomorphisms of {${\bf C}^2$}. {IV}. The measure of maximal entropy and
  laminar currents}, Invent. Math. \textbf{112} (1993), no.~1, 77--125.
  
  \bibitem[BR20]{BeniniRempe}
A. M. Benini and L. Rempe, \emph{A landing theorem for entire functions
  with bounded post-singular sets}, Geometric and Functional Analysis
  \textbf{30} (2020), no.~6, 1465--1530.
  
  \bibitem[Ber91]{BerPer}
W. Bergweiler, \emph{Periodic points of entire functions: proof of a conjecture of Baker}, Complex Var. Theory Appl. \textbf{17(1–2)}, 57--72 (1991)
  
  
  \bibitem[Bil99]{Bil99}
  P. Billingsley, \emph{Convergence of Probability Measures},
  John Wiley and Sons, Inc. (1999)
  
  \bibitem[BF64]{BF}
  D. Blackwell and D. Freedman, \emph{The tail $\sigma$-field of a Markov chain and a theorem of Orey}, Ann. Math. Stat. {\bf 35} (1964), no. 3, 1291--1295. 
  
  \bibitem[Bow71]{Bow71}
  R. Bowen, \emph{Entropy for group endomorphisms and homogeneous spaces}, Trans. Amer. Math. Soc. {\bf 153} (1971), 401--414.

\bibitem[Bro65]{Brolin}
H. Brolin, \emph{Invariant sets under iteration of rational functions}, Ark.
  Mat. \textbf{6} (1965), 103--144. 

\bibitem[CF96]{FiCh}
J.P.R. Christensen and P.~Fisher, \emph{Ergodic invariant probability measures
  and entire functions}, Acta Math. Hungar \textbf{73} (1996), no.~3, 213 --
  218.

\bibitem[DH85]{DH85}
A. Douady and J. H. Hubbard, \emph{On the dynamics of polynomial-like mappings}, Ann. Sci. \'Ec. Norm. Sup\'er. {\bf 18} (1985), no. 2, 287--343.

\bibitem[Din70]{Din70}
E. Dinaburg, \emph{A correlation between topological entropy and metric entropy} Dokl. Akad. Nauk SSSR {\bf 190} (1970),19--22.

\bibitem[EL92]{EL92}
A. {\`E}. Eremenko and M. Lyubich, \emph{Dynamical properties
  of some classes of entire functions}, Ann. Inst. Fourier (Grenoble)
  \textbf{42} (1992), no.~4, 989--1020.

\bibitem[FLMn83]{FLM}
A. Freire, A. Lopes, and R. Ma\~n\'e, \emph{An invariant measure
  for rational maps}, Bol. Soc. Brasil. Mat. \textbf{14} (1983), no.~1, 45--62.
 

\bibitem[Fos53]{Foster}
F.~G. Foster, \emph{{On the Stochastic Matrices Associated with Certain Queuing
  Processes}}, The Annals of Mathematical Statistics \textbf{24} (1953), no.~3,
  355 -- 360.

\bibitem[FS94]{FS94}
J. E. Forn{\ae}ss and N. Sibony, \emph{Complex dynamics in higher
  dimensions}, Complex potential theory ({M}ontreal, {PQ}, 1993), NATO Adv.
  Sci. Inst. Ser. C: Math. Phys. Sci., vol. 439, Kluwer Acad. Publ., Dordrecht,
  1994, Notes partially written by Estela A. Gavosto, pp.~131--186.

\bibitem[Gro03]{Gromov}
M. Gromov, \emph{On the entropy of holomorphic maps}, Enseign. Math. (2)
  \textbf{49} (2003), no.~3-4, 217--235.



\bibitem[KH95]{KatokHasselblatt}
A.~Katok and B.~Hasselblatt, \emph{Introduction to the modern theory of
  dynamical systems}, Encyclopedia of Mathematics and its Applications,
  Cambridge University Press, 1995.
  
  \bibitem[Kol58]{Kol58}
  A.N. Kolmogorov, \emph{New Metric Invariant of Transitive Dynamical Systems and Endomorphisms of Lebesgue Spaces}, Doklady of Russian Academy of Sciences {\bf 119} N5 (1958), 861--864.
  
    \bibitem[Kol59]{Kol59}
  A.N. Kolmogorov, \emph{Entropy per unit time as a metric invariant of automorphism}, Doklady of Russian Academy of Sciences {\bf 124} (1959), 754--755.

\bibitem[Lyu83]{Lyubich83}
M.  Lyubich, \emph{Entropy properties of rational endomorphisms of the
  {R}iemann sphere}, Ergodic Theory Dynam. Systems \textbf{3} (1983), no.~3,
  351--385. 

\bibitem[McM96]{McmullenRenormalization}
C. T. McMullen, \emph{Renormalization and 3-manifolds which fiber over the
  circle}, no. 142, Princeton University Press, 1996.

\bibitem[MP77]{MP1977}
M. Misiurewicz and F. Przytycki, \emph{Topological entropy and
  degree of smooth mappings}, Bull. Acad. Polon. Sci. S\'er. Sci. Math.
  Astronom. Phys. \textbf{25} (1977), no.~6, 573--574. 

\bibitem[MU10]{MaUrBook}
V. Mayer and M. Urba{\'n}ski, \emph{Thermodynamical formalism and
  multifractal analysis for meromorphic functions of finite order}, vol. 203,
  American Mathematical Society, 2010.

\bibitem[MU21]{MaUrSurvey}
\bysame, \emph{Thermodynamic formalism and geometric applications for
  transcendental meromorphic and entire functions}, Thermodynamic Formalism:
  CIRM Jean-Morlet Chair, Fall 2019, Springer, 2021, pp.~99--139.


\bibitem[Osb13]{Osborne}J. Osborne, \emph{Connectedness properties of the set where the iterates of
an entire function are bounded},
Mathematical Proceedings of the Cambridge Philosophical Society {\bf 155}  Issue 03, 2013, 391--410.

\bibitem[Pat76]{Pat76}
S. J. Patterson, \emph{The limit set of a fuchsian group}, Acta Math. (1976),
  no.~136, 241--273.

\bibitem[Pat87]{Pat87}
\bysame, \emph{Lectures on measures on limit sets of kleinian groups, in
  analytical and geometric aspects of hyperbolic space}, London Math. Soc.,
  Lecture Notes, vol. 111, Cambridge Univ. Press, 1987.

\bibitem[Prz90]{Prz1990}
F. Przytycki, \emph{On the {P}erron-{F}robenius-{R}uelle operator for
  rational maps on the {R}iemann sphere and for {H}\"older continuous
  functions}, Bol. Soc. Brasil. Mat. (N.S.) \textbf{20} (1990), no.~2, 95--125.


\bibitem[Rem16]{RempeArclike}
L. Rempe, \emph{Arc-like continua, Julia sets of entire functions, and
  {E}remenko's conjecture}, arXiv preprint arXiv:1610.06278 v4 (2016).

\bibitem[RRRS11]{RRRS}
G. Rottenfusser, J. Rückert, L. Rempe, and D. Schleicher,
  \emph{Dynamic rays of bounded-type entire functions}, Annals of Mathematics
  \textbf{173} (2011), no.~1, 77--125.

\bibitem[Sin59]{Sin59}
Y. G. Sinai, \emph{On the Notion of Entropy of a Dynamical System}, Doklady of Russian Academy of Sciences {\bf 124} (1959), 768--771.

\bibitem[Six18]{SixsmithB}
D. Sixsmith, \emph{Dynamics in the Eremenko-Lyubich class}, Conformal
  Geometry and Dynamics of the American Mathematical Society \textbf{22}
  (2018), no.~9, 185--224.

\bibitem[Smi90]{Smillie90}
J. Smillie, \emph{The entropy of polynomial diffeomorphisms of {${\bf
  C}^2$}}, Ergodic Theory Dynam. Systems \textbf{10} (1990), no.~4, 823--827.


\bibitem[Sul79]{Su79}
D. Sullivan, \emph{The density at infinity of a discrete group}, Inst.
  Hautes Etudes Sci. Pub. Math. \textbf{50} (1979), no.~1.

\bibitem[Sul82]{Su82}
\bysame, \emph{Disjoint spheres, approximation by imaginary quadratic numbers
  and the logarithmic law for geodesics}, Acta Math. \textbf{149} (1982),
  215--237.

\bibitem[Sul84]{Su84}
\bysame, \emph{Entropy, Hausdorff measures old and new, and the limit set of geometrically finite Kleinian groups}, Acta Math. \textbf{153} (1984),
  259--277.


\bibitem[Wen05a]{WendtMan}
M. Wendt, \emph{The entropy of entire transcendental functions},
  arXiv:2011.02163 [math.DS], 2005.

\bibitem[Wen05b]{Wendt}
M. Wendt, \emph{Zuf{\"a}llige {J}uliamengen und invariante {M}a{\ss}e mit
  maximaler {E}ntropie}, Ph.D. thesis, University of Kiel,
  https://macau.uni-kiel.de/receive/dissertation\_diss\_00001412 (German),
  2005.

\end{thebibliography}

\end{document}